\documentclass[oneside, 11pt]{amsart}

\usepackage{amsmath, amssymb, amsthm, hyperref, graphicx,   verbatim, enumerate, xcolor, tikzsymbols, stmaryrd, epsfig, xypic}
\usepackage{array, colortbl}
\usepackage{mathrsfs}  
\usepackage{mathabx}
\usepackage{times}
 
 \DeclareFontFamily{U}{mathc}{}
\DeclareFontShape{U}{mathc}{m}{it}%
{<->s*[1.03] mathc10}{}

\DeclareMathAlphabet{\mathscr}{U}{mathc}{m}{it}

\usepackage[paper=a4paper, marginpar=2.4cm]{geometry}
\usepackage[all]{xy}
\usepackage[utf8]{inputenc}

\usepackage[english]{babel}
 
\newcommand{\e}{\varepsilon}

 \newcommand{\disk}{\mathbb{D}}

\newcommand{\fr}{\partial}

\newcommand{\set}[1]{\left\{#1\right\}}
\newcommand{\norm}[1]{{\left\Vert#1\right\Vert}}
\newcommand{\abs}[1]{\left\vert#1\right\vert}

\newcommand{\rest}[1]{ \arrowvert_{#1}}

\newcommand{\unsur}[1]{\frac{1}{#1}}

\newcommand{\cst}{\mathrm{C}^\mathrm{st}}
\newcommand{\lrpar}[1]{\left(#1\right)}

\newcommand{\hot}{\mathrm{h.o.t.}}
\newcommand{\loc}{\mathrm{loc}}
\newcommand{\inv}{^{-1}}
\renewcommand{\d}{\mathrm{d}}
\newcommand{\jstar}{J^\varstar}

\DeclareMathOperator{\supp}{Supp}
\DeclareMathOperator{\Int}{Int}
 \DeclareMathOperator{\im}{Im}

\DeclareMathOperator{\id}{id}
\DeclareMathOperator{\jac}{Jac}
\DeclareMathOperator{\diam}{Diam}
\DeclareMathOperator{\dist}{dist}

\newcommand{\C}{\mathbf{C}}
\newcommand{\R}{\mathbf{R}}
\newcommand{\Q}{\mathbf{Q}}
\newcommand{\Z}{\mathbf{Z}}
\newcommand{\N}{\mathbf{N}}

\newcommand{\cs}{{\text{\sc{j}}}}

\renewcommand{\P}{\mathbb{P}}
\newcommand{\A}{\mathbb{A}}

\def\Jac{\mathrm{Jac}}

\def\diam{{\mathrm{diam}}}

\newcommand{\Aut}{\mathsf{Aut}}

\newcommand{\GL}{{\sf{GL}}}

\newcommand{\m}{\mathscr{m}}

\theoremstyle{plain}

\newtheorem{thm}{Theorem}[section]
\newtheorem{cor}[thm]{Corollary}
\newtheorem{pro}[thm]{Proposition}
\newtheorem{lem}[thm]{Lemma}

\newtheorem{mthm}{Theorem}

\newtheorem{que}[thm]{Question}

\theoremstyle{definition}

\newtheorem{rem}[thm]{Remark}

\numberwithin{equation}{section}       

\addtocounter{section}{0}             
\numberwithin{equation}{section}       

\title[Rigidity for  polynomial automorphisms of $\C^2$]{Some rigidity results for 
 polynomial automorphisms of $\C^2$}

\date{November 15, 2024}

\author{Serge Cantat}
\address{Serge Cantat, IRMAR, Campus de Beaulieu,
b\^atiments 22-23
263 avenue du G\'en\'eral Leclerc, CS 74205
35042  RENNES C\'edex, France}
\email{serge.cantat@univ-rennes1.fr}
 \author{Romain Dujardin}
\address{Romain Dujardin,  Sorbonne Universit\'e, CNRS, Laboratoire de Probabilit\'es, Statistique  et Mod\'elisation  (LPSM), F-75005 Paris, France}
\email{romain.dujardin@sorbonne-universite.fr}
\thanks{{\footnotesize{The research activities of the authors  are partially funded by the European Research Council (ERC GOAT 101053021). The authors benefited from the support of the French government "Investissements d'Avenir" program integrated to France 2030 (ANR-11-LABX-0020-01).}}
}
\begin{document}

\setlength{\parskip}{.2em}
\setlength{\baselineskip}{1.26em}   

\maketitle

\begin{abstract}
We prove several new rigidity results for automorphisms of $\C^2$ with positive entropy. 
A first result is that 
a complex  slice of the (forward or backward) Julia set is never a smooth, or even rectifiable, curve.
We also show that such an automorphism cannot preserve a global holomorphic foliation, 
nor  a real-analytic foliation with complex leaves.   

These results are used to show  that under mild assumptions, 
two real-analytically conjugate automorphisms  are polynomially conjugate. 

For mappings defined over a number field, we also study the fields of definition 
 of multipliers of saddle periodic orbits.  
\end{abstract}

\setcounter{tocdepth}{1}
\tableofcontents

\section{Introduction}

\subsection{Rigidity problems} 
In our previous paper~\cite{conjugate}, answering a question of Friedland and Milnor~\cite{friedland-milnor}, 
we established the following result: 
if two polynomial automorphisms $f$ and $g$ of $\C^2$ 
of positive entropy are conjugated by a biholomorphism of $\C^2$, then they are conjugate 
in the group $\Aut(\C^2)$ of polynomial automorphisms. In~\cite{friedland-milnor}, the authors study 
  more generally the following rigidity problem: what can be said when $f$ and $g$ are conjugated by a real   diffeomorphism? 
  They prove  that if $f$ and $g$ are 
complex Hénon maps of degree 2 which are conjugated by a real  $C^1$ diffeomorphism then 
$f$ is conjugated to $g$ or $\overline g$ in $\Aut(\C^2)$, where $\overline g$ is obtained from $g$ by 
applying   complex conjugation to the coefficients.   
 Our first goal will be to solve the real-analytic version of this problem under a generic hypothesis. 
 
\begin{mthm}\label{mthm:conjugate_real}
Let $f$ and $g$ be      polynomial automorphisms of  $\C^2$ of positive entropy, which are 
conjugated by a real-analytic 
diffeomorphism $\varphi:\C^2\to \C^2$. Assume that $f$ admits a saddle periodic point $p$, of some period $n\geq 1$,  at which the two eigenvalues of $df^n_p$ are both non-real.
Then  $f$ is conjugate to $g$ or $\overline g$ in $\Aut(\C^2)$. 
\end{mthm}

An obvious approach to this rigidity problem would be to show that the set of 
conjugacy classes of the differential $df^n$ at periodic orbits of period $n$, for all $n\geq 1$,
 characterizes  an automorphism modulo conjugacy in $\Aut(\C^2)$ (possibly up to complex conjugation if the conjugacy class is considered in the real sense). This is an instance of the classical 
 \emph{multiplier rigidity problem}, itself part of the celebrated \emph{spectral rigidity problem}. 
For instance, a Hénon map $f$ of  degree $2$ 
 is indeed characterized by the real conjugacy class of $Df$ at its fixed points, which allows Friedland and Milnor to proceed with this case in~\cite[Thm 7.5]{friedland-milnor}.  But, despite recent advances in one-dimensional dynamics~\cite{ji-xie:multipliers,ji-xie:injective}, this problem is essentially untouched in our two-dimensional context, so we take a different path. 
 
Our proof of Theorem~\ref{mthm:conjugate_real} 
relies on a number of rigidity properties  \emph{concerning  a single map}, 
which are interesting for their own sake. 
In particular, we deal with the classical  problem of 
\emph{smoothness of the stable lamination}
which has been 
extensively studied in hyperbolic  dynamics, notably in connection with the classification of 
Anosov diffeomorphisms  and flows: see for instance~\cite[Chapter 9]{fisher-hasselblatt:hyperbolic_flows} for an introduction, and~\cite{Xu-Zhang} for a recent contribution in a holomorphic context. 
Another key step of the proof is  \emph{the  non-existence
of automorphisms with smooth Julia sets}, a problem that  was previously addressed by Bedford and Kim~\cite{BK1, BK2}.   The underlying 
philosophy behind these results  
is that  there is no ``integrable" automorphism, which would play a role analogous to 
 monomial, Chebychev or Lattès  mappings in one-dimensional dynamics.  


\subsection{Invariant foliations} Let us be more specific. From now on, a polynomial automorphism of $\C^2$ of positive topological entropy will be called \emph{loxodromic}; such an automorphism $f$ is conjugate, in $\Aut(\C^2)$, to a composition of Hénon maps. We denote the  forward and backward Julia sets of $f$ by $J^+$ and $J^-$, respectively. The closure of the union of all saddle periodic orbits of $f$ is denoted by $\jstar$. It is a subset of the  Julia set $J := J^+\cap J^-$. 
The  stable (resp. unstable) manifold of any saddle periodic point is an immersed Riemann surface, biholomorphic to $\C$, which is 
dense in $J^+$ (resp. $J^-$) and endows  $J^+$ (resp.\ $J^-$) with some kind of laminar structure. When $f$ is hyperbolic (that is, when $J$ is hyperbolic  as an $f$-invariant set)  these are actual  
laminations by Riemann surfaces. Most of these facts are due to Bedford and Smillie, and we refer to the original papers~\cite{bs1, bs2, bls} for details. 
 
 We shall say that $J^+$ is subordinate to a foliation if there is a neighborhood $U$ of $J^+$ and a foliation $\mathcal F$ of $U$ by Riemann surfaces such that $J^+$ is saturated by $\mathcal F$: if $x\in J^+$, the leaf $\mathcal F(x)$ is contained in $J^+$. We allow foliations to have isolated singularities. 
 Another possible definition is 
 that every disk contained in a stable manifold is contained in a leaf of $\mathcal F$
(see~\S\ref{par:invariance} for more details). 
 
\begin{mthm}\label{mthm:no_real_global_foliation}
If $f$ is a loxodromic automorphism of $\C^2$, then $J^+$  (resp. $J^-$)
cannot be subordinate to a global real-analytic (in particular to a holomorphic) foliation.  
\end{mthm}
 

We could also formulate this result  as the non-existence of
an $f$-invariant  real-analytic foliation with complex leaves
(see Remark~\ref{rem:invariant_real_foliation}).  
Brunella~\cite{brunella}   proved Theorem~\ref{mthm:no_real_global_foliation} 
when the global foliation is defined by an algebraic $1$-form, and his theorem is actually 
  a key step in  the proof. The chain of arguments runs as follows: 
\begin{center}
no invariant algebraic foliation (Brunella~\cite{brunella}) $\leadsto$ no invariant holomorphic foliation (Theorem~\ref{thm:no_invariant_global_foliation}) $\leadsto$ no invariant real-analytic foliation (Theorem~\ref{mthm:no_real_global_foliation}).
\end{center}

Theorem~\ref{mthm:no_real_global_foliation}  provides 
 a partial answer to Question~31 in~\cite{henon-problem-list}, which asks whether $J^+$ can be  \emph{locally} subordinate to a  
holomorphic foliation. 
For instance,  for a dissipative and hyperbolic map, the stable lamination always admits a local extension to a $C^{1+\e}$ foliation with complex leaves
   (see \cite[Lem. 5.3]{lyubich-peters}).  In \S~\ref{subs:local_subordination},
   we show in particular  that for an automorphism with disconnected Julia set
  (e.g. a horseshoe), such a foliation is never holomorphic, nor even real analytic,
   in a neighborhood of $\jstar$ 
   (see~Theorem~\ref{thm:no_real_global_foliation}).

\subsection{Smooth Julia slices}\label{par:intro_julia_slices}  In one variable, it was shown by Fatou~\cite{fatou:3e_memoire}
that if the Julia set is contained in a curve then it is contained 
in a circle, and that in this case 
it is either a Cantor set, an arc of the circle, or the circle itself; and these last two cases correspond to integrable maps. Our study leads to a similar problem. Indeed, along the proof  of Theorem~\ref{mthm:no_real_global_foliation}
we   have to consider a local slice of $J^+$ by some holomorphic transversal 
(e.g. by some disk in the unstable manifold of a saddle point) and to study the possibility that such a slice is contained in a smooth curve.  Such a slice  is a kind of relative to a one-dimensional Julia set.

To state an analogue of Fatou's result in our situation, recall that  
if $p$ is a saddle periodic point of period $n$, then its unstable manifold $W^u(p)$ is  
biholomorphic to $\C$. Thus, we can fix a parametrization $\psi^u_p:\C\to W^u(p)$. If $\lambda^u$ denotes the unstable eigenvalue of 
$df^n_p$, then $f^n\circ \psi^u_p(\zeta) =\psi^u_p(\lambda^u\zeta)$.  
The unstable Lyapunov exponent  of $p$
 is by definition $\chi^u(p) = \unsur{n} \log \abs{\lambda^u(p)}$.

It is not difficult to show that if a local holomorphic 
slice of $J^+$ is contained   in a $C^1$ curve, then  for any saddle periodic point $p$,
$(\psi^u_p)\inv(J^+)$ is contained in   a line (\footnote{There is a  subtle point   about 
the exact definition of  a local slice of $J^+$, so we are abusing slightly here, 
see \S~\ref{subs:transversals} for more details.});
likewise if a local slice is a $C^1$ curve, then  $(\psi^u_p)\inv(J^+)$ \textit{is}     a line. 
We say that $f$ is \emph{unstably real} in the former case, and  \emph{unstably linear} in the latter. 
Thus, in the previous analogy,
 unstably linear automorphisms correspond to integrable one variable polynomials. 

\begin{mthm}\label{mthm:no_unstably_linear} 
Unstably linear loxodromic automorphisms do not exist among loxodromic automorphisms of $\C^2$. 
\end{mthm}

This is an essential tool towards Theorem~\ref{mthm:no_real_global_foliation}. 
A weaker result was obtained by Bedford and Kim in~\cite{BK1} (see also~\cite{BK2}), 
who showed that  $J^+$ itself cannot 
be a smooth 3-manifold. The contradiction   in~\cite{BK1} comes from a global  topological argument together with 
 the following   multiplier rigidity statement: for 
a generalized Hénon map of degree $d$ it is not possible that  
 $d-1$ of its $d$ fixed points are saddles with the same unstable eigenvalue (for technical reasons, in \cite{BK1} it   needs to be  assumed that $f$ is a composition of at least 3 Hénon maps). 
This result also plays a key role in our proof. 
 Another    important tool in the proof of Theorem~\ref{mthm:no_unstably_linear}
   is the theory of quasi-expansion developed by Bedford and Smillie in~\cite{bs8}. More precisely, the results of~\cite{bs8} show that an unstably linear map is quasi-expanding; adapting an argument   from~\cite{BSR},
   and using the hyperbolicity criteria of~\cite{topological}, 
 we obtain  that it is actually uniformly hyperbolic.  Then, a  
  generalization  of~\cite{BK1} leads to the desired contradiction.

%
%
%
%
%

Coming back to Theorem~\ref{mthm:conjugate_real}, the  bulk of the proof  is to show that 
 $\varphi$ must be holomorphic or anti-holomorphic;  then  
we invoke our previous result~\cite{conjugate} to conclude.
Assuming, by way of contradiction, that 
$\varphi$ is neither holomorphic nor anti-holomorphic, then,  pulling back the complex structure   by $\varphi$   gives an ``exotic'' $f$-invariant real-analytic 
complex structure on $\C^2$. 
Analyzing this complex structure at saddle points 
ultimately produces an $f$-invariant real-analytic foliation with complex leaves, 
which  contradicts Theorem~\ref{mthm:no_real_global_foliation}, and we are done.

\subsection{The multiplier field}  
The proof of Theorem~\ref{mthm:conjugate_real} requires also a lemma  
 on the   multipliers associated to periodic orbits shadowing a homoclinic orbit: see Theorem~\ref{thm:homoclinic_multipliers}. This statement is  
inspired by a one dimensional result due to Eremenko-Van Strien~\cite{eremenko-vanstrien} and Ji-Xie~\cite{ji-xie:multipliers}, which was used by Huguin~\cite{huguin} to characterize rational maps on $\P^1(\C)$ whose multipliers belong to a fixed number field. 
Adapting  Huguin's argument, we obtain the following (see Theorems~\ref{thm:multipliers_huguin1} and~\ref{thm:multipliers_huguin2}):

\begin{mthm}\label{mthm:multipliers_huguin}
Let $f\in \Aut(\C^2)$ be a loxodromic automorphism. 
Assume that
\begin{itemize}
\item either $f$ is hyperbolic and admits two saddle periodic points $p$ and $p'$ with distinct Lyapunov exponents $\chi^u(p)\neq \chi^u(p')$;
\item or $f$ admits a saddle periodic point whose unstable 
Lyapunov exponent is larger than that of the maximal entropy measure.
\end{itemize}
Then the unstable (resp. stable) multipliers of $f$ cannot lie in a fixed number field. 
\end{mthm}

 To prove Theorem~\ref{mthm:multipliers_huguin}, we shall first assume that $f$ is defined over a number field and then provide specialization arguments to reduce the general case to this particular setting. These arguments, together with the ones from~\cite{DF}, are of independent interest. 
 Note that the first assumption of the theorem may also be replaced by ``$f$ is hyperbolic with a non totally disconnected Julia set'' (see Proposition~\ref{pro:examples_huguin1}).

The assumptions of Theorem~\ref{mthm:multipliers_huguin} are easily checkable 
in a perturbative setting, and it follows that its conclusion holds for 
 Hénon maps with sufficiently small Jacobian (see Theorem~\ref{thm:multipliers_henon}). 
 A delicate issue is that the one-dimensional result is \emph{not} true for Chebychev and monomial mappings, and we have to understand how this
 obstruction vanishes when going to two dimensions.


\subsection{Plan of the paper}  In Section~\ref{sec:totally_real}, 
we study unstably real and linear maps and prove Theorem~\ref{mthm:no_unstably_linear}.
Several natural questions are   discussed in \S~\ref{subs:questions_unstably_real}.
In Section~\ref{sec:rectifiable},  
we push this study one step further and  show that a 
  holomorphic slice of $J^+$ cannot  be a rectifiable curve. 
 A notable consequence of this study is that the unstable Hausdorff dimension of a dissipative  hyperbolic map with connected Julia set is always larger than 1 (see Corollary~\ref{cor:dimH}).
  
  Sections~\ref{sec:foliations} and~\ref{sec:real_analytic_foliations} are devoted to   Theorem~\ref{mthm:no_real_global_foliation}, as well as 
  some local versions of it.  
 In Section~\ref{sec:homoclinic} we prove Theorem~\ref{thm:homoclinic_multipliers} on the multipliers of   periodic orbits 
shadowing a homoclinic orbit, and Theorem~\ref{mthm:conjugate_real} is finally obtained  in Section~\ref{sec:conjugate_real}. 

Lastly, in
Section~\ref{sec:rational_multipliers} we  investigate the arithmetic properties of saddle point multipliers, and establish Theorem~\ref{mthm:multipliers_huguin}, as well as its application to perturbations of one-dimensional maps 
(Theorem~\ref{thm:multipliers_henon}). In Appendix~\ref{app:equidist}, 
we establish a stronger version of 
  the equidistribution theorem for saddle periodic points of~\cite{bls2}, which is necessary for Theorem~\ref{mthm:multipliers_huguin}.

%
%
 
\subsection{Notation} We use the standard notation and vocabulary of the field, as listed for instance in~\cite[\S 1]{bs6}  or \cite[\S 2.1]{tangencies}. For instance, given a loxodromic automorphism $f$ of
$\C^2$, we denote by  $G^+$ its dynamical Green function (or rate of escape function), and by $T^+=dd^cG^+$ the associated current (in~\cite{bs6}, the invariant currents are denoted by $\mu^\pm$ instead of $T^\pm$).  
The Jacobian determinant of an  automorphism $f$ of $\C^2$ is  constant:  $f$ is dissipative if $\vert\Jac(f)\vert < 1$, volume expanding if $\vert\Jac(f)\vert >1$, and conservative if  $\vert\Jac(f)\vert =1$.

By convention the Zariski  (resp. $\R$-Zariski) topology   
is the complex analytic (resp. real-analytic) Zariski topology. We sometimes refer to these 
simply as the (real) analytic topology.

\section{Unstably real  and unstably linear maps}\label{sec:totally_real}

\subsection{Transversals}\label{subs:transversals} 
Let $f\in\Aut(\C^2)$ be a loxodromic automorphism.  
 We say that a holomorphic disk $\Gamma$
 is a \emph{transversal to} $J^+$ if it satisfies one of the following equivalent conditions:
 \begin{enumerate}[(a)]
 \item $\Gamma\cap J^+ \neq\emptyset$ and $\Gamma \not\subset K^+$;
\item  $G^+\rest\Gamma$ is not harmonic;   
\item  $G^+\rest\Gamma$ vanishes and is not identically $0$. 
\end{enumerate}
It follows from the work of 
Bedford, Lyubich and Smillie~\cite[\S\S 8-9]{bls} that for any saddle periodic point $p$, $W^s(p)$ admits transverse intersections with $\Gamma$ 
(see~\cite[Lem. 5.1]{tangencies} and~\cite[Lem. 3.3]{DF} for more details). We set 
\begin{equation}
J^+_\Gamma = \supp(T^+\wedge [\Gamma])   = \supp(dd^c(G^+\rest\Gamma)) = 
\fr_\Gamma(K^+\cap \Gamma),
\end{equation}
where in the last equality $\fr_\Gamma$ refers to the boundary as a subset of $\Gamma$. 
Note that $J^+_\Gamma \subset   \Gamma\cap J^+$ but  this inclusion could be strict. Indeed, $\Gamma$ could a priori contain a disk 
that is entirely contained in $J^+$: such a disk would be part of $\Gamma\cap J^+$ but not of $J^+_\Gamma$.
 If in addition $\Gamma = \Delta^u_p$ is contained in  the unstable manifold of a saddle point $p$, we have (with the same caveat  for reverse inclusions)
\begin{equation}
J^+_{\Delta^u_p} \subset  \Delta^u_p\cap \jstar \subset\Delta^u_p\cap J^+.
\end{equation}
The first inclusion follows from the fact that $J^+_{\Delta^u_p}$ is the closure of the homoclinic intersections contained  in $\Delta^u_p$ (cf. \cite[Lem. 5.1]{tangencies}). Recall that 
a homoclinic intersection is a point of $W^s(p)\cap W^u(p)\setminus\set{p}$   and  that every homoclinic intersection is contained in $\jstar$ (see~\cite[Thm 9.9]{bls}).  

As for Julia sets in one complex variable, 
the sets $J^+_\Gamma$ vary  lower semicontinuously in the Hausdorff topology, in the following sense:

\begin{lem} Let $\Gamma$ be a transversal to $J^+$. 
If $(\Gamma_n)$ is a sequence of disks converging in the $C^1$ topology    
to $\Gamma$, then
\begin{enumerate}
\item  $J^+_\Gamma\subset \liminf_{n\to\infty} J^+_{\Gamma_n}$; 
\item $J^+\cap \Gamma \supset\limsup_{n\to\infty} J^+\cap {\Gamma_n}$.
\end{enumerate} \end{lem}

\begin{proof}Property~(1) follows from the continuity of $G^+$: for any $z\in J^+_\Gamma$, 
pick a small disk $U\subset \Gamma$ around $z$, then by definition $G^+\rest{U}$ is not harmonic; so 
if we lift $U$ to some disk $U_n\subset \Gamma_n$ then for large $n$, $G^+\rest{U_n}$ 
cannot be harmonic and we are done. Note that we may extend this argument to the case where 
$(\Gamma_n)$  converges to $\Gamma$ with some finite multiplicity (i.e. in the sense of analytic sets). 
Property~(2)  follows directly from $J^+$ being closed.   
\end{proof}

\subsection{Unstably real automorphisms}
 Let $p$ be a saddle periodic 
point of $f$, of exact period $n$. As in Section~\ref{par:intro_julia_slices}, we denote by $\psi_p^u:\C\to W^u(p)$ a parametrization  of its unstable manifold 
by an injective entire curve that maps $0$ to $p$. If a unit unstable vector $e^u_p$ is given, one may normalize $\psi_p^u:\C\to W^u(p)$ by fixing some $\alpha\in \C^\times$ and imposing $(\psi_p^u)'(0) = \alpha e^u_p$. The parametrization $\psi_p^u$ semi-conjugates $f^n$ to a linear map, that is, $f^n\circ \psi^u_p(\zeta) = \psi^u_p(\lambda^u\zeta)$, where $\lambda^u\in \C^\times$ is the unstable multiplier.

\begin{pro}\label{pro:contained_curve}
Let $f\in\Aut(\C^2)$ be a loxodromic automorphism. Assume that for some 
transversal $\Gamma$ to $J^+$, 
$ J^+_\Gamma$ is   contained in a $C^1$ smooth curve. 
Then for every saddle periodic point $p$, 
$(\psi^u_p)\inv(J^+_{W^u(p)})$   is contained in 
a line through the origin. In particular $\lambda^u(p)$ is real and $J^+_{W^u(p)}$ 
is  contained in a real-analytic curve. 
In addition 
\begin{itemize}
\item either  $J^+_\Gamma$ is not a Cantor set, and 
$(\psi^u_p)\inv(J^+_{W^u(p)})$  is a line through the origin for every saddle periodic point $p$; 
\item  or $J^+_{W^u(p)}$ is a Cantor set and
 $J^+_{W^u(p)} = J^+\cap W^u(p)$ for every saddle periodic point $p$. 
\end{itemize}
\end{pro}

If the assumption of the proposition holds, 
we say that $f$ is \emph{unstably real}.  Thus, for such a  map, 
all unstable multipliers are real.
By symmetry we have a similar result for transversals to $J^-$, yielding the notion of \emph{stably real} automorphism.


\begin{proof} The argument  goes back to Fatou~\cite[\S 46]{fatou:3e_memoire}. 
Pick a saddle periodic point $p$, and replace $f$ by a positive iterate to assume that $p$ is fixed. As already explained, $W^s(p)$ admits transversal intersections with $\Gamma$, and more precisely with $J^+_\Gamma$. As in~\cite[Lem. 1.12]{DF}, we can find holomorphic 
coordinates $(x,y)\in \disk^2$ near $p$ in which 
$p = (0,0)$, $W^{s}_\loc(p) = \set{x=0}$, $W^u_\loc(p) = \set{y=0}$ and 
 \begin{equation}\label{eq:saddle_normal_form}
 f(x,y) =\left( \lambda^u x(1+xy g_1(x,y)), \lambda^s y(1+xy g_2(x,y)) \right)~,
 \end{equation}
 with $\abs{\lambda^s}<1<\abs{\lambda^u}$, and $\norm{g_1}$, $\norm{g_2}$ as small as we wish. 
 In these coordinates, $f\rest{\set{y=0}}$ is linear, so $\zeta\mapsto (\zeta, 0)$ is an unstable 
 parametrization near the origin. 
 Changing $\Gamma$ in $f^m(\Gamma)$ for a large positive integer $m$, and taking the connected component of $f^m(\Gamma)\cap \disk^2$ containing $p$, we may assume that $\Gamma$ is a graph over the first coordinate, 
 of the form $y = \gamma(x)$. Let $\sigma\mapsto t(\sigma)$ be a germ of $C^1$ curve at $0\in \C$, 
 with $t'(0)\neq 0$, such that $J^+_\Gamma$ is contained in the image of 
 $\sigma\mapsto (t(\sigma), \gamma(t(\sigma)))$. With our choice of coordinates, 
 Lemma 4.2 in~\cite{DF} asserts that for some $\delta>0$, if   $\abs{x}\leq \delta$,  then  
$f^n\lrpar{ {x}/{(\lambda^u)^n}, y}\to (x, 0) $ as $n\to\infty$. Actually the proof says a little more:   
if $\abs{x_n} \leq \delta$ and $\abs{y_n}<1$, then 
\begin{equation}\label{eq:renorm}
f^n\lrpar{\frac{x_n}{(\lambda^u)^n}, y_n} = (x_n, 0)+ o(1).
\end{equation}
Fix a subsequence $n_j$ such that $\frac{(\lambda^u)^{n_j}}{\abs{\lambda^u}^{n_j}} \to 1$. 
Since $t(\abs{\lambda^u}^{-n} \sigma) = \abs{\lambda^u}^{-n} (t'(0)\sigma +o(1))$, Equation~\eqref{eq:renorm} for sufficiently 
small $\sigma\in \R$ implies
\begin{equation}
f^{n_j}\lrpar{ t \lrpar{\abs{\lambda^u}^{-{n_j}} \sigma}, \gamma\lrpar{t\lrpar{\abs{\lambda^u}^{-{n_j}} \sigma}}}
\underset{j\to\infty}\longrightarrow \lrpar{ t'(0)\sigma, 0 }.
\end{equation}
Now we use the lower semi-continuity of $J^+_\Gamma$: if $z$ belongs to $J^+_{\set{y=0}}$, then it must be accumulated by $J^+_{f^{n_j}(\Gamma)}$, hence  $z\in t'(0)\R$, and we conclude 
that $J^+_{\set{y=0}}$ is contained in a line, as asserted.
 We also get 
that if $(n'_j)$ is any other subsequence such that 
$ {(\lambda^u)^{n'_j}}/{\abs{\lambda^u}^{n'_j}} $ converges to $ e^{i\theta}$, then $z\in e^{i\theta}t'(0)\R$. Since 
$J^+_{\set{y=0}}$ is not reduced to $\set{0}$, this argument shows that $e^{i\theta} = \pm 1$ so 
$ {(\lambda^u)^{n}}/{\abs{\lambda^u}^{n}}$ converges to $\set{\pm 1}$, and 
therefore $\lambda^u$ is real.  

To get the last assertion of the proposition, we observe that since $G^+\rest{\Gamma}$ is continuous, 
$J^+_\Gamma = \supp(dd^c (G^+\rest\Gamma))$ has no isolated points, so if it is contained in a smooth curve, it is either a Cantor set or it contains a non-trivial arc. 
In the latter case,  by reducing $\Gamma$, we may assume 
that  $J^+_\Gamma$ is a $C^1$ curve, and in this case 
the argument shows that $J^+_{\set{y=0}}$ is  a line through the origin. 
And if $J^+_{W^u(p)}$ is a Cantor set, then
 $J^+_{W^u(p)} = J^+\cap W^u(p)$, since in this case the unique connected component of 
 $W^u(p)\setminus  J^+_{W^u(p)}$ must be contained in $\C^2\setminus K^+$. 
\end{proof}

\begin{rem}\label{rem:pesin}
The proof shows more generally that if  for some saddle periodic point $q$, $J^+_\Gamma$ admits a tangent at 
some transverse intersection $W^s(q)\cap \Gamma$ (see the first lines of \S~\ref{subs:extension_curve} for a formal definition of 
having a tangent at some point),
hence in particular if $J^+_{W^u(q)}$ admits a tangent at $q$,
then the conclusion of the proposition holds at every saddle periodic point $p$.
The same is true for general hyperbolic measures and 
Pesin stable manifolds,  with a slightly different argument (see Lemma~\ref{lem:pesin_tangent}). 
\end{rem}

\subsection{Unstably linear automorphisms: multipliers}\label{subs:unstably_linear}

In this subsection we revisit and extend   some results of Bedford and Kim~\cite{BK1, BK2}. These will be used in the next subsection to prove the non-existence of unstably linear automorphisms.

Let $p$ be a saddle periodic point of period $n$ such that $(\psi^u_p)\inv\big(J^+_{W^u(p)}\big)$ is a line through the origin. Then, by Proposition~\ref{pro:contained_curve}, the same property holds at any other saddle point; in this situation, we say that $f$ is \emph{unstably linear}. 
Since $W^u(p)$ is not contained in $K^+$, at least one side of $J^+_{W^u(p)}$ in 
$W^u(p)\simeq\C$  must be contained in
 $\C^2\setminus K^+$. In particular $f$ is unstably connected in the sense of~\cite{bs6}. 
 Corollary~7.4 in~\cite[Cor. 7.4]{bs6}) 
 shows that a volume expanding map cannot be unstably connected,
 hence necessarily $\abs{\jac(f)}\leq 1$. 
Lemma 4.2 in~\cite{BK1} shows that the unstable multiplier equals $\pm d^n$.
For convenience we recall 
the argument: $G^+\circ \psi^u_p$ is harmonic and positive in some half plane (of slope, say, $\tan(\theta)$) and zero on its boundary, hence in this half plane   $G^+\circ \psi^u_p(\zeta)$ is proportional to $\im(e^{-i\theta} \zeta)$; in particular it is $\R$-linear. Then 
the invariance relation $G^+\circ \psi^u_p(\lambda^u \zeta)   = 
d^n G^+\circ \psi^u_p(\zeta)$  forces 
$\lambda^u  = \pm d^n$. Note that $\lambda^u  = d^n$ (resp. $\lambda^u  = -d^n$) iff $f^n$ preserves (resp. exchanges) the two half planes. 
 Let us summarize this discussion in a lemma.
 
\begin{lem}\label{lem:bilan_unstably_linear}
Assume that $f$ is unstably linear. Then for any saddle periodic point $p$, of period $n$,  
$(\psi^u_p)\inv(J^+_{W^u(p)})$ is a line through the origin, which cuts $\C$ in two half-planes. 
Moreover:
\begin{enumerate}[\rm(1)]
\item each component of $(\psi^u_p)\inv(\C^2\setminus K^+)$ is a half plane in which 
  $G^+\circ \psi^u_p(\zeta)$ is proportional to $\im(e^{-i\theta} \zeta)$,  
    where $\tan(\theta)$ is the slope of the line $(\psi^u_p)\inv(J^+_{W^u(p)})$;
\item the unstable multiplier at $p$ is equal to $\pm d^n$, where $d$ is the degree of $f$;
it is equal to $d^n$ iff $f^n$ preserves the two half planes;
\item $f$ is unstably connected and  $\abs{\jac(f)}\leq 1$;
\end{enumerate}
 \end{lem}

 
Our primary focus in the next proposition is on unstably linear maps; nevertheless, 
it may be useful to remark that it holds in the unstably real case as well so we state it in this generality.  
 
\begin{pro}\label{pro:unstably_linear}
Let $f$ be a loxodromic automorphism which is unstably real and dissipative. 
Then  every periodic point $p$ of $f$ is a saddle.
In addition, in the unstably linear case, 
 the two sides of $W^u(p)\setminus J_{W^u(p)}$ are contained in $\C^2\setminus K^+$. 
\end{pro}

\begin{proof}

The assumptions and conclusions of the proposition are not affected if we replace $f$ by some iterate.  
For technical reasons, we replace $f$ by $f^3$  (while still denoting its degree by $d$).

In a first stage, assume that $f$ is unstably linear. 

\smallskip

{\bf{Step 1.--}} Assume that for some saddle point $p$, the two sides of $W^u(p)\setminus J_{W^u(p)}$ 
are contained in $\C^2\setminus K^+$. We claim that the same property holds for every other 
saddle periodic point $q$. Indeed, $W^s(p)$ intersects transversally $W^u(q)$ at some point 
$\tau\in J_{W^u(q)}$. By the inclination lemma,  
there is a sequence of neighborhoods $U_n$ of $\tau$ in $W^u(q)$ such that 
$f^n(U_n)$ is a sequence of disks converging in the $C^1$ sense to $W^u_\loc(p)$.  We may assume that $U_n$ is a topological disk such that $U_n\setminus J_{W^u(q)}$ has two components. Since 
$\C^2\setminus K^+$ is open, for large $n$, the two sides of 
$f^n\big(U_n\setminus J_{W^u(q)}\big)$ intersect $\C^2\setminus K^+$, therefore they are contained in 
$\C^2\setminus K^+$; pulling back by $f^n$, we conclude that the same holds for $W^u(q)\setminus J_{W^u(q)}$. 

This discussion also shows that if for some saddle point $p$, one side of $W^u(p)\setminus J_{W^u(p)}$ is contained in $K^+$, then the same holds for every other saddle point.

\smallskip

{\bf{Step 2.--}} We claim that $f$ admits at most one non-saddle fixed point.

The argument is similar to that of~\cite[\S 2.2]{closing}. 
Consider a non-saddle periodic point $q$. 
Without loss of generality, assume that $q$ is fixed. Since $f$ is dissipative, 
$q$ is  a sink or $q$ is semi-neutral. In the latter case, the neutral eigenvalue is either a root of unity and  $q$ is said to be semi-parabolic, or it is not and $q$ is either semi-Siegel or semi-Cremer according to the existence of an invariant holomorphic disk containing $q$. 
In all these cases, 
there exists an invariant manifold through $q$ which is biholomorphic to $\C$ 
and contracted by the dynamics; for instance, when $q$ has eigenvalues of distinct moduli we can take the so-called  strong stable manifold $W^{ss}(q)$, associated to the most contracting eigenvalue.  The theory of Ahlfors currents associated to entire curves shows that 
this invariant manifold 
must intersect $W^u(p)$ (see \cite[Lem. 2.2]{closing} for this very statement and  
 \cite[Lem. 5.4]{tangencies} for the details of the proof).
If $q$ is a sink or if it is semi-Siegel, this stable manifold is contained in $\Int(K^+)$, so it must intersect 
$W^u(p)$ in
 a component of $\Int(K^+)\cap W^u(p)$. It then follows that 
 $G^+\equiv 0$   on the corresponding component of  
 $W^u(p)\setminus J^+_{W^u(p)}$,  which is then a Fatou disk entirely contained in the Fatou component of $q$. 
Since this is true for every such $q$, it follows that there can be at most one  
 sink or semi-neutral point.

If $q$ is semi-parabolic, the argument is the same 
except that instead of $W^{ss}(q)$, we consider any entire curve contained in the semi-parabolic basin of $q$ (such curves exist since the basin is known to be biholomorphic to $\C^2$). 

The last possibility is that $q$ is semi-Cremer. Then by~\cite{FLRT, LRT},  to 
any local center manifold $W^c_\loc(q)$ 
of $q$ we can associate a hedgehog $\mathcal H$:
\begin{itemize} 
\item  $\mathcal H$ is not locally connected since it is homeomorphic to the hedgehog of a non-linearizable one-dimensional germ. 
\item By~\cite[Thm E]{LRT}, $\mathcal H\subset \jstar$, and 
every point in $\mathcal H$ admits a strong stable manifold which is biholomorphic to $\C$, and 
has uniform geometry near $\mathcal H$ (because $f\rest{\mathcal H}$ admits 
 a dominated splitting, see~\cite[Thm A]{FLRT}). 
 \end{itemize}
Therefore, as above,  this strong stable manifold admits transverse intersections with $W^u(p)$ and 
we can transport by holonomy a non-trivial relatively open subset of $\mathcal H$ to 
$W^u(p)$. 
The resulting piece is contained in $J^+_{W^u(p)}$,
as follows from  the proof of~\cite[Thm E]{LRT}(\footnote{For the reader's convenience we explain the argument of~\cite[Thm E]{LRT}: if for some $z\in \mathcal H$, some transverse intersection 
point of $W^{ss}(z)\cap W^u(p)$ were not contained  in $J^+_{W^u(p)}$, then 
one side of $W^u\setminus J^+_{W^u(p)}$ 
would be a Fatou disk $\Omega$.
 It can be shown that if $(n_j)$ is a subsequence such that 
$f^{n_j}(\Omega)$ converges to a holomorphic disk $\Gamma$ (possibly reduced to a point), then $\Gamma$ should be contained in every local center manifold of $q$, hence in $\mathcal H$ (by definition of the hedgehog).  So by~\cite[Thm C]{LRT} the whole of $\Omega$ would be  contained in $W^{ss}(\mathcal H)$, which contradicts the fact that $\mathcal H$ has relative zero interior in any center manifold of $q$.
 }).
This is a 
contradiction because $J^+_{W^u(p)}$ is locally a smooth curve  while $\mathcal H$ is not locally connected, which shows that Semi-Cremer points cannot exist in this context.

\smallskip

{\bf{Step 3.--}}We can now complete the proof of the proposition in the unstably linear case. 
Since  $f$ is conjugated to a composition of  Hénon maps  of 
total degree $d$, it admits $d$ fixed points. 
By the second step, at most one of them is not a saddle, 
and if such a non saddle point exists, for every saddle point $p$, one of the components of 
$W^u(p)\setminus J_{W^u(p)}$ is contained in $K^+$. Therefore by 
Lemma~\ref{lem:bilan_unstably_linear}.(3), $\lambda^u(p)=d$. 
 In particular 
$f$ has $d$ distinct fixed points and at least $d-1$ of them are saddles   
with unstable multiplier equal to $d$. Since $f$ is conjugated to a product of at least 3 Hénon maps and $f$ is dissipative,  Proposition~5.1 in~\cite{BK1} shows that this is impossible. 
Thus,  all periodic points are saddles.

Then, by Step 1, if there is such a saddle point $p$ 
for which a component of $W^u(p)\setminus J_{W^u(p)}$ is contained in $K^+$, then the same is true for each of the fixed points of $f$, so that the unstable multiplier at each of the fixed points is equal to $d$. Again, \cite[Prop. 6.1]{BK1}  provides a contradiction. Thus, 
the two components of $W^u(p)\setminus J_{W^u(p)}$ are contained in $\C^2\setminus K^+$ for every periodic point, and the proof for the unstably linear case is complete.

\smallskip

{\bf{Step 4.--}} If $f$ is unstably real but not unstably linear, then by Proposition \ref{pro:contained_curve},
for every saddle point $p$, $J^+_{W^u(p)}$ is a Cantor set contained in a line. Thus 
$W^u(p)\setminus J^+_{W^u(p)}$ is connected, hence contained in $\C^2\setminus K^+$. Applying the arguments of Step 2 above shows that $f$ cannot possess any sink, semi-Siegel or semi-parabolic point for it would give rise to a Fatou disk contained in $W^u(p)\cap K^+$; in the case of a 
semi-Cremer point we would obtain a non-trivial (connected) continuum contained in $J^+_{W^u(p)}$, which again is a contradiction. 
\end{proof}

%
%
%
 
\begin{cor}\label{cor:hyperbolic_unstably_linear}
 A hyperbolic  loxodromic automorphism is never unstably linear.
\end{cor}

 
\begin{proof}
Assume that $f$ is unstably linear and hyperbolic. 
By Lemma~\ref{lem:bilan_unstably_linear}.(4), $f$ is unstably connected, so $J$ is connected, 
and also $\abs{\jac(f)}\leq 1$. 
If $f$ is conservative and hyperbolic, then $J$ cannot be connected (see~\cite[Cor. A.3]{bs7} 
or~\cite[Cor. 3.2]{connex}), so $f$ is dissipative. 
By Proposition~\ref{pro:unstably_linear} all periodic 
points are saddles. On the other hand, Theorem 3.1 in~\cite{connex}, asserts that 
 $f$ has  an attracting point. This  contradicts Proposition~\ref{pro:unstably_linear}. 
\end{proof}

\subsection{Unstably linear automorphisms: non-existence}  We are now ready to prove Theorem~\ref{mthm:no_unstably_linear}.

\begin{thm}\label{thm:no_unstably_linear}
A loxodromic automorphism of $\C^2$ is never unstably linear.
\end{thm}

\begin{cor}
If $\Gamma$ is a transversal to $J^+$, then $J^+_\Gamma$ is never a $C^1$ curve. 
More generally,  
 a non trivial component of $J^+_\Gamma$ does
not admit a tangent line at any transverse intersection of $\Gamma$ with $W^s(p)$, 
for any saddle point $p$.  
\end{cor}

The remainder of this subsection is devoted to the proof of the theorem. Henceforth we assume that 
$f$ is an unstably linear loxodromic automorphism. Throughout the proof,  the unstable parametrizations 
$\psi^u_p: \C\to W^u(p)$ are normalized  by (see~\cite{bs8}): 
\begin{equation}\label{eq:G+u}
 \psi^u_p(0)= p, \quad \max_{\abs{\zeta}\leq 1}  G^+\circ 
\psi^u_p (\zeta)=1, \quad \text{ and }  \quad (\psi^u_p)\inv(J^+_{W^u(p)}) = \R.
\end{equation}

\begin{lem}\label{lem:G+u}
Let  $p$ be a periodic point of  $f$, and $\psi^u_p$ be  normalized as above. 
Then up to replacing $\psi^u_p(\zeta)$ by $\psi^u_p(-\zeta)$, we have 
$$G^+\circ \psi^u_p (\zeta) = 
\begin{cases}  {\im(\zeta)} \text{ for }\im(\zeta) \geq 0 \\ 
c \abs {\im(\zeta)} \text{ for }\im(\zeta) < 0 \end{cases} 
$$ for some $0\leq c\leq 1$.
In particular 
$\set{G^+\circ \psi^u_p \leq 1}$ is the strip $\set{  - 1/c \leq \im(\zeta)\leq 1}$ (it is a half plane when $c=0$). 
In addition, when $f$ is dissipative, $c$ is non-zero.
\end{lem}

\begin{proof}
We already know from Lemma~\ref{lem:bilan_unstably_linear} that 
$G^+\circ \psi^u_p (\zeta)$ is proportional to $\abs{\im(\zeta)}$ on both sides of the real axis, say 
$G^+\circ \psi^u_p (\zeta) = c_+ \abs{\im(\zeta)}$ on the upper half plane and 
$G^+\circ \psi^u_p (\zeta) = c_- \abs{ \im(\zeta)}$  on  the lower half plane. After replacing $\zeta$ by $-\zeta$, we may assume $c_+\geq c_-\geq 0$, and then $c_+>0$ because $G^+$ can not vanish identically on $W^u(p)$. Since $\max_{\overline {\mathbb D}}  G^+\circ  \psi^u_p (\zeta)=1$, we 
conclude  that $c_+=1$. 
The last assertion follows from
Proposition~\ref{pro:unstably_linear}. 
\end{proof}

\begin{rem}
If $p$ is of period $n$ and $\lambda^u = -d^n$, then the relation 
$G^+ \circ \psi^u (\lambda^u \zeta) = d^n G^+ \circ \psi^u(\zeta)$ forces $c=1$. 
\end{rem}

The following result is essentially (but not exactly) contained in~\cite[Thm 4.8]{bs8}. Before stating it, let us introduce the concept of quasi-expansion. Let $\mathcal S$ be the set of saddle periodic points of $f$. This is a dense, $f$-invariant subset of $\jstar$. Let $\Psi_{\mathcal S}$ be the set of parametrizations $\psi ^u_p$, $p\in \mathcal S$, normalized as in Equation~\eqref{eq:G+u}.  
For $p$ in $\mathcal S$, there is a linear map of the form $L\colon \zeta\mapsto \lambda \zeta$ such that    $ f \circ \psi^u_p = \psi_{f(p)}^u \circ L$. One of the equivalent definitions of 
 {\emph{quasi-expansion}} is to require the existence of a constant $\kappa>1$ such that $\vert \lambda \vert \geq \kappa$ uniformly for all $p\in \mathcal S$, and then $\kappa$ is  a {\emph{quasi-expansion factor}}.
We refer to  \cite{bs8}  for this notion, in particular to Theorem~1.2 there.

\begin{lem}\label{lem:expansion_factor}
An unstably linear loxodromic automorphism is quasi-expanding, with quasi-expansion factor $\kappa=d$.
\end{lem}

\begin{proof}
Fix $p\in \mathcal S$ and consider the map $L\colon \zeta\mapsto \lambda \zeta$ introduced above. The relation $G^+\circ f = d G^+$ implies that 
$L$ maps the closed unit disk to a disk $\overline D(0, r)$ such that 
$\max_{\overline D(0, r)} G^+ \circ \psi_{f(p)}^u = d$. By Lemma~\ref{lem:G+u}, $r = d$, so 
$\vert \lambda \vert= d$ is uniformly bounded from below.  \end{proof}

In the next two paragraphs we assume  that $f$ is  quasi-expanding and state some general facts. 
Quasi-expansion implies that   $\Psi_{\mathcal S}$ is a normal family  (and vice versa). 
We denote by $\widehat \Psi$ the set of all its normal limits (i.e.\ $\widehat \Psi$ is the closure of $\Psi_{\mathcal S}$ with respect to  local uniform convergence). 
For any $x\in \jstar$, and any sequence $(p_n)\in \mathcal S^\N$ converging to 
$x$, one can extract a subsequence such that $(\psi^u_{p_n})$ converges towards an element $\widehat \psi \in  \widehat \Psi$ with $\widehat \psi(0) = x$. 
It is   a  non-constant entire curve  because 
$\max_{\overline{\disk}} G^+\circ \widehat \psi =1$, but it is not necessarily injective.  For 
$x\in \jstar$, we let 
\begin{equation}
\tau(x) = \max_{\widehat \psi\in  \widehat \Psi,  \ \widehat\psi(0 )= x}   \mathrm{ord}_0(\widehat \psi),
\end{equation}
where $\mathrm{ord}_0(\widehat \psi) = \min \set{k\geq 0\, ;\,  \psi^{(k)}(0) \neq 0}<\infty$ 
is the vanishing order of $\widehat \psi$ at 0. By~\cite[Lem 3.1]{guerini-peters}, $\tau$ is uniformly bounded (in our setting, we give a direct proof of this fact in Lemma~\ref{lem:tau_bounded} below). 

For every $x\in \jstar$ we set $\mathcal W^u(x)=\hat\psi(\C)$.  It does not depend on the choice of normal limit $\hat\psi\in \widehat\Psi$ with $\hat\psi(0)=x$ in the above construction; it is contained in $K^-$, and, in fact, it is the image of an injective entire curve (see~\cite[\S 1]{BGS} \footnote{This is essentially contained in~\cite{bs8} but a missing ingredient there was that $\mathcal W^u(x)$ is smooth at $x$.}). If $x$ is a saddle point, then $\mathcal W^u(x)=  W^u(x)$.
 As usual we denote by  $\psi^u_x:\C\to \mathcal W^u(x)$   an injective parametrization of $\mathcal W^u(x)$
such that $\psi^u_x(0)  = x$. If we consider a sequence 
$p_n\to x$ as above, then the the limit $\widehat \psi$ satisfies $\widehat \psi =\psi^u_x\circ h$ for some polynomial function $h$ (see~\cite[Lem. 6.5]{bs8},  and~\cite[\S 3]{guerini-peters}) such that $h(\zeta) = c\zeta^k+ \hot$, $k= \mathrm{ord}_0  \widehat \psi$. In particular, $G^+$ does not vanish identically on $ \mathcal W^u(x)$.

We now resume the proof of Theorem~\ref{thm:no_unstably_linear}.

\begin{lem}\label{lem:vanish}
If $f$ is quasi-expanding, then  for every $x\in \jstar$, 
 $G^+\circ\psi^u_x$ does not vanish identically in a neighborhood of the origin. 
\end{lem}
 
 \begin{proof}  
 If $p$ is a saddle periodic point, the maximum of $G^+\circ \psi^u_p$ on $\disk_1$ is equal to $1$. 
 Fix a radius $r\in ]0,1[$. Then, by Property (3) in Theorem 1.2 of \cite{bs8}, there is a constant $\alpha(r) >0$ 
 that does not depend on $p$ such that the maximum of $G^+\circ \psi^u_p$ on $\disk_r$ is  at least
 $ \alpha(r)$. 
 Since the family $(\psi^u_p)_{p\in \mathcal S}$ is normal and $G^+$ is continuous, this property is satisfied by all $\hat\psi_x$ in $\hat\Psi$, and the lemma follows. 
 \end{proof}
 
\begin{proof}[Second proof, specific to unstably linear maps]
Pick a sequence  $(p_n)$ in $\mathcal S^\N$ converging to $x$. 
On the upper half plane $\mathbb H$
we have that 
$G^+\circ \psi^u_{p_n}(\zeta) = \im(\zeta)$. After possible extraction, with notation as above,
$\psi^u_{p_n}$ converges uniformly 
to $\widehat \psi$, so by continuity of $G^+$ we get that 
\begin{equation}\label{eq:g+_rond_psi}
G^+\circ \widehat \psi (\zeta)  = G^+\circ \psi^u_x\circ h= \im(\zeta) \; \text{ on } \; \mathbb H,
\end{equation}
 so  $G^+\circ \psi^u_x$ takes positive values arbitrary close to 0. 
\end{proof}

The following result is reminiscent from~\cite[Prop. 2.2]{BSR}. 

\begin{lem}\label{lem:tau_bounded}
If $f$ is unstably linear, then $\tau(x)\leq 2$ for every $x\in \jstar$. 
\end{lem}

\begin{proof}
Assume that $x\in \jstar$ is such that $\tau(x) = k$. Then we have a sequence 
$\psi^u_{p_n}$ converging to $\widehat \psi =\psi^u_x\circ h$ with 
$h(\zeta) = c\zeta^k+ \hot$ at the origin.  
From Equation~\eqref{eq:g+_rond_psi}, we have   
$G^+\circ \widehat \psi (\zeta)  =  \im(\zeta)$.
On the other hand,  locally near the 
origin, $\set{G^+\circ \widehat\psi  = 0}$ has $k$-fold rotational symmetry. 
 This is possible only if $k\leq 2$. 
\end{proof}

For $i=1, 2$, set $\jstar_i = \set{x\in \jstar\, ;\,  \tau(x)= i}$, which is an invariant set (denoted by $\mathcal J_i$ in \cite{bs8}).  Theorem 6.7 in~\cite{bs8} shows that  $\jstar_1$ is open and dense in $\jstar$, so $\jstar_2$ is closed. From the comments preceding Lemma 6.4 in~\cite{bs8}, we see that  (in their notation) $\mathcal J'_2 = \mathcal J_2$, so it follows from Lemma 6.5 there that 
for every $x\in \jstar_2$, if $\widehat \psi$ is a non-injective parametrization of 
$\mathcal W^u(x)$, then $\widehat \psi$ is of the form $\zeta\mapsto \psi^u_x(c\zeta^2)$ with $\psi^u_x$ injective (see also~\cite[Prop. 2.6]{BSR} for a related argument). 

\begin{lem}\label{lem:tau_bounded_1}
If $f$ is unstably linear, then $\jstar = \jstar_1$; in other words
$\tau(x)\leq 1$ for every $x\in \jstar$. 
\end{lem}

\begin{proof}[First proof of the lemma]
Assume that $\jstar_2$ is non-empty.
We claim that for every $x\in \jstar_2$, $(\psi^u_x)\inv(J)$ is contained in a half line.  
Indeed,  if $\widehat \psi\in \widehat \Psi$ is  a non-injective parametrization of $\mathcal W^u(x)$ with $\widehat \psi(0) = x$, then as in Equation~\eqref{eq:g+_rond_psi}
we have $G^+\circ \widehat \psi(\zeta)   = \im(\zeta)$ in 
$\overline {\mathbb H}$, and by the previous comments 
$G^+\circ \widehat \psi(\zeta) = G^+\circ   \psi^u_x(c\zeta^2)$. If we rotate 
 $\psi^u_x$ so that $c=1$, we see that $\set{G^+\circ   \psi^u_x = 0}$ 
 is the positive real axis, as asserted. 

 Since $\jstar_2$ is a closed invariant set, it supports an ergodic  invariant measure $\nu$. 
By~\cite[Thm 6.2]{bs8}, $\nu$ has a positive Lyapunov exponent, 
hence also a negative one since $\abs{\jac(f)}\leq 1$. If $y$ is a $\nu$-generic point, its stable manifold intersects transversally  the unstable manifold of some saddle point $p$, 
and at such a point $J^+_{W^u(y)}$ has a tangent line. Since $y$ is recurrent,  from Lemma~\ref{lem:pesin_tangent} below, 
$(\psi^u_y)\inv\big(J^+_{W^u(y)})$ is a line, which is a contradiction. \end{proof}

\begin{proof}[Second proof of the lemma]
As in the first proof,  if $\jstar_2$ is non-empty, then it  supports  
a hyperbolic,  ergodic,  invariant probability measure $\nu$. Theorem 6.2 in~\cite{bs8} asserts that the positive exponent of 
$\nu$ is  not smaller that $2\log \kappa = 2 \log d$ (see Lemma~\ref{lem:expansion_factor}). It is a general fact from Pesin theory 
that the Lyapunov exponents of hyperbolic measures are approximated by those of periodic orbits 
(see~\cite{wang-sun}; in our setting we can directly adapt~\cite[Thm S.5.5]{KH}).
But all  Lyapunov exponents of periodic points are equal to $\log d$: this contradiction finishes the proof. 
\end{proof}

We are now ready to complete the proof of Theorem~\ref{thm:no_unstably_linear}. The previous lemma shows that if $f$ is unstably linear, then $\jstar_1 = \jstar$. Therefore the local   manifolds 
$\mathcal W^u_\loc(x)$  
form a lamination in a neighborhood of $\jstar$ (see~\cite[Prop. 5.3]{bs8}). 
In the vocabulary of~\cite{topological}, every point in $\jstar$ is $u$-regular. In addition by Lemma~\ref{lem:vanish}, 
$G^+$ does not vanish identically along any of the $\mathcal W^u_\loc(x)$, $x\in \jstar$,   and Lemma~\ref{lem:bilan_unstably_linear} shows that  $\abs{\jac(f)}\leq 1$. 
Therefore,   \cite[Prop 2.16]{topological} (or~\cite[Thm 2.18]{topological}) 
shows that $f$ is hyperbolic. 
This contradicts  Corollary~\ref{cor:hyperbolic_unstably_linear}, and the proof is complete. 
\qed

\subsection{Questions on unstably real automorphisms} \label{subs:questions_unstably_real}
Up to conjugacy, the only examples of unstably real mappings that we know of 
are  real automorphisms with maximal entropy $h_{\rm top}(f\rest{\R^2}) = h_{\rm top}(f)$. In particular they are stably real as well.

\begin{que}~\label{question:real}
\begin{enumerate}
\item Is every unstably real automorphism conjugate to a real automorphism with maximal entropy? \item   If 
$f$ is unstably real, then is $\jac(f)$ real?  
\item Assume that all unstable multipliers of $f$ are real, then is $f$ unstably real? 
\end{enumerate}
\end{que}

The second question is a weak form of the first one. We refer to~\cite{eremenko-vanstrien} for a positive answer to the third question in dimension $1$.

\subsection{Comments on other affine surfaces} Consider the affine surface $S=\C^\times \times \C^\times$, i.e. the complex multiplicative group of dimension $2$. The group $\GL_2(\Z)$ acts by automorphisms on $S$: if $M=\lrpar{ \begin{smallmatrix} a&b \\ c&d\end{smallmatrix}} \in \GL_2(\Z)$, 
the corresponding automorphism is $f_M\colon (x,y)\mapsto (x^a y^b, x^c y^d)$. Then, $f_M$ is loxodromic if and only if $M$ has an eigenvalue $\lambda$ with $\vert \lambda \vert > 1$ (the second eigenvalue is $\det(M)/\lambda$, and $\det(M)=\pm 1$). Such a loxodromic monomial automorphism has positive entropy, preserves a pair of holomorphic foliations, and has a Julia set which is a real analytic surface, namely $\set{(x,y)\; ; \; \vert x \vert =1= \vert y\vert}$. A similar example is obtained by looking at the Cayley cubic surface $S/\eta$, where $\eta(x,y)=(1/x,1/y)$ 
(see~\cite{cantat:BHPS, cantat-loray}). Thus, Theorems~\ref{mthm:no_real_global_foliation} and~\ref{mthm:no_unstably_linear}  do not extend directly to arbitrary affine surface. Following the arguments presented in this article in view of such a generalization, the main missing input  to characterize unstably linear automorphims of affine surfaces 
would be a version of the  Bedford-Kim theorem on multipliers of fixed points (Proposition~6.1 in~\cite{BK1}). 

\begin{que}~\label{question:affine_surfaces} Let $f$ be an automorphism of a complex affine surface $X$, 
with first dynamical degree $\lambda_1(f)>1$. Assume that for every sadlle periodic point, except finitely many of them, the unstable multiplier $\lambda^u_p$ is equal to $\pm \lambda_1(f)^n$, where  $n$ is the period of $p$. Is $X$ isomorphic to $\C^\times \times \C^\times$ or its quotient by $(x,y)\mapsto (1/x, 1/y)$?
\end{que}


\section{Rectifiable Julia sets}\label{sec:rectifiable}

Examples of automorphisms such that $J^+_\Gamma$ is locally  a curve include perturbations of $1$-dimensional hyperbolic 
polynomials with quasi-circle Julia sets.  Here we  push 
the techniques of the previous section one step further to show the following result (see Section~\ref{par:rectifiable_sets} for an introduction to rectifiability).

\begin{thm}\label{thm:rectifiable}
Let $f\in \Aut(\C^2)$ be a loxodromic automorphism. 
If  $\Gamma$ is a transversal to    $J^+$,  then  $J^+_\Gamma$ is  not  a rectifiable curve. 
\end{thm}

 Since a Jordan arc has finite $1$-dimensional measure if and only if it is rectifiable (see e.g.~\cite[Lem. 3.2]{falconer}), we obtain the following corollary.

\begin{cor}\label{cor:Jordan_Hausdorff}
 If  $\Gamma$ is a transversal to    $J^+$ such that 
  $J^+_\Gamma$ is a Jordan arc,
  then its $1$-dimensional Hausdorff measure is infinite.  
\end{cor}

This generalizes some results of  Hamilton~\cite{hamilton} in $1$-dimensional dynamics to plane polynomial automorphisms.  
In dimension $1$,  a deeper result states that if a polynomial Julia set is not totally disconnected, 
then either it is smooth (a circle or an interval) or 
its Hausdorff dimension is greater than $1$. 
In the connected case, this follows from Zdunik's results in~\cite{zdunik:dimension} 
together with Makarov's celebrated 
theorem~\cite{makarov}
on the dimension of plane harmonic measure;  the general case is established  
 by Przytycki and Zdunik in~\cite{przytycki-zdunik}. 
In our setting, one may expect that such an alternative holds for unstable slices (except that, as we saw,  
smooth examples   do not exist). 

\subsection{An extension of Proposition~\ref{pro:contained_curve}} \label{subs:extension_curve}
Recall that a subset $F\subset \R^2$ has a (geometric) tangent at $x$ if for every 
$\theta>0$ there exists $r>0$ such
that $F\cap B(x,r)$ is contained in a (two-sided) angular sector of width $\theta$. This notion is invariant under diffeomorphisms so it makes sense on any real surface. The first step is a generalization of Proposition~\ref{pro:contained_curve}  (see also Remark~\ref{rem:pesin}). 

 \begin{pro}\label{pro:tangent}
 Let $f\in \Aut(\C^2)$ be a loxodromic automorphism. Let $\Gamma$ be a transversal to~$J^+$.   
 If $J^+_\Gamma$ admits a tangent 
 on a set of positive $T^+\wedge [\Gamma]$-measure, then 
 $f$ is unstably real. 
 \end{pro}

Before starting the proof, let us fix some vocabulary. Since the canonical measure  $\mu = \mu_f=T^+\wedge T^-$ is hyperbolic, by Pesin's theory, $\mu$-almost every point admits a stable and an unstable manifold, biholomorphic to $\C$ (see e.g.~\cite{bls} for details).  We choose a 
measurable family of    unstable parametrizations $\psi^u_x:\C\to W^u(x)$, normalized by 
\begin{equation}\label{par:normalisation_norme_1}
\psi^u_x(0)=x \; {\text{  and  }} \;
\norm{(\psi^u_x)'(0)}= 1.
\end{equation}
 In these parametrizations, $f$ acts linearly. 
We say that \emph{$W^u(x)$ has size $r$ at $x$} if it
is a graph of slope at most 1 over  the disk of radius $r$ in $E^u(x)$ relative to the orthogonal projection $\pi$ from $\C^2$ to $E^u(x)$.
Then   
the Koebe distortion theorem gives universal bounds on the distortion of $\pi\circ \psi^u_x$ on any disk of radius $<r/4$ (see~\cite[Lem. 3.7]{berger-dujardin}).  We denote by $W^u_r(x)$ the connected
component of $W^u(x)\cap B(x, r)$ containing $x$, and by $e^u(x)$ a unit vector tangent to $W^u(x)$ at~$x$.

\begin{proof}  \hfill \\
{\bf{Step 1: an intermediate case}}.-- 
In a first stage, we assume that for some $\mu$-generic $x$ (that is, 
satisfying a finite number of full measure conditions to be made clear below), 
$(\psi^u_x)\inv\big(J^+_{W^u(x)}\big)$ 
has a tangent at $x$. We claim that $f$ is unstably real.

 Indeed, set
$I(x) = (\psi^u_x)\inv\big(J^+_{W^u(x)}\big)$, fix some $\rho>0$, and let $A(x)\subset \N$ be the set 
of integers $n$ such that $W^u(f^n(x))$ has size  at least $\rho$ at $f^n(x)$. For generic $x$ and sufficiently small $\rho$, the lower density of $A(x)$ in $\N$ is arbitrary close to 1. 
Fix such a $\rho$, 
and assume  furthermore  that 
$\norm{df^n_x(e^u(x))}$ tends to infinity, which   again is a generic property. 

Viewed in the unstable parametrizations, $f^n$ acts linearly; so, by the normalization~\eqref{par:normalisation_norme_1}, 
it  maps a disk of radius $r$ in $(\psi^u_x)\inv(W^u(x))$  to a disk of radius 
$ \norm{df^n_x(e^u(x))} r$ in $(\psi^u_{f^n(x)})\inv (W^u_{f^n(x)})$. Thus, if $I(x)$ has a tangent at $x$, it follows that for   
every   $\theta>0$ and large enough $n$, 
$I(f^{n}(x))\cap \disk(0;\rho)$ is contained in an angular  sector  $S_n$ 
of width $\theta$.  

Since $A(x)$ has positive density, we may fix a subsequence 
$n_j$ in $A(x)$ such that $f^{n_j}(x)$ converges to a Pesin generic point $x_0$, whose unstable manifold has size $4\rho'$, with $\rho'<\rho/4$  
and 
$W^u_{\rho'}(f^{n_j}(x))$ converges in the $C^1$ sense to $W^u_{\rho'}(x_0)$.  
 Choose local coordinates $(\xi,\eta)\in \disk(0, \rho')^2$  
  such that  $W^u_{2\rho'}(x_0)\cap \disk(0, \rho')^2 
 = \set{\eta=0}$ and $W^u_{2\rho'}(f^{n_j}(x))\cap \disk(0, \rho')^2$ is a graph over the 
 first coordinate, and denote by $\pi\colon (\xi,\eta)\mapsto \xi$  the  first projection. As explained above, 
 since $n_j\in A(x)$, by 
  the Koebe distortion theorem,  $(\pi\circ \psi^u_{f^{n_j}(x)})$ is a sequence of univalent mappings 
   in some fixed disk, say  $\disk(0, \rho'/10)$. 
 Upon extraction we may assume that it converges to some univalent 
 holomorphic map $\gamma:\disk(0, \rho'/10) \to \disk(0, \rho')$, and extracting further if necessary, 
the images of sectors $\pi\circ \psi^u_{f^{n_j}(x)}(S_{n_j})$ 
converge to the analytic curve $\gamma(e^{i\theta}\R)$. 
Exactly as in 
Proposition~\ref{pro:contained_curve},  the 
lower semicontinuity of $\Gamma\mapsto J^+_\Gamma$ 
  implies  that this limit curve is unique:  if not, 
 we would get at point of $J^+_{\set{y=0}}$ outside $\gamma(e^{i\theta}\R)$, which would prevent 
 $J^+_{W^u_{\rho'/10}(f^{n_j}(x))}$ to be contained in an arbitrary small angular sector for large $j$. 
  Hence,  $J^+_{W^u_{\rho'/10}(x_0)}$ is contained in a smooth curve, and since 
  $W^u_{\rho'/10}(x_0)$ is a transversal to $J^+$
  we conclude from Proposition~\ref{pro:contained_curve} that $f$ is unstably real. 
 
For further reference, let us record a mild generalization of what we just proved, where we incorporate an additional inclination lemma argument. 

\begin{lem}\label{lem:pesin_tangent}
Let $\nu$ be a hyperbolic, ergodic, $f$-invariant measure. There exists a set $E$ of full measure 
made of Pesin regular points, with the following property. Let $\Gamma$  be a  transversal  to $J^+$,
let $x$ be an element of $E$ ,  and assume that for some sequence $(n_j)$, $f^{n_j}(x)$ converges to a point $y\in E$. If $J^+_\Gamma$ admits a tangent 
  at a transverse intersection point $t\in  W^s(x) \cap \Gamma$, then   $(\psi^u_y)\inv\big(J^+_{W^u(y)})$ is a line. 
\end{lem}

 \medskip
 
\noindent {\bf{Step 2: conclusion}}.-- We first state a  lemma which will be  proven afterwards.

 \begin{lem}\label{lem:laminar_transverse} 
Let $\Gamma$ be any transversal to $J^+$. 
For every subset 
$A$ of full $\mu_f$-measure there exists  
$E \subset \Gamma$ of full   $(T^+\wedge [\Gamma])$-measure such that if   $t\in E$, there exists 
$x\in A$ such that 
\begin{itemize}
\item $t\in W^s(x)$,
\item $W^s(x)$ intersects $\Gamma$ transversally at $t$. 
\end{itemize}
\end{lem}

Thus, if $J^+_\Gamma$ 
has a tangent at $t$ for $t$ in a set of positive $(T^+\wedge [\Gamma])$-measure, 
 applying Birkhoff's theorem together with  Lemma~\ref{lem:pesin_tangent}, 
 we obtain  that $(\psi^u_y)\inv\big(J^+_{W^u(y)}\big)$ 
has a tangent at $y$ for $y$ in a set of total $\mu_f$-measure. The conclusion then 
follows   from the first step.
\end{proof}

\begin{proof}[Proof of Lemma~\ref{lem:laminar_transverse}]
This is identical to~\cite[Lem. 3.3]{DF}, except that the algebraic curve $C$ in that lemma needs to be  
  replaced by the transversal $\Gamma$. To adapt the proof we make use of \cite[Thm 3]{bs2}: 
since $\int T^+\wedge [\Gamma]>0$ and, reducing 
$\Gamma$ slightly if necessary,   the trace measure $\sigma_{T^+}$ gives no mass to $\fr \Gamma$,
the sequence of currents $d^{-n} f^n_\varstar [\Gamma]$ converges to some positive multiple of $T^-$. 
\end{proof}

\subsection{Proof of Theorem~\ref{thm:rectifiable}}
Assume that for some transversal disk $\Gamma$, $J^+_\Gamma$ is a rectifiable curve. We will show that $f$ is unstably linear, which is impossible by Theorem~\ref{thm:no_unstably_linear}. 

Let $\Omega\subset \C$ be a domain.  In what follows, we denote by $\omega(x, A, \Omega)$ the harmonic 
measure   of a subset $A\subset \fr \Omega$ viewed from $x$ in $\Omega$. 
In all the cases considered below, $\Omega$ will be  a Jordan domain, and $\omega(x, A, \Omega)$ is 
  the value at $x$ of the solution of the Dirichlet problem in $\Omega$ with boundary value $\mathbf 1_A$. 
In other words,
  if $\varphi$ is a continuous function on $\partial \Omega$ and $u_\varphi$ 
  is its harmonic extension to $\Omega$, then $\int \varphi(y) \omega(x, dy,\Omega) = u_\varphi(x)$ 
   (see
 ~\cite{garnett-marshall} for a detailed  account). If $\phi:\Omega\to \phi(\Omega)$ is a biholomorphism that extends to a homeomorphism $\overline \Omega \to 
  \overline{\phi(\Omega)}$, then 
  $\omega(x, A, \Omega) = \omega(\phi(x), \phi(A), \phi(\Omega))$ (conformal invariance).

\begin{lem}\label{lem:harmonic_measure}
For any $x_0\in \Gamma$ such that $G^+(x_0)>0$,
there exists an open subset $\Omega$ of $\Gamma$ containing $x_0$ whose boundary contains a relatively open 
subset of  the Jordan curve $J^+_\Gamma$, an open subset $U$ of $\overline \Omega$ 
 intersecting $J^+_\Gamma$  
 and a positive constant $c$ such that 
 $\omega(x_0, \cdot , \Omega)\rest{U\cap J^+_\Gamma}\leq c \Delta G^+ \rest{U}$.
 \end{lem}

We provide below a proof of this lemma, which is certainly known to some experts.

Let us  complete the proof of the theorem. Since $J^+_\Gamma$ is a rectifiable curve, a classical theorem of F. and M. Riesz 
asserts that  the harmonic measure is mutually absolutely continuous with the $1$-dimensional Hausdorff measure $\mathcal H^1$
on $J^+_\Gamma$ 
(see \cite[Thm VI.1.2]{garnett-marshall}). In addition, $J^+_\Gamma$ admits a tangent at  $\mathcal H^1$-almost every point 
(this is due to Besicovitch, see~\cite[Thm 3.8]{falconer} (\footnote{In~\cite{falconer}, tangents are generally
 understood in a measurable sense (see \S~\ref{subs:rectifiable_rectifiable} below, but it is clear that in Theorem~3.8 a geometric tangent is obtained.})).
Thus, from Lemma~\ref{lem:harmonic_measure} and 
Proposition~\ref{pro:tangent}, we infer that for every saddle or Pesin generic point $p$, $(\psi^u_p)\inv(J^+_{W^u(p)})$ is contained in a line. 
Furthermore, $ (\psi^u_p)\inv\big(J^+_{W^u(p	)}\big)$ 
has no compact component, for 
such a compact component would give rise to small components of $J^+_\Gamma$ under stable holonomy (see~\cite[Prop. 1.8]{connex}). Therefore we conclude that 
$ (\psi^u_p)\inv\big(J^+_{W^u(p	)}\big)$ is a line, as asserted.

\subsection{Proof of Lemma~\ref{lem:harmonic_measure}} To explain the idea, suppose for a moment that $\Gamma$ is a piece of unstable manifold of some saddle periodic point $p$.
Since  $J^+_\Gamma$ is a rectifiable curve,  $ (\psi^u_p)\inv\big(J^+_{W^u(p)}\big)$ 
is a line that cuts $\C$ in two half-planes and it follows from 
 Lemma~\ref{lem:bilan_unstably_linear}
 that $\Delta (G^+\circ \psi^u_p)$ is proportional to the Lebesgue measure on the line 
 $(\psi^u_p)\inv\big(J^+_{W^u(p)}\big)$. 
Now, if $\Omega$ is a domain with piecewise smooth boundary intersecting  $(\psi^u_p)\inv\big(J^+_{W^u(p)}\big)$ along a segment, then the harmonic measure has a smooth density along this segment (see~\cite[\S II.4]{garnett-marshall}, and the result follows.  
  
%

For a general transversal $\Gamma$ we need a local and rectifiable 
 version of this picture. Let 
$\Omega_0\subset \Gamma$ be a smoothly bounded simply connected domain containing $x_0$ such that 
$\Omega_0\cap J^+_\Gamma$ is an arc separating $\Omega_0$ into two components; such a set exists because $J^+_\Gamma$ is rectifiable. Let $\Omega$ be the component containing $x_0$. 
Consider a Riemann uniformization 
$\phi: \mathbb{H}\to \Omega$ that maps $\infty$ outside $J^+_\Gamma$. Since 
$J^+_\Gamma$ is rectifiable, 
$\phi$ extends continuously to the boundary, and $\phi\inv (J^+_\Gamma\cap \fr\Omega)=:I$ is  
a segment of the real axis. Set $\tilde G = G^+\circ \phi$. Since 
$\tilde G=0 $ on $I$,  by Schwarz reflexion, $\tilde G$ extends to a harmonic function  
across the interior $\mathrm{Int}(I)$ 
 of $I$; the extension is positive on one side and negative on the other side. Let $\zeta\in \mathrm{Int}(I)$ be a point at which the differential of $\tilde G$ does not vanish. In some 
  neighborhood $\tilde U$
of  $\zeta$, the level sets   $\big\{\tilde G = \e\big\}\cap \tilde U$, with $\epsilon >0$, are graphs over the real axis, converging in the $C^1$ topology to $\tilde U\cap I$ as  $\e\to 0$.  
We then shrink $\Omega$ as in Figure~\ref{fig:omega} to insure that it coincides with $\phi(\tilde U)$ near $\phi(\zeta)$.

\begin{figure}[h]
\includegraphics[width=13cm]{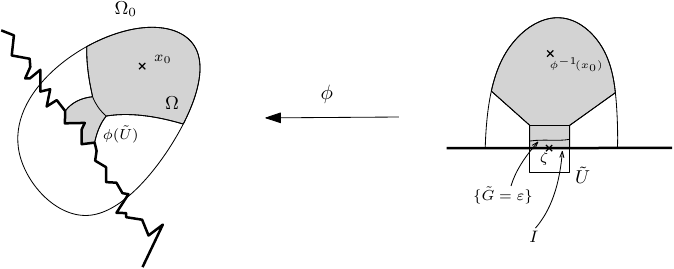}
\caption{{\small Sets appearing in the proof of Lemma~\ref{lem:harmonic_measure}  ($\Omega$ is the shaded region).  }}\label{fig:omega}
\end{figure}

Then, for small $\e$,  the   harmonic measure of $\big\{\tilde G = \e\big\}\cap\tilde U$ 
 viewed from $\phi\inv(x_0)$ 
  is  equivalent to the arclength measure on $\big\{\tilde G = \e\big\}\cap\tilde U$, and the implied constants are uniform. 
  Likewise,  $\Delta (\max(\tilde G, \e))$  is equivalent to the
  arclength measure on $\big\{\tilde G = \e\big\}\cap\tilde U$, with uniform constants: indeed since $\tilde G$ is locally the imaginary part of a univalent holomorphic function, there is a holomorphic coordinate 
  $z= x+ i y$ in which $\max (\tilde G, \e)$ becomes $\max(y,\e)$; thus, its Laplacian is a multiple of  the Lebesgue measure on $\set{y=\e}$ and when $\e$ goes to $0$, these measures converge to a measure supported on $\tilde U\cap I$, 
  which is equivalent to   Lebesgue measure, with uniform constants (see~\cite[Cor. 4.7 of Chapter II]{garnett-marshall} for similar computations).

 Now we transport this to $\Gamma$ by $\phi$.  The conformal  invariance of  the Laplacian and the harmonic measure shows
 that $\Delta(\max(G^+, \e))\rest{\phi(\tilde U\cap \mathbb{H})}$ 
  converges to a measure equivalent to the harmonic measure $\omega(x_0, \cdot , \Omega)$
on  a piece of $J^+_\Gamma$ (namely  $J^+_\Gamma\cap \overline{\phi(\tilde U)}$). On the other hand 
   this limit is dominated by $\Delta G^+$ 
 (the domination may be strict because we are considering only one side of the curve and $G^+$ could be positive on the other side as well), and the result follows.  \qed

\subsection{Rectifiable curves vs rectifiable sets and measures}\label{subs:rectifiable_rectifiable} (\footnote{This section can be skipped on a first reading.})

\subsubsection{Rectifiable sets (see~\cite{falconer} for an introduction)}\label{par:rectifiable_sets}

A curve $\gamma\colon [0,1]\to X$ drawn in a metric space is rectifiable if one of the following equivalent properties is satisfied: (a) there is a homeomorphism $\varphi\colon [0,1]\to [0,1]$ such that $\gamma\circ \varphi$ is Lipsccitz;  (b) there is a uniform bound $L$ such that $\sum_{i=0}^{n-1} \dist_X(\gamma(t_i), \gamma(t_{i+1}))\leq L$ for any finite sequence $t_0=0<t_1< \cdots < t_n=1$. The supremum of these sums is then equal to the length of $\gamma$. When $X$ is a real vector space, this means that the length of $\gamma$ is the supremum of the lengths of the polygonal approximations $([\gamma(t_i),\gamma(t_{i+1})])_i$. It is also equal to the $1$-dimensional Hausdorff measure $\mathcal H^1(\gamma([0,1]))$.

Now, let $E$ be a subset of the place $\R^2=\C$. One says that $E$ is \emph{rectifiable} 
(\footnote{Rectifable sets are called regular $1$-sets in~\cite{falconer}. Here, we use the fact that regular $1$-sets are $Y$-sets up to subsets of $\mathcal H^1$-measure $0$ (see~\cite{falconer}, p. 33). And approximate tangents are just called tangents in \cite{falconer}.}) 
if $0<\mathcal H^1(E)<+\infty$
and there exists an at most countable collection of Lipschitz curves $(C_i)_{i\in I}$ such that
$\mathcal H^1\lrpar{E\setminus \bigcup_{i\in I} C_i}  = 0$. The following properties are \cite[Thm 3.14]{falconer} and \cite[Cor 3.10]{falconer}:
\begin{enumerate}
\item any compact connected subset $E\subset \R^2$ with $\mathcal H^1(E)<\infty$ is rectifiable;
\item if $E$ is rectifiable, then  at $\mathcal{H}^1$-almost every $x\in E$ there is an 
\emph{approximate tangent} to $E$ at $x$. That is, there exists a line $L$ such that  
\begin{equation}
\lim_{r \to  0} \frac{1}{r} \mathcal H^1( (E\cap B(x,r)) \setminus  S(L, \theta)) = 0
\end{equation} 
for every $\theta>0$, where $S(L, \theta)$ is the symetric angular sector of angle $\theta$ about 
$L$. 
\end{enumerate}
Note that since $E$ can be dense, we cannot expect to have a tangent in the geometric sense. 

\subsubsection{Rectifiable measures (see e.g.~\cite[\S4]{delellis:lectures_on_rectifiability}). }Closely related is the notion of \emph{rectifiable measure}. A positive measure $\m$ is rectifiable (or $1$-rectifiable) if it is of the form 
\begin{equation}
\m = \varphi \mathcal H^1\rest{E},
\end{equation} 
where $E$ is rectifiable and $\varphi\colon E\to \R_+$ is a Borel function. 

To look at the local properties of such a measure, define $\tau_{x, r}(\xi) = (\xi-x)/r$. Then, 
for $\m$-almost every $x$ there is a line $L_x\subset \R^2$  and 
a constant $c_x>0$ such that 
\begin{equation}
 \frac{1}{r} (\tau_{x, r})_\varstar \m \longrightarrow  c_x\mathcal H^1\rest{L_x}
 \end{equation}
  weakly as $r\to 0$. 

\begin{rem}\label{rem:normalization}
Observe  that if $c_r>0$ is any normalizing factor, chosen so that 
$c_r(\tau_{x, r})_\varstar \m$  is a probability measure in some fixed neighborhood of 
$0$, then $ c_r (\tau_{x, r})_\varstar \m \to  c'_x\mathcal H^1\rest{L_x}$ for some positive $c'_x$. 
\end{rem}

\subsubsection{Dynamical applications} Let $f$ be a loxodromic automorphism of the plane. 
By \cite[Prop 3.2]{bls}, for 
 $\mu_f$-almost every $x$, the unstable conditional $\mu_f(\cdot\vert W^u_\loc(x))$ is  induced by $T^+$. That is, 
\begin{equation}\label{eq:conditionals}
\mu_f(\cdot\vert W^u_\loc(x)) = \frac{T^+\wedge W^u_\loc(x)}{\mathbf {M}(T^+\wedge W^u_\loc(x))}.
\end{equation}
The Hausdorff dimension of this conditional 
measure is $\mu_f$-almost everywhere equal to a constant denoted by 
$\dim^u(\mu_f)$. 
We also denote by $\chi^u(\mu_f)$ the positive Lyapunov exponent of $\mu_f$, and recall 
Young's  formula: 
\begin{equation}\label{eq:Young_Formula}
\dim^u(\mu_f) =  \frac{\log d}{\chi^u(\mu_f)},
\end{equation}
where moreover  $\log(d)=h_\mu(f)=h_{top}(f)$. In particular, $\dim^u(\mu_f)=1$ if and only if $h_\mu(f)=\int  \log \norm{Df\rest{E^u(x)}} \, \d\mu_f$.

\begin{thm}\label{thm:rectifiable_measures}
If $f$ is a loxodromic automorphism of $\C^2$, then  for $\mu_f$-almost every $x$, the unstable conditional 
$\mu_f(\cdot\vert W^u_\loc(x)) $ is not a rectifiable measure. 
\end{thm}

The following consequence is  in the spirit of the results of Przytycki and Zdunik mentioned after Corollary~\ref{cor:Jordan_Hausdorff}. 

\begin{cor}\label{cor:dimH}
Let $f$ be a loxodromic automorphism of $\C^2$.
If the dynamics of $f$ is hyperbolic, $\abs{\jac(f)}\leq 1$, and the Julia set $J$ is connected, then for every $x\in J$, 
$$\dim_H(J\cap W^u_\loc(x)) = \dim_H(J^+\cap W^u_\loc(x)) >1.$$ 
\end{cor}

Note that  under these assumptions, by~\cite[Cor. A.3]{bs7}, $f$  must actually be dissipative. 

\begin{proof}[Proof of Corollary~\ref{cor:dimH}] 
By hyperbolicity, we know that  
$J\cap W^u_\loc(x) = J^+\cap W^u_\loc(x)$.

Since $\abs{\jac(f)}\leq 1$, $f$ is unstably connected. Hence, 
by \cite[Thm 3.5]{bs7},  $J\cap W^u_\loc(x)$ has 
finitely many connected components, which are not points.  Thus 
\begin{equation}
\delta^u := \dim_H(J\cap W^u_\loc(x))   \geq 1.
\end{equation}

 Assume by contradiction that this unstable dimension $\delta^u$ equals $1$. We now rely on a few notions from thermodynamics formalism,   for which  we refer to~\cite{pesin_book} and Chapter 20 of~\cite{KH}. 
 Following~\cite{pesin_book} (Appendix II and Section 22), we denote by $\kappa^u$ the equilibrium state associated to the 
 geometric potential $\varphi=- \log \norm{Df\rest{E^u(x)}}$. By definition, 
 $\kappa^u$ maximises the pressure 
 \begin{equation}
 P(\m , \varphi )=h_\m(f)+\int_J \varphi\, \d\m 
 \end{equation}
 among all $f$-invariant probability measures $\m$ supported on $J$. By Ruelle's inequality, and Young's Formula~\eqref{eq:Young_Formula}, the maximum is $0$, and $\mu_f$ realizes this maximum. Thus, by uniqueness of $\kappa^u$ (see~\cite{KH}, Corollary 20.3.8), we have $\mu_f=\kappa^u$, and the root of Bowen equation is $t^u=1$ (see~\cite{pesin_book}, Appendix II, p. 106). Now, Theorem~22.1 of~\cite{pesin_book} implies that the unstable conditionals   $\mu_f(\cdot \vert J\cap W^u_\loc(x))$ have positive and locally integrable  density with respect to 
 $\mathcal H^1\rest{J\cap W^u_\loc(x)}$, thus they 
 are rectifiable measures, and Theorem~\ref{thm:rectifiable_measures} yields a 
contradiction.
\end{proof}

 \begin{proof}[Proof of Theorem~\ref{thm:rectifiable_measures}]
Assume that $\mu_f(\cdot\vert W^u_\loc(x)) $ is  rectifiable for a set of positive $\mu_f$-measure.  Since this property is $f$-invariant, by ergodicity this holds for $\mu_f$-almost every $x$. We will show that 
$f$ is unstably linear, which is impossible by Theorem~\ref{thm:no_unstably_linear}. For this, we 
  reproduce Proposition~\ref{pro:tangent} by keeping track of the unstable conditionals.  

Let $x$ be Pesin-generic. Since  $f$ acts linearly in unstable parameterizations, $f^n$
 maps a disk of radius $r$ in $(\psi^u_x)\inv(W^u(x))$  to a disk of radius 
$ \norm{df^n_x(e^u(x))} r$ in $(\psi^u_{f^n(x)})\inv (W^u_{f^n(x)})$. Pick a subsequence 
$(n_j)$ such that $W^u_{\rho}(f^{n_j}(x))$ converges in the $C^1$ sense to 
$W^u_{\rho}(x_0)$ for some generic $x_0$. The invariance of $\mu_f$ implies that 
$(f^n)_\varstar \mu_f(\cdot\vert W^u_\loc(x))\rest{W^u_\loc(f^n(x))}$ is proportional to  
$\mu_f(\cdot\vert W^u_\loc(f^n(x)))$.   
From Equation~\eqref{eq:conditionals}, the $C^1$ convergence of unstable manifolds, 
  and the continuity of the potential of $T^+$,  we get 
\begin{equation}
\mu_f(\cdot\vert W^u_\loc(f^{n_j}(x))) \to \mu_f(\cdot\vert W^u_\loc( x_0))
\end{equation}
as $j$ goes to $+\infty$ 
 (for this, we need to choose the radii 
 of local unstable manifolds so that no mass is carried on the boundary). 
 
 By the rectifiability assumption, viewed in the 
 unstable parameterizations 
 we have that 
 \begin{equation}
c_n (\psi^u_{f^{n_j}(x)})\inv_\varstar \lrpar{(f^{n_j})_\varstar \mu_f(\cdot\vert W^u_\loc(x))\rest {W^u_\loc(f^{n_j}(x))}} \underset{j\to\infty}\longrightarrow c\mathcal H^1\rest{L},
 \end{equation}
where $c>0$, $L$ is a line through $0\in \C$, and $c_n$ is a normalization constant (which 
satisfies $c_n\asymp d^n$ due to the invariance property of $T^+$, 
see Remark~\ref{rem:normalization}). 

Fix local coordinates $(\xi, \eta)$
such that 
$W^u_\loc(x_0) = \set{\eta=0}$. For large $n_j$,
 $W^u_\loc(f^{n_j}(x))$ is a graph over the first coordinate.
Then   the Koebe distortion theorem and the 
 previous analysis show that $\supp(T^+\wedge [W^u_\loc(x_0)])$ 
 is a smooth curve through the origin, and $f$ is unstably linear, as asserted. 
 \end{proof}

\section{Invariant holomorphic foliations}\label{sec:foliations}
 Let $\mathcal F$ be a foliation with complex leaves in some open set $U\subset \C^2$. We allow $\mathcal F$ to have isolated singularities (see below). By definition, the leaf $\mathcal F(q)$ of $\mathcal F$ through a point $q\in U$ is   reduced to $\{q\}$ if $q$ is a singularity, and otherwise it is the leaf of the regular folitation defined by $\mathcal F$ in the complement of its singularities.
We say that a set $S\subset \C^2$ is {\emph{subordinate to $\mathcal F$}} in $U$ if $S\cap U$ is not empty and is {\emph{$\mathcal F$-saturated}}, that is, for every $q\in S$ the leaf $\mathcal F(q)$ of $\mathcal F$ through $q$ is also contained in $S$. 

In this section we address the following question: is it possible that $J^+$ be locally (resp. globally) subordinate to a holomorphic foliation? The case of real analytic foliations will be considered in the next section.

\subsection{Preliminary observations}\label{par:invariance} 
Let us first gather a few general facts.
  
\begin{pro}\label{pro:subordination}
Let $\mathcal F$ be a continuous, regular foliation with complex leaves in some connected open set $U$. The following properties are equivalent:
\begin{enumerate}[\rm (a)]
\item the forward Julia set $J^+$ is subordinate to $\mathcal F$ in $U$;
\item for some saddle periodic point $p$, every holomorphic 
disk   contained in $W^s(p)\cap U$ is contained in a leaf of $\mathcal F$. 
\end{enumerate}
If these properties hold, then:
\begin{enumerate}[\rm (1)]
\item in $U$, each of the three sets in the partition $\mathrm{Int}(K^+)\cup J^+ \cup \C^2\setminus K^+$ is  $\mathcal F$-saturated;
\item  $\mathcal F$ is $f$-invariant in the following sense:
 if $x\in J^+\cap U$ and some iterate $f^m$ of $f$ 
maps  $x$  to another point $x'=f^m(x)$ of $U$, then 
$f^m$ maps the leaf $\mathcal F(x)$ to the leaf $\mathcal F(x')$ (as germs of Riemann surfaces). 
\end{enumerate}
\end{pro}

\begin{proof}
Assume that $J^+$ is subordinate to $\mathcal F$ in $U$. 
Since $W^s(p)$ is dense in $J^+$,  it intersects $U$. 
Let $L\subset W^s(p)\cap U$ be a holomorphic disk. Then   $L$ is   contained in a leaf of $\mathcal F$, since otherwise
 its $\mathcal F$-saturation 
 \begin{equation}\label{eq:saturation}
\mathcal F(L)=\bigcup_{q\in L} \mathcal F (q)
\end{equation}
would have non-empty interior and be contained in $J^+$, a contradiction.
Conversely, assume that   $W^s(p)\cap U$ is  subordinate to $\mathcal F$. Let $q$ be a point of $J^+\cap U$ and let $U'$ be a flow box for $\mathcal F$ around $q$.  By density of $W^s(p)$,  the plaque $L$ through $q$ is approximated in the $C^1$ topology by plaques contained in $W^s(p)$. Since $J^+$ is closed, $L$ is contained in $J^+$, as asserted. 
 
For  Assertion (1), we observe that if a leaf $F$ in $U$ intersects $\mathrm{Int}(K^+)$ then it  cannot meet 
$\C^2\setminus K^+$, otherwise it would intersect $J^+$, and be entirely contained in it. For Assertion~(2), note that,  if  $f^m(\mathcal F(x))$ were not contained in a leaf, then its saturation, as in Equation~\eqref{eq:saturation}, would produce an open subset of $J^+$.   \end{proof}

%
 
 
 
%

 \begin{pro}  \label{pro:laminar_subordinate}
Assume that, in the connected open set $U$, $J^+$ is subordinate to a continuous foliation $\mathcal F$ by complex leaves. Then, 
 
\begin{enumerate}[\rm (1)]
\item  the current $T^+$ is uniformly laminar in $U$, the laminar structure of $T^+$ 
is subordinate to  $\mathcal F$, and $T^+{\rest U}$ is induced by a transverse measure on $\mathcal F$;
\item   if $\Gamma\subset U$ is a disk that is transverse to $\mathcal F$,  the transverse measure is given by 
$$T^+\wedge [\Gamma]=\Delta (G^+\rest{\Gamma});$$
\item if  $\Gamma'$ is another disk transverse to $\mathcal F$, {\emph{the $\mathcal F$-holonomy $h\colon \Gamma\to \Gamma'$ maps (on its domain of definition) the measure $T^+\wedge [\Gamma]$ to the measure $T^+\wedge [\Gamma']$}}. 
\end{enumerate}
\end{pro}

\begin{proof}  This is a version of~\cite[\S 8-9]{bls} in a simpler context.
Shrinking $U$, we may assume that $U$ is biholomorphic to $\disk^2$, with coordinates 
$(z,w)$, and  for  $U'\subset U$ intersecting $J^+$, every leaf of 
  $\mathcal F$ intersecting $U'$ is a  vertical disk in $U$. 
We already know that the current $T^+$ is  laminar (in a weak sense) and its support equals  $J^+$. Moreover, by construction (see~\cite{bls}), the  disks constituting its laminar structure  are limits of (pieces of) stable manifolds. Since these stable pieces are leaves of $\mathcal F$,  the asserted properties follow. 
\end{proof}



 \subsection{No invariant global holomorphic foliation} \label{subs:no_invariant_global_foliation}
In this subsection, we prove the holomorphic version of Theorem~\ref{mthm:no_real_global_foliation}, which is an extension of 
 Brunella's theorem~\cite[Corollary]{brunella} to arbitrary (transcendental) foliations.
 
 \begin{thm}\label{thm:no_invariant_global_foliation} 
A loxodromic automorphism $f\in\Aut(\C^2)$   does not preserve any global singular
 holomorphic  foliation. 
\end{thm}

We start with a lemma which is a variation on Proposition~\ref{pro:subordination}, Assertion~(2), when $\mathcal F$ is a  global holomorphic foliation.

\begin{lem}\label{lem:invariant_global_foliation}
 Let $\mathcal F$ be a global, possibly singular, holomorphic foliation of $\C^2$. If $J^+$ 
is  subordinate to $\mathcal F$ in some  open subset, then $\mathcal F$ is $f$-invariant. Conversely, if $\mathcal F$ is $f$-invariant, then $J^+$ or $J^-$ is subordinate to $\mathcal F$. \end{lem}

\begin{proof} 
Since there are infinitely many saddle periodic points in the compact set $\jstar$ and the singularities of $\mathcal F$ are isolated, 
there is  a saddle periodic point $p$ that is not a singularity of $\mathcal F$. Let $n$ be its period.
 As observed in Proposition~\ref{pro:subordination}, the stable manifold $W^s(p)$ is a leaf of $\mathcal F$, and so is the stable manifold $W^s({f(p)})$. Thus,  $W^s(p)$ is a common leaf of the foliations $f^*\mathcal F$ and $\mathcal F$. Now, observe that  $W^s(p)$ is  dense for the complex analytic Zariski topology:
 this follows   from the fact that $J^+ = \overline{W^s(p)}$ carries a unique closed positive current which gives no mass to curves (see~\cite{fornaess-sibony}). 
 So we deduce that 
 $f^*\mathcal F=\mathcal F$. 
  
  For the converse, since $\mathcal F$ is $f$-invariant, $\mathcal F(p)$ is $f^n$-invariant. But then, $\mathcal F(p)$ must coincide with $W^s(p)$ or $W^u(p)$ near $p$, and by density of $W^s(p)$ 
in $J^+$ (resp.\ of $W^u(p)$ in $J^-$), we conclude that $J^+$ or  $J^-$ is subordinate to $\mathcal F$.
\end{proof}
 
\begin{rem}\label{rem:invariant_real_foliation}
In Theorem~\ref{thm:dim3} below, we prove that $J^+$ is   dense for the real-analytic topology. It 
follows that the previous result holds for real-analytic foliations with complex leaves (and possibly  
a locally finite singular set).
\end{rem}

The key step in the proof of Theorem~\ref{thm:no_invariant_global_foliation} 
 is to compare the alleged invariant foliation with the  canonical 
$f$-invariant holomorphic foliation $\mathcal F^+$ 
 of $\C^2\setminus K^+$, given near infinity by the level sets of the Böttcher function 
 $\varphi^+$ (see~\cite[\S 5]{HOV1}). Recall that $\varphi^+$ 
 is defined near infinity in $\C^2\setminus K^+$
 by 
 \begin{equation}
 \varphi^+  = \lim_{n\to\infty} \lrpar{f^n}^{1/d^n}
 \end{equation}
 for some appropriate branch of the $(d^n)^{\rm th}$ root; then, by construction, 
 $G^+ = \log\abs{\varphi^+}$ there.

%

\begin{lem}\label{lem:foliation_G+}
Let $f\in\Aut(\C^2)$ be a loxodromic automorphism.  
If $J^+$ is subordinate to a holomorphic foliation $\mathcal F$ 
in some   connected 
open set $U$ then  
$\mathcal F$ coincides with $\mathcal F^+$ in $U\setminus K^+$.
\end{lem}

\begin{proof} 
In $\C^2\setminus K^+$, $G^+$ is positive and 
pluriharmonic, and each  level set $\set{G^+ = c}$, $c>0$, admits a unique foliation by Riemann surfaces, 
given by $\mathcal F^+$; more precisely, if $x$ belongs to 
 $\set{G^+ = c}$ then the local leaf of $\mathcal F^+$ through $x$ is the 
unique local Riemann surface in  $\set{G^+ = c}$ containing~$x$. 

By analytic continuation 
it is enough to prove the result locally near $J^+$. 
Thus, shrinking $U$ as in Proposition~\ref{pro:laminar_subordinate}
we may assume  that $U\simeq \disk^2$, with coordinates $(z,w)$,  
the leaves of $\mathcal F$ are the disks $\set{z=\cst}$, and $\Gamma$ is a horizontal disk. 
 The current $T^+$ is uniformly laminar in $U$, and its transverse measure is given by $\Delta(G^+\rest{\Gamma'})$ for any horizontal disk $\Gamma'$ of $\disk^2$. Fix 
such a disk $\Gamma'$ and let 
$h:\Gamma\to \Gamma'$ be the holonomy along $\mathcal F$. Then
$h_\varstar (T^+\wedge [\Gamma]) = T^+\wedge [\Gamma']$ and $h$ is holomorphic, thus 
$H_{\Gamma'}:=G^+\rest{\Gamma} - G^+\rest{\Gamma'}\circ h$ is a harmonic function on $\Gamma$. Since $J^+$ is subordinate to $\mathcal F$,  
$H_{\Gamma'}$ vanishes on the infinite set $J^+\cap \Gamma$, so either $H_{\Gamma'}$ vanishes identically on the disk $\Gamma$ or  $\{H_{\Gamma'}=0\}$ is a real analytic curve  containing $J^+\cap \Gamma$. 

 If  $H_{\Gamma'}$ vanishes identically for all $\Gamma'$ (or equivalently, for a dense set of horizontals $\Gamma'$), then $G^+$ is invariant under holonomy,
    $\mathcal F$ coincides with $\mathcal F^+$, and   we are done. 
    
Thus, we can assume that $\{H_{\Gamma'}=0\}$  is a curve  for a non-empty and open set of horizontal disks $\Gamma'$. 
The intersection of these curves as $\Gamma'$ varies determines a real analytic curve $C$ in $\Gamma$ (it can not be reduced to a finite set because $J^+\cap \Gamma$ is infinite). 
If this curve contains a point $x$ where $G^+(x)>0$, then by definition of $C$, 
the level set $\{ G^+=G^+(x)\}$ contains the leaf of $\mathcal F$ through $x$, so
 this leaf must coincide with the leaf of $\mathcal F^+$. 
 Changing $x$ in nearby points $x'\in C$ we see that $\mathcal F$ and $\mathcal F^+$ have infinitely many leaves in common. Thus, 
  $\mathcal F$ coincides locally with $\mathcal F^+$ in this case. 
Otherwise, $G^+\equiv 0$ on the curve $C$
 and we must have  $J^+_\Gamma= 
J^+\cap \Gamma = K^+\cap \Gamma = C$.
Indeed since $\mathcal F$ preserves $\mathrm{Int}(K^+)$ (Proposition~\ref{pro:subordination}),     
$\Gamma\cap \mathrm{Int}(K^+)=\emptyset$, for otherwise $C$ would contain an open set. So $J^+\cap \Gamma = K^+\cap \Gamma$. Then, since $J^+$ is locally foliated, $J^+\cap \Gamma$ 
must have empty interior, hence $J^+_\Gamma= 
J^+\cap \Gamma$  (see \S~\ref{subs:transversals}). 
Thus, $f$ is unstably linear, which is impossible by Theorem~\ref{thm:no_unstably_linear}. 
 \end{proof}


\begin{proof}[Proof of Theorem~\ref{thm:no_invariant_global_foliation}]
Denote by $\mathcal F$  such a hypothetical invariant foliation. 
By Lemma~\ref{lem:invariant_global_foliation}, replacing $f$ by $f\inv$ if needed, 
$J^+$   is subordinate to $\mathcal F$. Hence by 
Lemma~\ref{lem:foliation_G+}, 
$\mathcal F$ coincides with $\mathcal F^+$ outside $K^+$. 
On the other hand, by~\cite[Lem. 3.5]{connex}, $\mathcal F^+$   extends to a singular foliation of $\P^2$ by adding the line at infinity as a leaf. Thus, $\mathcal F$ extends to a global singular
holomorphic   foliation of $\P^2$. Such a foliation is automatically defined by a global algebraic $1$-form, hence  this contradicts Brunella's theorem and we are done.
 \end{proof}

\subsection{Local subordination}\label{subs:local_subordination}
 For the question of {\emph{ local subordination}}  of $J^\pm$ to a holomorphic
foliation, we only have partial answers. 

\begin{pro}\label{pro:local_subordination} 
Let $f$ be a loxodromic automorphism  such  that $J^+$ is subordinate 
to a holomorphic foliation in some open set $U$. 
Then $f$ is unstably connected. In particular $J$   is connected and $\abs{\jac(f)}\leq 1$. 
\end{pro}

 
\begin{cor}
If $\abs{\jac(f)}< 1$, $J^-$ is not locally subordinate to a holomorphic foliation. 
\end{cor}

\begin{proof}[Proof of the corollary]
Indeed, by~\cite{bs6}, a dissipative automorphism is not stably connected. 
\end{proof}

\begin{proof}[Proof of Proposition~\ref{pro:local_subordination}] 
By   Lemma~\ref{lem:foliation_G+}, 
 $\mathcal F$ coincides with $\mathcal F^+$ outside $K^+$. 
Consider a flow box of $\mathcal F$ intersecting $J^+$, together with a 
holomorphic   transversal disk $\Gamma$ to this flow box. 

We claim that $G^+\rest{\Gamma}$  has no critical points in 
$\Gamma \cap (\C^2\setminus K^+)$. Indeed, 
for every $x\in\Gamma  \setminus K^+$, some iterate $f^n(x)$ belongs to the region $V^+$ where $\varphi^+$ 
is well-defined and defines a non-singular fibration. Since 
$\mathcal F^+$ is transverse to $\Gamma$, there is a neighborhood $N$ of $x$ in $\Gamma$ 
such that $\varphi^+\rest{f^n(N)}$ is univalent, hence $G^+\rest{f^n(N)}$ has no critical points, and our claim follows by invariance.

We now rely on the results of~\cite{bs6}.  Lemma~\ref{lem:unstably_connected} below implies that  $f$ is unstably connected. 
The remaining conclusions respectively  follow from~\cite[Thm 5.1]{bs6} and the fact that 
a volume expanding automorphism is never unstably connected, as follows from Theorem 7.3 and Corollary 7.4 in~\cite{bs6}.
 \end{proof}

\begin{lem}\label{lem:unstably_connected}
If $\Gamma$ is a transversal to $J^+$ 
such that $G^+\rest{\Gamma}$ has no critical points, then $f$ is 
unstably connected. 
\end{lem}

\begin{proof} 
This is a mild generalization of some results of~\cite{bs6}. 
Assume by way of contradiction that $f$ is unstably disconnected. 
Fix a transverse intersection point of 
$W^s(p)$ and $\Gamma$. Iterating $\Gamma$ forward and applying the inclination lemma and the continuity of $G^+$, as  in the proof of~\cite[Prop. 1.8]{connex}, we  
construct arbitrary many compact components of 
$K^+\cap \Gamma$. None of these components is isolated. Indeed, otherwise $f$ 
would be  transversely connected in the sense of~\cite[Def. 1.5]{connex} hence unstably connected, by~\cite[Prop. 1.8]{connex}. But then, we obtain a contradiction with 
Lemma~7.2 in~\cite{bs6} 
(\footnote{This lemma is stated for an unstable manifold, but the proof evidently works in any transversal.}). 
\end{proof}

\begin{pro}\label{pro:local_subordination_sink}
Let $f$ be a    loxodromic automorphism such that $\abs{\jac(f)}<1$ and 
$\mathrm{Int}(K^+)$ is either empty or a union of sink and parabolic basins. 
Then  $J^+$ cannot be  subordinate  to a holomorphic foliation in a neighborhood of $\jstar$. 
\end{pro}

\begin{proof}
The closing lemma of~\cite{closing} (see Corollary 1.2 there) shows that 
 the orbit of every point of $J^+$ accumulates $\jstar$.  So, by pulling back, $\mathcal F$ extends to a holomorphic foliation on a neighborhood of $J^+$. (Note that if $x\in J^+$ and $n_1, n_2$ are such that $f^{n_i}(x)$ are close 
 to $\jstar$, then the local foliations obtained by pulling back under $f^{n_1}$ and $f^{n_2}$ coincide, because they coincide on $\jstar$.)
Thus by Lemma~\ref{lem:foliation_G+}, one can extend 
 $\mathcal F$ holomorphically by $\mathcal{F}^+$ to $\C^2\setminus K^+$.
  
 Let now $\Omega$ be a component of $\mathrm{Int}(K^+)$. A first possibility is that $\Omega$ is  
  the basin of some attracting periodic point $a$. 
  Since the domain of definition of $\mathcal F$ contains a  
  neighborhood of $\fr\Omega$, by pulling back we see that 
$\mathcal F$ extends to an invariant holomorphic foliation of $\Omega\setminus\set{a}$, 
hence by  the Hartogs' extension theorem $\mathcal F$ extends to a singular holomorphic foliation of $\Omega$.(\footnote{More precisely, on a small neighborhood of $a$, the foliation is defined by a line field $z\mapsto L(z)$. Denoting by $\ell(z)$ the slope of $L(z)$, we obtain a meromorphic function $\ell$ on a small neighborhood of $a$. By Levi's extension theorem~\cite[p. 133]{Narasimhan}, this function extends to a meromorphic function on some neighborhood of $a$. }) If $\Omega$ is the basin of a semi-parabolic periodic point $q$, since $p\in \jstar$, $\mathcal F$ is well defined in some neighborhood of $q$, and again pulling back extends $\mathcal F$ to $\Omega$. Altogether, we have constructed an extension of $\mathcal F$ to $\C^2$, and we conclude by Theorem~\ref{thm:no_invariant_global_foliation}. 
%
\end{proof}
  
\begin{rem}\label{rem:hyperbolicity}
The only known examples for which $J^+$ is locally laminated 
by stable manifolds near $\jstar$  are hyperbolic maps or maps with a dominated splitting on $\jstar$. For hyperbolic maps the assumption on $\mathrm{Int}(K^+)$
 holds by~\cite{bs1}. For maps with a dominated splitting on $\jstar$ (with the additional hypothesis  that $\abs{\jac(f)}< 1/(\deg(f))^2$), 
it follows from the structure theorem of  Lyubich-Peters~\cite{lyubich-peters}. 
See~\cite[\S 2]{topological} for sufficient conditions for hyperbolicity based on the laminarity properties of $J^\pm$.  Thus it is very plausible that the   assumption that 
$\mathrm{Int}(K^+)$ is a union of sink and parabolic basins is superfluous  in Proposition~\ref{pro:local_subordination_sink}. \end{rem}
  
%
%
%

\section{Real-analytic foliations}\label{sec:real_analytic_foliations}

In this  section we study the situation where $J^+$ is  locally or globally 
subordinate to a real-analytic foliation with complex leaves. 
 
\subsection{Proof of Theorem~\ref{mthm:no_real_global_foliation}} 
Recall the statement of Theorem~\ref{mthm:no_real_global_foliation}:

\begin{thm}\label{thm:no_real_global_foliation}
Let $f$ be a  loxodromic automorphism of $\C^2$.
Then $J^+$ is not subordinate to a global real-analytic foliation (with a possibly locally finite singular set).  
\end{thm}
 
%

The proofs goes by showing that   $J^+$ is locally subordinate to 
a holomorphic foliation near $\jstar$, 
which  requires separate arguments  in the unstably real and 
non unstably real cases. 
Altogether, the theorem is a consequence of  Theorem~\ref{thm:no_invariant_global_foliation}, Corollary~\ref{cor:global_real_analytic_foliation1} and Proposition~\ref{pro:analytic_foliation2}. The argument also implies that    Proposition~\ref{pro:local_subordination_sink} holds in the real-analytic case.

\begin{pro}\label{pro:analytic_foliation1}
Let $f$ be a loxodromic automorphism of $\C^2$. Assume that
\begin{enumerate}[\rm(i)] 
\item $f$ is not unstably real;
\item in some open set $U$ intersecting $\jstar$, 
$J^+$ is subordinate to a real analytic foliation $\mathcal F$ 
with complex leaves. 
\end{enumerate}
Then $\mathcal F$ is holomorphic and
 $f$ is unstably connected.
\end{pro}

\begin{cor}\label{cor:global_real_analytic_foliation1}
If $f$ is a loxodromic automorphism which is not unstably real, then $J^+$ cannot be 
subordinate to a global real analytic foliation with complex leaves. 
\end{cor}

%

\begin{proof}[Proof of the corollary]
By Proposition~\ref{pro:analytic_foliation1}, $\mathcal F$ is holomorphic in a 
neighborhood of $\jstar$. This property propagates to $\C^2$ by real analyticity, and 
Theorem~\ref{thm:no_invariant_global_foliation} finishes the proof. 
\end{proof}

The proof of Proposition~\ref{pro:analytic_foliation1}
 relies on a variation on a well-known result of Ghys~\cite{Ghys:holomorphic_anosov}. 
The essence of Ghys' argument is contained in the following lemma.  

\begin{lem}\label{lem:ghys}
Let  $f$ be a loxodromic automorphism. Assume 
 that in some open set $U$ intersecting $\jstar$, $J^+$ is subordinate to an $f$-invariant
 $C^1$ foliation $\mathcal F$
 with complex leaves. Let $p$ be a saddle periodic point in $U$ 
 and $L$ be a leaf of $\mathcal F$ contained in 
 $W^s(p)$. Let $x$, $x'$ be two points of $L$, $\tau$ and $\tau'$ be local holomorphic transversals to 
 $L$, and $h_{\tau, \tau'}: (\tau, x)\to (\tau', x')$ the  germ of holonomy along~$\mathcal F$. Then 
 $h_{\tau, \tau'}$ is holomorphic at $x$. 
\end{lem}
 
Recall that the notion of invariance for a local foliation was explained in 
Proposition~\ref{pro:subordination}, Assertion~(2). 
  
\begin{proof}[Proof of the lemma] 
Let $k$ be the period of $p$. 
Fix a path $\gamma$  in $L\cap U$ joining $x$ and $x'$. 
  For large enough $n\in k\Z$, $f^n(\gamma)$ is contained in $U$ and close to $p$. By the inclination lemma, there are two disks $V_n\subset \tau$ and $V'_n\subset \tau'$, containing respectively $x$ and $x'$,
such that $f^n(V_n)$ and $f^n(V_n')$ are close to $W^u_{\loc}(p)$. Then the holonomy between $f^n(V_n)$ and $f^n(V_n')$ is a $C^1$ quasiconformal map whose distortion is 
bounded by $\e_n$, with $\e_n\to 0$. Since $f$, $\tau$ and $\tau'$ are   holomorphic, it follows that the quasiconformal distortion of $h_{\tau, \tau'}$  in $V_n$    is bounded by $\e_n$, so  by letting $n$ tend to infinity we conclude that the distortion at $x$ vanishes, 
that is, $h_{\tau, \tau'}$ is holomorphic  at $x$, as 
was to be shown. 
\end{proof}

\begin{proof}[Proof of Proposition~\ref{pro:analytic_foliation1}]
By Proposition~\ref{pro:subordination},
if $x$ and $x'$ belong to 
$J^+\cap U$ and 
$f^n$ maps $x$ to $x'$, then it must locally send $\mathcal {F}(x)$ to $\mathcal {F}(x')$. Let $\tau$ be some transversal to $\mathcal F$ at $x'$.
Since 
$f$ is not unstably real, Proposition~\ref{pro:contained_curve}, 
 shows that the local $\mathcal F$  saturation of $J^+_\tau$ is $\R$-Zariski dense.  Therefore we infer that 
$\mathcal F$ is invariant (this is a real analytic version of Lemma~\ref{lem:invariant_global_foliation}). 

To prove that $\mathcal F$ is holomorphic, we fix two holomorphic 
 transversals $\tau$, $\tau'$ in  some flow box of $\mathcal F$; we have to 
 show that the holonomy $h_{\tau, \tau'}$ between these transversals is holomorphic. 
Fix a saddle point $p\in \jstar\cap U$. Then $W^s(p)\cap U$ is dense in $J^+\cap U$, 
so $W^s(p)\cap \tau$ is dense in $J^+_\tau$.
By Lemma~\ref{lem:ghys}, $h_{\tau, \tau'}$ is holomorphic at every such point. Since in addition 
$J^+_\tau$ is $\R$-Zariski dense in $\tau$, and $h_{\tau, \tau'}$ is real analytic, this entails that $h_{\tau, \tau'}$ is holomorphic everywhere. Thus, $\mathcal F$ is holomorphic, and the unstable connectedness follows from Proposition~\ref{pro:local_subordination}. 
\end{proof}

In the unstably real case we readily get a contradiction. 

\begin{pro}\label{pro:analytic_foliation2}
If $f$ be a loxodromic automorphism which is  unstably real, 
then $J^+$ cannot be  subordinate to a real analytic foliation $\mathcal F$ 
with complex leaves in an open set $U$ intersecting $\jstar$. 
\end{pro}

\begin{proof}
Let $\mathcal F$ be the alleged foliation and 
$p$ be a saddle periodic point in $U$.  Since $f$ is unstably real, $J^+_{W^u(p)}$ 
is  contained in a real-analytic curve, therefore,   saturating by $\mathcal F$, 
we deduce that $J^+$ is contained in a real-analytic hypersurface $\Sigma$, which must be  Levi-flat. A classical theorem of 
Cartan~\cite{cartan} asserts that the Levi foliation of a real analytic Levi flat 
hypersurface locally extends to a holomorphic foliation. 
Thus, in some neighborhood $U$ of $p$, there is a holomorphic foliation 
$\mathcal F'$ extending  $\mathcal F\rest{\Sigma}$ (which may not coincide with $\mathcal F$), hence $J^+$ is subordinate to 
$\mathcal F'$ in $U$. 
Then,  Proposition~\ref{pro:local_subordination} implies that 
$f$ is unstably connected, hence unstably linear, which is impossible by Theorem~\ref{thm:no_unstably_linear}. 
\end{proof}

\subsection{An application}  
 The idea underlying Proposition~\ref{pro:analytic_foliation2} leads to
 the following result in the spirit of \cite{BK1}. 
 
 \begin{thm}\label{thm:dim3}
 Let $f\in \Aut(\C^2)$ be loxodromic. Then 
 \begin{enumerate}[\rm (1)]
 \item $J^+$ (resp. $J^-$) is not contained in a proper real-analytic  subvariety; 
 \item  the dimension of the $\R$-Zariski closure of $\jstar$ cannot be equal to $0$, $1$, or $3$.
 \end{enumerate} 
 \end{thm}

On the other hand, the $\R$-Zariski closure of $\jstar$ can be of dimension 2: this happens for  
 mappings with $\jstar \subset \R^2$ (or equivalently 
 $h_{\rm top}(f\rest{R^2})= h_{\rm top}(f)$) and their conjugates. Therefore the $\R$-Zariski closure of $\jstar$ has dimension 2 or 4; it is easy to show that if one saddle point  has a non-real stable or unstable eigenvalue, then it cannot be of dimension 2 
 (see below Propositions~\ref{pro:dense_set_of_saddle_point} and~\ref{pro:dense_set_of_saddle_point_2} for stronger results in this spirit). 
Thus, the second statement of the theorem is an argument in favor of the implication ``unstably real implies stably real'' (see \S~\ref{subs:questions_unstably_real}). 

\begin{proof}
Assume that $J^+$ is contained in a proper real-analytic subvariety $V$. Only finitely many components of $V$ intersect $\jstar$, so these components are periodic and one of them has positive (not necessarily full) $\mu_f$-measure\footnote{At this stage it could be that these components are strictly periodic, and intersect along a  2-dimensional subset 
  which must be fixed because $\mu_f$ is mixing.}. Let $V_1$ be such a component. Since $\mu_f$ is hyperbolic, 
  by Pesin theory, $V_1$ contains a flow box of
   positive measure $\mathcal L^s$ made of Pesin local stable manifolds. Since $\mathcal L^s$ contains infinitely  many holomorphic disks, it is $\R$-Zariski dense in $V_1$. At every point   $x\in \mathcal L^s$, since $V_1$ contains a holomorphic disk 
   through $x$, the Levi form of $V_1$ vanishes. 
 By Zariski density, we conclude that $V_1$ is Levi flat. In addition, 
 since $T_xW^s(x)\oplus T_xW^u(x) = T_x\C^2$ and $V_1$ is smooth at $x$
  the intersection of $V_1$ with a local unstable manifold $W^u_\loc(x)$ 
 is contained in an analytic curve, therefore 
$f$ must be unstably real.  
Thus,  Proposition~\ref{pro:analytic_foliation2} yields the desired contradiction. 
 
 Now assume that the $\R$-Zariski closure $V$ of $\jstar$ has dimension 3, and fix a component $V_1$ of $V$ of dimension 3.
 By Zariski density, there is a saddle point $p$ together with an open subset $U\ni p$ such that 
 $\jstar \cap U\subset \mathrm{Reg}(V_1)$. By the local 
  $f^n$-invariance of $V_1$ at $p$, we see that either $W^s_\loc(p)$ or $W^u_\loc(p)$ is contained in $V_1$: indeed, if $T_pV_1$ does not contain the stable direction $E^s(p)$ (resp. the unstable direction $E^u(p)$), the Inclination Lemma shows that $V_1$ contains $W^u_{\loc}(p)$ (resp. $W^s_{\loc}(p)$). In this way we can find an $\R$-Zariski dense subset of $V_1$ made of local unstable (resp. stable) manifolds $W^u_\loc(p) \subset V_1$, and we complete the argument  similarly as in the previous case.  
  
To conclude, note that $\jstar$ is infinite, so the dimension of its $\R$-Zariski closure is $\geq 1$. If its Zariski closure were a curve, a component of it would be contained in the stable (or unstable) manifold of a saddle periodic point. But then, this stable manifold would contain more than one periodic point, a contradiction.
\end{proof}

\section{Homoclinic orbits and the arguments of periodic point multipliers}\label{sec:homoclinic}

In this section we extend to two dimensions a technical result which has recently proven useful in certain multiplier rigidity results in dimension 1 (see~\cite{eremenko-vanstrien, ji-xie:multipliers, huguin}).

\begin{thm}\label{thm:homoclinic_multipliers}
Let $f$ be a holomorphic diffeomorphism on a  complex surface. 
Assume that  $p$ is a saddle fixed   point  
  admitting
a homoclinic intersection. 
Then   there exists 
a sequence of saddle periodic points $p_n$  
of period $n$ accumulating $p$, whose  stable and unstable multipliers satisfy 
$\lambda^s(p_n)\sim c  (\lambda^s(p))^{n}$ and 
$\lambda^u(p_n)\sim c'  (\lambda^u(p))^{n}$ for some $c$, $c'\in \C^*$. 
Moreover, the sequence $(p_n)$ can be chosen to be Zariski dense for the complex analytic topology
\end{thm}

Recall that for a  loxodromic automorphism of $\C^2$, every saddle periodic point generates (transverse) 
homoclinic intersections (see~\cite[Thm. 9.6]{bls}).
 

\begin{proof}~

{\bf{Step 1.-- Horseshoes.}} By~\cite[Prop. 3.2]{degenerate}, there is a  point $\tau$ of transverse intersection between $W^u(p)$ and $W^s(p)$. If one 
adds $\set{p}$ to the orbit of such a point one gets a compact hyperbolic set $\mathcal O(\tau)\cup\set{p}$; and for any small enough 
neighborhood $V'$ of 
$\mathcal O(\tau)\cup\set{p}$, there exists  a smaller neighborhood 
$V$ such that  the invariant set 
\begin{equation}
\Lambda = \bigcap_{n\in \Z} f^{-n}(V)
\end{equation} 
is a 
hyperbolic set with positive entropy (see e.g.~\cite[\S 7.4.2]{robinson}), 
which is locally maximal by construction. We shall refer to it as {\emph{a horseshoe $\Lambda$ (adapted to $\tau$)}}.
Then for sufficiently large positive $M$, $M'$,
$\big\{f^{-M}(\tau), \ldots , f^{M'}(\tau)\big\}$ is an $\e$-pseudo-orbit, with $\e = O(\abs{\lambda^u}^{-M}+ \abs{\lambda^s}^{M'})$,  
so by the Closing Lemma it 
is shadowed by a unique periodic orbit of length $M+M'+1$ contained in $\Lambda$.

{\bf{Step 2.-- Shadowing.}} 
As in the proof of Proposition~\ref{pro:contained_curve}, we consider local holomorphic coordinates 
$(x,y)$ near $p$ in which $p= (0,0)$ and $f$ is of the form 
 \begin{equation}\tag{\ref{eq:saddle_normal_form}}
 f(x,y) =\left( \lambda^u x(1+xy g_1(x,y)), \lambda^s y(1+xy g_2(x,y)) \right).
 \end{equation}
In these coordinates, the axes coincide   with the local unstable and stable manifolds. Renormalizing if needed, we assume that these coordinates are well defined over the unit bidisk~$\disk^2$. 

From Step~1, there is a point 
$\tau_0 = (x_0, 0)\in W^u_\loc(p)$, with $x_0\neq 0$, 
a neighborhood $U$ of $x_0$ in $\C$, and an integer $k\geq 1$ such that the piece of unstable manifold $f^k( {U}\times\set{0})$
\begin{itemize}
\item   intersects $W^s_{\loc}(p) = \set{x=0}$ transversally at $f^k(\tau_0) = (0, y_0)\in \disk^2$, with $y_0\neq 0$, and
\item  is a graph (relative to the first coordinate) over some neighborhood of the origin.
\end{itemize} 
  If we fix a small
 neighborhood $V''$ of $\set{f(\tau_0), \ldots , f^{k-1}(\tau_0)}$ and put $V'  = \disk^2\cup V''$, the first step shows that, for every  large enough $N$, there is a unique saddle periodic orbit,   of length $n:=2N+k$,  shadowing 
\begin{equation}
f^{-N}(\tau_0), \ldots , \tau_0, \ldots, f^k(\tau_0), \ldots ,f^{k+N-1}(\tau_0).
\end{equation} 
We label this periodic orbit by $q_{-N}$, $\ldots$, $q_0$, $\ldots$,  $q_{N-1}$, $\ldots$, $q_{N+k-1}$ in such a way that $q_{-N}$, $\ldots$,   $q_{N}$ are in $\disk^2$, 
 $q_{-N}$ is close to $f^k(\tau_0)$, $q_0$ is close to the origin, and $q_{N}$ is close to $\tau_0$.
 The period of $q_0$ is equal to $2N+k$. 
 
At the end of the proof, for a given large $n$, if $n$ of the form $2N+k$ (that is $n \equiv k~\mathrm{mod.}~2$)   
we shall set, $p_n=q_0$ where $q_0$ is as above; otherwise,  
for $n$ of the form $2N+k+1$, the reader can reproduce the same proof 
with the pseudo-orbit $f^{-N+1}(\tau_0), \ldots, \tau_0, \ldots, f^{k+N-1}(\tau_0)$.
 
{\bf{Step 3.-- Estimating the multiplier.}}  In the following we   study the action of $f$ on tangent vectors in $\disk^2$ which are close to the horizontal direction. 
 These vectors will be  written as $v= \nu (1,m)$, where $\nu\in \C^*$ 
and $m\in \C$ is  the complex slope. In particular, 
\begin{equation}
df^k_{\tau_0}(1,0) = \nu_0(1, m)
\end{equation}
for some $\nu_0\in \C^*$, where $(1,m)$ is tangent to 
$f^k( {U}\times\set{0})$ at $f^k(\tau_0)$. In other words, if we write $f^k=((f^k)_1, (f^k)_2)$, $\nu_0$ is given by the partial derivative\footnote{This quantity  is intrinsic in the sense that it takes the same value in any system of  coordinates $(x,y)$ in which 
  $f$ is of the form~\eqref{eq:saddle_normal_form}. This can be shown by a direct calculation, but it also follows a posteriori from Lemma~\ref{lem:homoclinic_multipliers_precised}.}
$$\nu_0 = \frac{\partial (f^k)_1}{\partial x} (x_0, 0).$$
 We will complete the proof of the theorem by establishing the following
  more precise lemma. The estimate for $\lambda^s(q_0)$ is obtained in the same way.

\begin{lem}\label{lem:homoclinic_multipliers_precised}
The unstable multiplier of $q_0$ satisfies $$\lambda^u(q_0) = \nu_0 (\lambda^u(p))^{2N} (1+O(\theta^N))$$ for some $\theta<1$.  
\end{lem}

Note that if instead of the pair 
 $(\tau_0 ,f^{k}(\tau_0))$ we choose $(f^{\ell}(\tau_0), f^{k+\ell'}(\tau_0))$, then, 
 as soon as these points remain in $\disk^2$,  the same result holds with $\nu_0$   replaced by 
$\nu_0 (\lambda^u(p))^{\ell' - \ell}$ and $(\lambda^u(p))^{2N}$ replaced by $(\lambda^u(p))^{2N+\ell - \ell'}$. 

\begin{proof}  
For $-N\leq j\leq N$, we  single out  the unit vector $e^u(q_j) = (1, m_j)$ as a complex tangent vector to the unstable direction of $q_j$, and write 
\begin{equation}
df_{q_j} (e^u(q_j)) = \mu_j e^u(q_{j+1}) \;\, {\text{  and }} \;\, df^k(e^u(q_{N})) = \nu e^u(q_{-N}).
\end{equation} These quantities depend on $N$ but we don't mark  this dependence for notational simplicity.  
With these conventions, the unstable multiplier at $q_0$ is 
\begin{equation}
\lambda^u(q_0)=  \left(\prod_{j=-N}^{N-1} \mu_j \right)  \times \nu.
\end{equation} 

In terms of partial derivatives, we have 
\begin{equation}
\mu_j  = \frac{\partial f_1}{\partial x} (q_j)+ m_j \frac{\partial f_1}{\partial y} (q_j).
\end{equation}
 From Equation~\eqref{eq:saddle_normal_form} we get 
\begin{equation}
\frac{\partial f_1}{\partial x} (x,y)  = \lambda^u(p)\lrpar{1+xy h_1(x,y)} \;  \text{ and  } \; 
\frac{\partial f_1}{\partial y} (x,y) = \lambda^u(p) x^2 h_2(x,y),
\end{equation}
where $h_1$ and $h_2$ are holomorphic. Thus, 
\begin{equation}
\mu_j  = \lambda^u(p)\lrpar{1+ x_jy_j h_1(x_j, y_j) + x_j^2 m_j h_2(x_j, y_j)}, \text{ where } q_j= (x_j, y_j).
\end{equation}
Since $q_0$ admits $N-1$ forward iterates and $N$ backward iterates 
in $\disk^2$, we obtain the estimates 
\begin{equation}
\abs{x_j} = O(\theta^{N-j})\; {\text{ and }} \;  
\abs{y_j} = O(\theta^{N+j})
\end{equation}
for $-N\leq j\leq N $ and some $\theta<1$ (any $\theta>\max (\abs{\lambda^s(p)} , \abs{(\lambda^u(p))\inv})$ will do); in particular $\abs{x_jy_j} = O(\theta^{2N})$. When $j\leq 0$ we also have 
$\abs{x_j^2 m_j} =  O(\theta^{2N})$.

It is known that on a compact hyperbolic set, the stable and unstable distributions are Hölder  continuous (see~\cite[Thm 19.1.6]{KH}).
 Applying this result to the horseshoe $\Lambda$, we deduce that 
  when $j\geq 0$ the slope $m_j$ of the unstable direction at $q_j$ satisfies $\abs{m_j}=O(\abs{y_j}^\alpha) $ for some $0< \alpha \leq 1$ 
and from this we deduce that  
$\abs{x_j^2 m_j} =  O(\theta^{\alpha N})$. Finally, we obtain 
\begin{equation}
\frac{ \prod_{j=-N}^{N-1} \mu_j}{(\lambda^u(p))^{2N}}  = \prod_{j=-N}^{N-1}(1+   O(\theta^{\alpha N}))  = 1+O(N\theta^{\alpha N}).
\end{equation}

We conclude the proof by showing that $\nu$ is exponentially close to $\nu_0$. Indeed, the precise form of the Anosov Closing Lemma given in~\cite[Cor. 6.4.17]{KH} shows that $q_N$ is $O(\theta^N)$ close to $\tau_0$. Since   the slope $m_N$ satisfies $\abs{m_N} = O(\theta^{\alpha N})$, 
we get 
\begin{equation}
\norm{df^k_{\tau_0}(1, 0 )  - df^k_{q_N}(1, m_N)} = O(\theta^{\alpha N}),
\end{equation} 
hence  
$\abs{\nu - \nu_0} = O (\theta^{\alpha N})$, and, replacing $\theta$ by some $\theta'>\theta^\alpha$
completes the proof. 
\end{proof}

{\bf{Step 4.-- Density.}} As said above, we   choose $p_n=q_0$ as above for each $n=2N+k$ (resp.\ $2N+k-1$). 
It remains to show that the points $p_n$ can be chosen to be Zariski dense for the complex analytic topology. 

Indeed, replacing $p_n$ (i.e.\ $q_0$ above) by an appropriately chosen point in its orbit (i.e.\ another $q_m$), we can arrange that the (Euclidean) 
closure of $(p_n)$ contains the whole negative orbit of $\tau_0$. 
Now assume for the sake of contradiction 
that  $(p_n)$ is  contained in a complex 
 analytic subvariety, which is decomposed into irreducible components as $V_1\cup \cdots \cup V_k$. One of the $V_i$ accumulates infinitely many points in $\set{f^{-j}(\tau), \ j\geq 0}$, hence it contains them; therefore it must locally
coincide with $W^u_{\loc}(p)$. 
But since the $p_n$ are periodic, 
they cannot be contained in an unstable manifold, which yields the desired  contradiction.
\end{proof}

\begin{rem}
Since  in our situation, the stable and unstable  distributions of the horseshoe
 are complex 1-dimensional, the bunching condition of~\cite[Thm 19.1.8]{KH} holds, so they are actually Lipschitz, and we could have chosen $\alpha=1$ in the above estimates. This shows that in Lemma~\ref{lem:homoclinic_multipliers_precised},
 any $\theta>\max(\abs{\lambda^s(p)}, \abs{(\lambda^u(p))\inv})$ is convenient. 
\end{rem}

 \begin{rem}\label{rem:zariski} 
If the complex analytic closure of $P:=\set{p_n\; ;\; n\geq 1}$ were a curve, it would be an 
$f$-invariant, holomorphic curve.
But a holomorphic diffeomorphism of a curve of infinite order has at most finitely many periodic points, so only finitely many of the $p_n$ could be on that curve, a contradiction.
This gives a simple alternate proof that $P$ is dense for the complex analytic topology. Step 4, above, could also be proven along similar lines. 

 \end{rem}



The following result will be used in the proof of Theorem~\ref{mthm:conjugate_real}.

\begin{pro}\label{pro:dense_set_of_saddle_point}
Let $f$ be a loxodromic automorphism of $\C^2$. If $f$ has   a saddle periodic point $p$ 
with non-real stable and  unstable multipliers, then there is a set of such periodic points which is dense for the real analytic topology. 
\end{pro}

\begin{proof} Without loss of generality, assume that $p$ is fixed.
By~\cite{bls}, every saddle point admits homoclinic intersections; furthermore,  
$J^+_{W^u(p)}$ is   the closure of homoclinic intersections (see~\cite[Lem. 5.1]{tangencies}).
Since $\lambda^u(p)\notin\R$, $f$ is not unstably real, so $J^+_{W^u(p)}$ is not contained in a $C^1$ curve. Likewise $J^-_{W^s(p)}$ is not contained in a $C^1$ curve. 
For any transverse homoclinic intersection point $\tau\in W^u(p)\cap W^s(p)$, by Theorem~\ref{thm:homoclinic_multipliers}
we can construct a sequence 
  of saddle periodic points $(p_n)$ such that 
   the closure of $(p_n)$ contains   $\tau$ and  $\lambda^u(p_n)\sim c (\lambda^u(p))^n$ and $\lambda^s(p_n)\sim c' (\lambda^s(p))^n$. Thus,  we have 
\begin{enumerate}[-]
\item[-] if  $\arg(\lambda^u(p))$ or $\arg(\lambda^s(p))$ differs from $\pm \frac{\pi}{2}$, then  both $\lambda^u(p_n)$ and $\lambda^s(p_n)$ are non-real for infinitely many $n$;
\item[-] otherwise, $\arg(\lambda^u(p))=\e \frac{\pi}{2}$ and $\arg(\lambda^s(p))=\e' \frac{\pi}{2}$ with $\e$ and $\e'$ in $\set{\pm 1}$; in particular, $\jac(f)\in \R$. Then it might a priori occur that  $\lambda^u(p_n)$ and $\lambda^s(p_n)$ are alternatively real, depending on the parity of $n$, so that exactly one of them is always real; but such a case would imply $\jac(f)\notin \R$, a contradiction.  \end{enumerate}
Altogether, we infer that 
 $\lambda^u(p_n)\notin\R$ and  $\lambda^s(p_n)\notin \R$ for all but at most a proper periodic subsequence of indices $n$.
Repeating this construction  by varying $\tau$  in  $W^u_{\loc}(p)$ and $W^s_{\loc}(p)$, we obtain a set of periodic points $p_n$ with non-real multipliers  that accumulates $W^u_{\loc}(p)$ and $W^s_{\loc}(p)$ on subsets which are not contained in $C^1$ curves. 

To prove $\R$-Zariski density, we apply this construction to a dense subset of homoclinic intersections $\tau_j$ in $J^+_{W^u(p)}$, getting for each $j$ a sequence $(p_n^j)$, 
 and we denote by $(q_n)$ a sequence of saddle periodic points obtained from the sequences 
 $(p_n^j)$   by a diagonal process, which accumulates all the $\tau_j$.  
 
Assume first that $V$ is a proper  irreducible real-analytic subset 
containing the $q_n$. Then $V$
contains  a set of homoclinic points $\tau\in W^u_{\loc}(p)$ (resp.  $W^s_{\loc}(p)$) 
which is not contained in a $C^1$ curve, because $J^+\cap W^u_{\loc}(p)$ (resp. 
$J^-\cap W^s_{\loc}(p)$)  is not contained in such a curve. 
As in Remark~\ref{rem:zariski}, since none of the $q_n$ belongs to 
$W^u_{\loc}(p)\cup W^s_{\loc}(p)$  we see that   $V$ properly contains $W^u_{\loc}(p)\cup W^s_{\loc}(p)$. In particular, $\dim(V)= 3$
and since  $T_pW^s(p)\oplus T_pW^u(p) = T_p\C^2$, $V$ must be  singular at $p$.

To conclude, we can apply the above construction at each of the $q_n$ instead of $p$, getting  periodic points $q_{n,m}$ with non-real multipliers accumulating $W^u_{\loc}(q_n)$ and $W^s_{\loc}(q_n)$ along subsets which are not contained in $C^1$ curves. 
Let $V = V_1\cup \cdots \cup V_k$ be a proper real-analytic subset containing these points. 
Then  for one of the irreducible components, say $V_1$,  there exists an 
infinite family of points $q_n$, which is not contained in a $C^1$ curve in both 
$W^u_{\loc}(p)$  and $W^s_{\loc}(p)$, and such that at each $q_n$, $V_1$ contains 
a set of homoclinic points in $W^u_{\loc}(q_n) \cup W^s_{\loc}(q_n)$ which is not contained in a 
$C^1$ curve in both $W^u_{\loc}(q_n) $ and $ W^s_{\loc}(q_n)$. Therefore $V_1$ has dimension 3 and is singular at each such $q_n$. 
But the $q_n$ are not contained in an analytic subset of 
dimension $\leq 2$, so we get a contradiction. 
\end{proof}

Using   Theorem~\ref{thm:dim3} 
 we get a version of Proposition~\ref{pro:dense_set_of_saddle_point} involving 
  unstable multipliers only. 

\begin{pro}\label{pro:dense_set_of_saddle_point_2}
Let $f$ be a loxodromic automorphism of $\C^2$.
 If $f$ has   a saddle periodic point $p$ 
with non-real  unstable multiplier, then there is a $\R$-Zariski dense 
set of  periodic points with the same property.   
\end{pro}

\begin{proof}  
Arguing   as in the previous proposition (and without taking stable 
eigenvalues into account)
 we can construct a 
sequence of saddle points (analogous to the points $q_{n,m}$ from the previous proof), whose 
unstable multipliers are non-real and whose $\R$-Zariski closure  has dimension at least 3. 
Let $V$ be the $\R$-Zariski closure of the set of periodic points whose unstable multiplier is non-real, and assume that $\dim_{\R}(V)= 3$. Note that by construction $V$ contains $W^u(p)$. 
The subvariety $V$ is periodic, thus we can find an integer $k\geq 1$ such that  $f^k(V)=V$ and $f^k(p)=p$. By Theorem~\ref{thm:dim3}, $\jstar$ is not contained in $V$ so there is a saddle periodic point $q\in \jstar\setminus V$. Fix $t\in W^s(q)\cap W^u(p)$. Since $t\in W^u(p)$, $t$ is in $V$, and since $t\in W^s(q)$, $f^{kn}(t)$ converges toward $q$ as $n$ goes to $+\infty$, hence $q\in V$ (because $f^k(V)=V$), a contradiction.  
\end{proof}

%
%

\section{Real analytic conjugacies}\label{sec:conjugate_real}
 
 In this section, we prove Theorem~\ref{mthm:conjugate_real} in  the following  slightly more precise  form. 

\begin{thm}\label{thm:conjugate_real}
Suppose $\varphi\colon \C^2\to \C^2$ is a real analytic conjugacy between two loxodromic automorphisms $f$, $g\in \Aut(\C^2)$. Assume that there exists a saddle periodic point of $f$ at which both stable and unstable multipliers are not real. 
Then $\varphi$ is a polynomial  automorphism of $\C^2$, or the composition of an automorphism with the complex conjugation $(x,y)\mapsto (\overline{x}, \overline{y})$.
\end{thm}

 
Before starting the proof, let us make a  few  preliminary observations, which do not depend 
on the assumption that there is a saddle point with non-real eigenvalues.
If $p$ is a saddle fixed point of $f$ then $q=\varphi(p)$ is a saddle fixed point of $g$. Let $\psi_f\colon \C\to \C^2$ be a
parametrization of the stable manifold of $f$ at $p$, and let $\psi_g\colon \C\to \C^2$ be a parametrization of the stable manifold 
of $g$ at $q$. Then $\varphi\circ \psi_f=\psi_g\circ \widetilde \varphi$ for 
some real analytic diffeomorphism $\widetilde \varphi\colon \C\to \C$, $\zeta\mapsto \widetilde\varphi(\zeta)$, with $\widetilde \varphi(0) = 0$. 
Note that 
$\widetilde \varphi$ is completely determined by its restriction to a small neighborhood of the origin. 
Denote by $\lambda^s_p$ (resp.\  $\lambda^s_q$) 
the eigenvalue of $df_p$ (resp.\  $dg_q$) along the stable direction. 

\begin{lem} 
Suppose $\varphi\colon \C^2\to \C^2$ is a real analytic conjugacy between two loxodromic automorphisms $f$, $g\in \Aut(\C^2)$.
With the above notation, we have that:
\begin{enumerate}[(1)]
\item $\widetilde \varphi$ is $\R$-linear;
\item either $\lambda^s_q=\lambda^s_p$ or $\lambda^s_q=\overline{\lambda^s_p}$, and  in the latter case,   $\widetilde \varphi$ reverses the orientation of $\C$;
\item if $\lambda^s_p$ (or equivalently $\lambda^s_q$) is not real, then either $\widetilde \varphi(\zeta)=\alpha \zeta$ and $\lambda^s_q=\lambda^s_p$ or 
$\widetilde \varphi(\zeta) =\alpha \overline{\zeta}$ and $\lambda^s_q=\overline{\lambda^s_p}$, where $\alpha$ is some nonzero complex number. 
\end{enumerate}
\end{lem}

\begin{proof}
To see this, it suffices to write the Taylor expansion of $\widetilde \varphi(\zeta)$ near the origin as $\sum_{k,\ell} a_{k,\ell}\zeta^k{\overline{\zeta}}^\ell$ and write that 
$\widetilde \varphi$ conjugates $\zeta\mapsto \lambda^s_p \zeta$ to $\zeta\mapsto \lambda^s_q \zeta$, which gives 
a relation of the form
\begin{equation}
\sum_{k,\ell = 1}^\infty a_{k,\ell} (\lambda^s_p)^k \lrpar{\overline{\lambda^s_p}}^\ell \zeta^k{\overline{\zeta}}^\ell = 
\lambda^s_q \sum_{k,\ell = 1}^\infty a_{k,\ell}  \zeta^k{\overline{\zeta}}^\ell.
\end{equation}
Equating the coefficients we first infer that 
\begin{equation}
a_{1, 0}\lrpar{ \lambda^s_p  - \lambda^s_q} = a_{0, 1} \lrpar{\overline{ \lambda^s_p}  - \lambda^s_q} =0;
\end{equation}
 hence, since $\widetilde \varphi$ is a diffeormorphism, either 
$\lambda^s_p  = \lambda^s_q$ or $\overline{ \lambda^s_p}  = \lambda^s_q$ (or both). And using that $\abs{\lambda^s_p} = \abs{\lambda^s_q}<1$, we obtain $a_{k,\ell} = 0$ for $k+\ell>1$.

Finally, if $\lambda^s_p$ is not real, then the $\R$-linear map $\varphi$ either commutes to $\zeta\mapsto \lambda^s_p \zeta$
and it is then $\C$-linear or it conjugates $\zeta\mapsto \lambda^s_p \zeta$ to $\zeta\mapsto {\overline{\lambda^s_p}} \zeta$ and it is then  the composition of a $\C$-linear map with $\zeta\mapsto \overline \zeta$. This proves Assertion~(3).
In other words, $\varphi$ is holomorphic of antiholomorphic along the stable manifold $W^s_p$. 
\end{proof}

\begin{proof}[Proof of Theorem~\ref{thm:conjugate_real}]
Since  $f$ has a saddle periodic point $p$ with non-real stable and unstable multipliers, Proposition~\ref{pro:dense_set_of_saddle_point}  provides an $\R$-Zariski dense set of such points. Therefore, from the previous lemma, we deduce that:
either (a) $\varphi$ is holomorphic, or (b) $\varphi$ is anti-holomorphic (since in these cases the differential will be respectively $\C$-linear or $\C$-antilinear at on a $\R$-Zariski dense set) , or (c) there is a dense set of saddle periodic points $p_j$ for the real analytic topology such that $\varphi$ is holomorphic along  $W^u_{p_i}$ and antiholomorphic along $W^s_{p_i}$   or vice versa. 

In case (a), we conclude with the Main Theorem of~\cite{conjugate}.
In case (b), composing $\varphi$ with the complex conjugation, we obtain a holomorphic 
conjugacy $\overline{\varphi}$ between $f$ 
and $\overline{g}$. Again, the Main Theorem of~\cite{conjugate}   implies that 
$\overline\varphi$ is an automorphism, as desired. 

To complete the proof it  remains to show that Case (c) leads to a contradiction. Denote by $\cs$ the complex structure of 
 $\C^2$ and by  $\cs'$ the pull back of $\cs$ under $\varphi$. 
Then, exchanging the roles of the stable and unstable manifolds if necessary, $\cs'$ coincides with $\cs$ along $W^u_{p_i}$ and with $-\cs$ along $W^s_{p_i}$, for a set of saddle periodic points $p_i$ which is dense in $\C^2$ for the real analytic topology. In particular for every such $p_i$ and every $z\in W^{s/u}(p_i)$, $T_zW^{s/u}(p_i)$ is $\cs$ and $\cs'$ invariant.

On the real tangent spaces $T_{p_i}\C^2\simeq \R^4$, there is a plane on which $\cs' v=\cs v$ and another one on which $\cs' v=-\cs v$. We claim that by density, this property holds everywhere. 
To see this, we study the properties of $\cs'\cs \in \mathrm{End}(T_z\C^2)$. Observe that 
for $\e =\pm 1$ and $v\in T_z\C^2$, 
$$\cs'\cs v=\e v  \Leftrightarrow  \cs v =-\e \cs' v \Leftrightarrow  \cs' v =-\e \cs v 
\Leftrightarrow \cs \cs' v=\e v,  $$ and 
in addition this property is invariant under $\cs$ and $\cs'$, i.e.\ $v$ satisfies it if and only if $\cs v$ or $\cs' v$ does. It follows that on an $\R$-Zariski dense set of points $z\in \C^2$, the characteristic polynomial 
$\chi_z(X) $ of
$\cs'\cs$ is $(X-1)^2(X+1)^2$; since $\chi_z(X) $
depends analytically on $z$, this property holds everywhere. Likewise, the minimal 
polynomial  of $\cs'\cs$ at each $p_i$ is $(X-1)(X+1)$ so again this property holds everywhere.
In this way, we obtain two distributions of planes $P_\pm(z)   = \ker (\cs'\cs  - \pm\id)$
such that 
\begin{enumerate}[\rm(a)]
\item $T_z\C^2=P_{-}(z)\oplus P_+(z)$ at each point $z$ of $\C^2$, 
\item  $P_-(z)$ and 
$P_+(z)$ are invariant under $\cs$ and $\cs'$.
\end{enumerate}
Furthermore,  we claim that 
$P_{-}(z) = T_zW^u_{p_i}$ and $P_{+}(z) = T_zW^s_{p_i}$ as soon as $z$ is a point on $W^u_{p_i}$ or $W^s_{p_i}$, respectively. Indeed,  linear algebra shows that 
$P_-(z)$ and 
$P_+(z)$ are the only $\cs\cs'$- and $\cs$-invariant  planes at $z$, so for 
$z\in W^u(p_i)$, $T_z W^u(p_i)$ must be one of $P_-(z)$ or 
$P_+(z)$ (and likewise for $W^s(p_i)$),
and the property follows by continuity.

These distributions of planes define two real analytic foliations, the leaves of which are holomorphic. 
Indeed, if a   $k$-dimensional distribution is not integrable, then there is an open set $U$
in which there is   a pair of vector fields  tangent to the distribution whose Lie bracket is everywhere transverse to the distribution. Conversely, this  property 
cannot be satisfied if there is a dense set of points at which one can find a local submanifold of dimension $k$ tangent to the distribution. 
In the real-analytic case, to guarantee integrability, 
it is enough to find such a set that  is dense for the real-analytic topology\footnote{Let $z\mapsto P(z)\subset T_z\C^2$ be a real analytic distribution of real planes. Given a real analytic vector field $z\mapsto v(z)$, its orthogonal projection  $v_P(z)$ on $P(z)$ for the standard euclidean metric is also real analytic.  In this way, we construct many real analytic vector fields which are everywhere tangent to $P$. The integrability of  $P$ means that the Lie bracket of any pair of such vector fields is again tangent to $P$. Thus, the integrability property propagates from any $\R$-Zariski dense subset to $\C^2$. }. 
Since the $W^u_{\loc}(p_i)$ and $W^s_{\loc}(p_i)$ provide such local submanifolds, we conclude that $P_-$ and $P_+$ define two foliations $\mathcal{F}_-$ and $\mathcal{F}_+$.  Since, by definition, $P_-$ and $P_+$ are $\cs$-invariant, the leaves are holomorphic. 

This argument shows  that  $f$   preserves two global 
real analytic foliations by holomorphic curves,  
so that Theorem~\ref{mthm:no_real_global_foliation} completes the proof. 
 \end{proof}

\section{Multipliers in number fields}\label{sec:rational_multipliers} 

\subsection{Two criteria}
Recall that the Lyapunov exponent of  a saddle periodic point $p$ of period $n$, 
is by definition
\begin{equation}
\chi^u(p) = \unsur{n} \log \abs{\lambda^u(p)}.
\end{equation}  
We also  denote by 
$\chi^u(\mu_f)$ the positive Lyapunov exponent of the unique measure of maximal entropy. 

 Huguin's theorem~\cite{huguin} asserts that a rational map in $\P^1(\C)$ whose multipliers lie in some 
 fixed  number field must be exceptional. Here we obtain two partial generalizations of this result, under 
 some --presumably superfluous-- additional assumptions. 

\begin{thm}\label{thm:multipliers_huguin1}
Let $f\in \Aut(\C^2)$ be a loxodromic automorphism. Assume that:
\begin{enumerate}[\rm (1)]
\item $f$ is uniformly hyperbolic;
\item $f$ admits two saddle periodic points with distinct Lyapunov exponents. 
\end{enumerate}
Then the unstable (resp. stable) multipliers of $f$ cannot lie in a fixed number field. 
\end{thm}


See below Proposition~\ref{pro:examples_huguin1} for a sufficient condition for a hyperbolic automorphism to satisfy Assumption~(2).  

\begin{proof} Without loss of generality we work with unstable multipliers. 

{\bf{Step 1.--}} We first prove the result under the  additional assumption:
\begin{enumerate}\setcounter{enumi}{2}
\item \textit{$f$ is defined over a number field.}
\end{enumerate}
By this we mean that   $f$ is conjugated to a  composition of Hénon maps with algebraic coefficients. 

Following \cite{huguin}, we argue by contradiction. 
Let $L$ be a number field containing the coefficients of $f$ as well as all 
unstable multipliers. 
Let $p$ be any saddle periodic point. We will show that $\chi^u(\mu_f) = \chi^u(p)$, thereby contradicting 
Assumption~(2). 

 For this, let $(q_n)$ be the sequence provided by Theorem~\ref{thm:homoclinic_multipliers}, 
associated to some homoclinic intersection, which we  assume moreover  to be Zariski dense by Remark~\ref{rem:zariski}.  By a diagonal process (see e.g.~\cite[p. 3455]{DF}) we may 
extract a subsequence $(q_{n_j})$ that 
converges to the generic point for the $\overline \Q$-Zariski topology.  Without loss of generality rename $(q_{n_j})$ into $(q_n)$. 
Let  $\mathrm{Gal}(q_n)$ be the orbit of $q_n$ under the action of the absolute Galois group $\mathrm{Gal}(\overline \Q /L)$. Yuan's arithmetic 
 equidistribution theorem (see~\cite{Yuan:Inventiones}, and~\cite{Lee:equidistribution} for this application) implies that the sequence of probability measures
\begin{equation}\label{eq:yuan}
\mu_n:=\unsur{\mathrm{Gal}(q_n)} \sum_{r\in \mathrm{Gal}(q_n)} \delta_r  
\end{equation}
converges to the equilibrium measure $\mu_f$ in the weak-$\star$ topology.

Since $f$ is uniformly hyperbolic, the unstable line field is continuous along $J$. For $z\in J$, 
denote by $\norm{df_z^u}$ the euclidean norm of $df_z(e^u(z))$ for any
unit vector $e^u(z)$ tangent to the unstable direction at $z$; this does not depend on the choice of $e^u(z)$.
Then the map 
$z\mapsto \log \norm{df_z^u}$ is continuous. For any invariant measure $\nu$, 
the average positive  Lyapunov exponent is $\chi^u(\nu) = \int \log \norm{df_z^u} d\nu(z)$. 
In particular, Yuan's equidistribution theorem implies that 
$\chi^u(\mu_n)\to \chi^u(\mu_f)$ as $n\to\infty$. 

Now, if we decompose $\mu_n$ as a sum of ergodic invariant measures along periodic orbits, we see 
that all these periodic orbits are $\mathrm{Gal}(\overline \Q /L)$-conjugate to that of $q_n$, so their multipliers are Galois conjugate as well. But by assumption these multipliers lie in 
$L$ so they are all equal and $\chi^u(\mu_n) = \chi^u(q_n)$. 

On the other hand, since  $\lambda^u(q_n) \sim c \lambda^u(p)^n$, it follows 
  that $\chi^u(q_n)\to \chi^u(p)$ as $n\to\infty$. Altogether, we conclude that 
$\chi^u(p)   = \chi^u(\mu_f)$, as asserted.

\medskip

{\bf{Step 2.--}} We treat the general case by using a specialization argument in the style of~\cite[\S 5]{DF}. By \cite{friedland-milnor}, we may assume that $f$ is a composition of complex Hénon maps. By assumption at least one coefficient of $f$ is transcendental. Let 
$R\subset \C$ be the     $\overline \Q$-algebra  generated by  the coefficients of $f$ 
and $f\inv$. Then, there is an algebraic variety $V$ defined 
over $\overline \Q$ and a Zariski open subset $S$ of $V$ such that  $K = \mathrm{Frac}(R)$ is the function field of $V$ and   the elements of $R$ 
correspond to regular functions on $S$ (we may assume $S = \mathrm{Spec}(R)$). In this way, we may 
view $f$ as a family over $S$, that is, for every $s\in S$, by evaluating the coefficients of $f$ at $s$, we 
obtain a polynomial map $f_s$ on $\mathbb{A}^2_s$. By \cite[Lem. 5.1]{DF}, by replacing $S$ by some  
 Zariski open subset (still denoted by $S$), we may assume that every $f_s$ is a
 polynomial automorphism of $\A^2$. 
 
 We actually view $S(\C)$ as a complex analytic variety endowed with its Euclidean topology, so that $f_s$ is viewed as a polynomial automorphism of $\C^2$. Then, there is a parameter $b\in S(\C)$ (referred to as the base point) 
 corresponding to our initial automorphism $f$. By construction, 
 $\set{b}$ is $\overline \Q$-Zariski dense in $S$, i.e. it corresponds to the generic point of $S$ over $\overline \Q$. 
Let us suppose by way of contradiction that  there is a  number field $L$ 
 containing all the unstable multipliers of $f_b$. 

\begin{lem}\label{lem:constant_multiplier}  
Let $(f_s)_{s\in S}$ be such a family of automorphisms of $\C^2$, with $b\in S$ being generic. Let $q_b$ be a saddle periodic point of $f_b$, 
of period $m$. 
Let $\mathcal N\subset S$ be a connected neighborhood of $b$  in  which the saddle periodic point $q_b$ persists, i.e.\ $q_b$ extends as a holomorphic map $s\mapsto q_s\in \C^2$ such that $q_s$ is a saddle periodic point of $f_s$ of period $m$. 
If the unstable multiplier $\lambda^u(q_b)$ is algebraic, then the function 
$s\in \mathcal N \mapsto \lambda^u(q_s)\in \C$ is constant. 
\end{lem}

Note that this constancy propagates in a sense  to the whole family $f_s$ (see the proof of Proposition~\ref{pro:unstable_algebraic} below), but since the stable multiplier 
of $q$ may be transcendental, it may happen that $q_s$ bifurcates and ceases to be a saddle. 

Assuming the lemma, we conclude the proof of  the theorem as follows. Since uniform hyperbolicity is an open property, there is a neighborhood $\mathcal N$ of $b$ such that 
\begin{itemize}
\item for $s\in \mathcal N$, 
  $f_s$ remains hyperbolic, and
  \item  the saddle points of $f_b$ can be followed holomorphically over $\mathcal N$ as saddle points of~$f_s$. 
  \end{itemize}
For  $s\in\mathcal N$, $f_s$ satisfies the assumptions~(1) and~(2)  and all its unstable multipliers belong to~$L$. Choosing $s\in\mathcal N\cap S(\overline \Q)$ and applying the first step of the proof to $f_s$ yields  the desired contradiction. 
\end{proof}

\begin{proof}[Proof of Lemma~\ref{lem:constant_multiplier}] 
To determine $\lambda^u(q_s)$, we first solve  the equation  
$f_s^m(q_s) = q_s$ and then   do  an extension of degree at most $2$ to find the eigenvalues of 
$Df_s^m(q_s)$. Thus, $\lambda^u(q_s)$ is algebraic over $K$ (i.e. it is a branch of a multivalued algebraic function on $S$). Let $P(X)=\sum_k a_k X^k$ be a polynomial equation 
with coefficients in $R$, of minimal degree, such that $P_s(\lambda^u(q_s))=0$ for all $s$ (where we write
 $P_s=\sum_k a_k(s)X^k$). 
It  is an irreducible polynomial over $K$.
By assumption, $\lambda^u(q_b)$ is an algebraic number $\alpha$. The equation $P_s(\alpha)=0$ is an algebraic equation on $S$  with coefficients in $\overline \Q$, namely $\sum_k \alpha^k a_k(s)=0$, and it is satisfied at the generic point $b$, hence it is satisfied everywhere. Thus, $P(\alpha)=0$, hence $(X-\alpha)$ is a factor of $P$,  and  $P=a_1(X-\alpha)$ since $P$ is irreducible. This implies that $\lambda^u(q_s)=\alpha$ for all $s\in \mathcal N$.
\end{proof}


\begin{thm}\label{thm:multipliers_huguin2}
Let $f\in \Aut(\C^2)$ be a loxodromic automorphism. Assume  :
\begin{enumerate}[\rm (1)]
\item $f$ admits a saddle periodic point $p$ such that $\chi^u(p)> \chi^u(\mu_f)$. 
\end{enumerate}
Then the unstable (resp. stable) multipliers of $f$ cannot lie in a fixed number field. 
\end{thm}

\begin{rem}\label{rem:zdunik}
If $f:\P^1(\C)\to \P^1(\C)$ is a non-exceptional rational map,
 a result of Zdunik~\cite{zdunik:exponents} guarantees that 
  $f$ admits repelling 
  periodic points with exponents greater than $\chi^u(\mu_f)$. The contradiction in~\cite{huguin} is 
  based on this property. 
\end{rem}

\begin{proof}
Since $\chi^u(p) + \chi^s(p)  = \log\abs{\jac(f)} = \chi^u(\mu_f) + \chi^s(\mu_f)$, Assumption~(1) implies
\begin{equation}
\chi^s(p)< \chi^s(\mu_f)<0< \chi^u(\mu_f)<\chi^u(p).
\end{equation}
So changing $f$ into $f^{-1}$ replaces Assumption~(1) by the equivalent inequality $\chi^s(p)< \chi^s(\mu_f)$; this shows that  the statements on the stable and unstable multipliers are equivalent.  Without loss of generality we show that the unstable multipliers cannot lie in a fixed number field.

{\bf{Step 1.--}}  As for Theorem~\ref{thm:multipliers_huguin1}, we first prove the theorem under the additional assumption: 
\begin{enumerate}[\rm (1)] \setcounter{enumi}{1}
\item \textit{$f$ is defined over a number field.}
\end{enumerate}

The proof follows closely that of Theorem~\ref{thm:multipliers_huguin1}, the only difference being that 
due to the lack of hyperbolicity, the unstable directions are a priori not continuous, thus the function $\log\norm{df_z^u}$  migth be discontinuous, and we cannot assert that $\chi^u(\mu_n) \to \chi^u(\mu_f)$ as $n\to\infty$. On the other hand, the upper semi-continuity of Lyapunov exponents 
(\footnote{Recall the argument for upper-semicontinuity: for any   invariant measure $\nu$, the average 
upper Lyapunov exponent 
$\chi^u(\nu)$ is the limit of $\unsur{k}\int \log\norm{Df^k_x}d\nu(x)$ as $k\to \infty$. 
Choosing a submultiplicative norm and taking the limit along the subsequence $k=2^q$ 
realizes   $\nu\mapsto \chi^u(\nu)$ as the limit of a decreasing sequence of continuous functions for the weak-*
topology, hence it is upper semi-continuous.})
shows that 
$\limsup_{n\to\infty} \chi^u(\mu_n)\leq \chi^u(\mu_f)$. 
But since $\chi^u(\mu_n)  = \chi^u(p)$, 
by choosing $p$ from  the hypothesis~(1), we get the desired contradiction. 

\medskip

{\bf{Step 2.--}}  
We use the same specialization formalism as in the proof of Theorem~\ref{thm:multipliers_huguin1}, so we resume the notation from that proof.  Again arguing by contradiction, we assume that there is a number field $L$  containing the unstable multipliers of $f_b$. 
 
\begin{lem}\label{lem:stability}
There exists a 
 neighborhood $\mathcal N$ of $b$ in the parameter space   $S$ such that   any saddle periodic point 
 $q_b$ of $f_b$  persists as a saddle point $q_s$ of $f_s$   for all  $s\in \mathcal N$. 
\end{lem}
 
Assuming this lemma for the moment, we complete the proof. By Lemma~\ref{lem:constant_multiplier}, all unstable multipliers of saddle periodic points are constant in $\mathcal N$, so they remain in $L$. 
We claim that the unstable Lyapunov exponent of the maximal entropy measure also remains constant in $\mathcal N$, so that assumption~(1) stays satisfied. Indeed, by~\cite[Thm 2]{bls2}, for any loxodromic automorphism $g$  of dynamical degree $d$, 
if $\mathrm{SPer}_n$ is any set of saddle periodic points of period $n$ for $g$
such that $\#\mathrm{SPer}_n\sim d^n$, then
\begin{equation}\label{eq:lyapunov_exponent}
\unsur{d^n}\sum_{q\in \mathrm{SPer}_n} \chi^u(q) \underset{n\to\infty}\longrightarrow \chi^u\lrpar{\mu_{g}}. 
\end{equation}
So if we fix such a set for $f_b$, it persists in $\mathcal N$,
 and~\eqref{eq:lyapunov_exponent} says that 
$\chi^u(\mu_{f_s})$ is constant in $\mathcal N$, as claimed. To conclude, 
we pick any parameter $s\in\mathcal N\cap S(\overline \Q)$ and apply the first step  to get the desired contradiction. 
\end{proof}

The proof uses some basic facts from the stability theory of polynomial automorphisms~\cite{tangencies}. In particular we say that a 
holomorphic family $(f_\lambda)_{\lambda\in \Lambda}$ 
of loxodromic automorphisms is \emph{weakly stable} if its periodic points do not bifurcate; equivalently, saddle points remain of saddle type and can thus be followed holomorphically  in the family.  

\begin{proof}[Proof of Lemma~\ref{lem:stability}]
To ease notation set $J_t = \abs{\jac(f_t)}$. 
At the parameter $b$ we have $\chi^u(\mu_{f_b}) + \chi^s(\mu_{f_b})   = \log (J_b)$, hence $\chi^u(\mu_{f_b}) > \log (J_b)$. 
Fix some $\e\in \R$ such that 
\begin{equation}\label{eq:choice_of_epsilon}
0<2 \e < \chi^u(\mu_{f_b}) - \log (J_b).
\end{equation} 
   By Theorem~\ref{thm:equidistribution_reinforced} (which is just a mild extension of~\cite{bls2}) 
   we can construct a family 
   $\mathrm{SPer}^+=\mathrm{SPer}^+_b$ of saddle periodic orbits  
  which is dense in $\jstar_b$ and for every $q_b \in \mathrm{SPer}^+_b$, $\chi^u(q_b)\geq \chi^u(\mu_{f_b})  -\e$.  Every  $q = q_b\in \mathrm{SPer}^+_b$ 
persists as a saddle point $q_t$ in some neighborhood $\mathcal N(q)$
of $b$, and by Lemma~\ref{lem:constant_multiplier}, $\chi^u(q_t)$ is constant in $\mathcal N(q)$, so 
$\chi^u(q_t) = \chi^u(q_b)\geq \chi^u(\mu_{f_b})  -\e$. 
On the other hand $\chi^s(q_t)+\chi^u(q_t) = \log (J_t)$, thus from Equation~\eqref{eq:choice_of_epsilon} we get 
\begin{equation}
\chi^s(q_t) \leq \log (J_t) - \chi^u(\mu_{f_b})  +\e \leq \log (J_t) - \log (J_b) -\e.
\end{equation}
From this we see that $\chi^s(q_t)$ remains negative on the open set  
$\{t\; ; \;  \log (J_t) < \log (J_b)+\e\}$. 
This property  holds in some neighborhood $\mathcal N$ of $b$, independent of $q$,  and we conclude  that in  $\mathcal N$ all 
periodic points from  $\mathrm{SPer}^+$ persist. Since $\mathrm{SPer}^+_b$ is dense in $\jstar_b$, by~\cite[Cor. 4.15]{tangencies}, 
the family $(f_t)_{t\in \mathcal N}$ is weakly  stable and it follows from~\cite[Thm 4.2]{tangencies} that all saddle periodic points of 
$f_b$ persist in $\mathcal N$, as asserted.  
 
There is actually a delicate point here: the  results of~\cite{tangencies} hold under the assumption that the family $(f_t)$ is \emph{substantial}. This means that either
\begin{itemize}
\item all the members of the family are dissipative (i.e. $J_t< 1$ for $t\in S$), 
\end{itemize}
or
\begin{itemize}
\item for any periodic point $q$, with eigenvalues $\alpha_1$, $\alpha_2$, there is no persistent relation of the form 
$\alpha_1^a\alpha_2^b   = c$ on the parameter space $S$ ($a$, $b$, $c$ are any complex numbers with $\abs{c}=1$). 
\end{itemize}
If $\abs{\jac(f_b)}< 1$, by reducing $\mathcal N$ if necessary 
this property holds for every $f_t$, $t\in \mathcal N$  and we are done. 
If $\abs{\jac(f_b)}> 1$, we can simply apply the stability results of~\cite{tangencies} to the family  $(f_t\inv)$ and we are done as well.  
On the other hand if  $\abs{\jac(f_b)}= 1$, we cannot a priori guarantee that $(f_t)_{t\in \mathcal N}$ is substantial. 
This assumption is used in~\cite{tangencies} at exactly one point (see Lemma 4.11 there): assume 
that $(p_n(t))_{n\geq 0}$ is a sequence of 
 holomorphically moving saddle periodic  points, and 
 $q(t)$ is a holomorphically moving periodic point 
  which is a saddle at some parameter $t_0$ and such that $p_n(t) \to q(t)$ as $n\to \infty$. 
  We have to show that $q(t)$ remains a saddle at all parameters. 
The substantiality assumption is used in \cite[Lem. 4.11]{tangencies} to exclude the possibility that 
the multipliers of $q_t$ both cross the unit circle at a parameter $t$ for which
 $\abs{\jac(f_t)}=1$. But in our situation, if there happens to be 
  a persistent relation of the form 
 $\alpha_1^a(t)\alpha_2^b(t)   = c$ between the multipliers of  $q(t)$, 
 since near $t_0$ the unstable multiplier ($\alpha_1$ say) 
 is constant by Lemma~\ref{lem:constant_multiplier},  so $\alpha_2$ is constant as well, and $q(t)$ remains a saddle for all parameters. Thus in our case the substantiality assumption is useless, and we are done.  
\end{proof}

The proofs of Theorems~\ref{thm:multipliers_huguin1} and~\ref{thm:multipliers_huguin2} 
would be more natural if we were directly able to show that the algebraicity of unstable mutipliers implies that $f$ is 
defined over $\overline \Q$ (up to conjugacy).  In~\cite{huguin} this corresponds to the first step of the proof of the main theorem, which relies on McMullen's theorem saying that the unstable multipliers determine the conjugacy class of $f$ up to finitely many choices (see~\cite{mcm_algorithms}).  We obtain  a partial result in this direction.  

\begin{pro}\label{pro:unstable_algebraic}
Let $f\in \Aut(\C^2)$ be a loxodromic automorphism such that all unstable multipliers of $f$ are algebraic. 
Suppose that  there is no non-trivial  weakly stable algebraic family through $f$. Then $f$ is defined over a number field. 
\end{pro}

Note that the non-existence of stable algebraic families is a key step in~\cite{mcm_algorithms}. It is expected to hold in $\Aut(\C^2)$ as well. Since here we only need the non-existence of such a family passing through $f$, the approach is different from~\cite{huguin}. 

\begin{proof}[Proof (sketch)]
Normalize $f$ as a product of Hénon mappings, and assume by  
contradiction that $f$ is not defined over $\overline \Q$.  
As in Theorems~\ref{thm:multipliers_huguin1} 
and~\ref{thm:multipliers_huguin2} we construct a non-trivial algebraic family $(f_t)_{t\in S}$ in which $f=f_b$ corresponds to a generic parameter. We claim that this family is weakly stable, which yields the desired contradiction. 

Indeed, for $f_b$ 
we consider the family of saddle periodic points $\mathrm{SPer}^+$ given by Theorem~\ref{thm:equidistribution_reinforced}
  By Lemma~\ref{lem:constant_multiplier}, the unstable multiplier of the continuation of 
 any $q_b\in   \mathrm{SPer}^+_n$ is locally constant. In addition, 
  exactly as in Lemma~\ref{lem:stability}, for large $n_0$
  we can construct a uniform neighborhood in 
  which all points in   $\mathrm{SPer}^+_n$, $n\geq n_0$ 
   persist as saddle points, depending only on the 
  Jacobian.  If   $t\in S$ is an arbitrary parameter, we can find a path joining $b$ to $t$ avoiding the 
  locus where the second multiplier hits the value 1 (see~\cite{BHI}), so $q_b$ can be followed along that 
  path, and we see that $\lambda^u(q_b)$ is still an eigenvalue of $q_t$. However, it might be that $q_t$ 
  is now a repelling or semi-repelling point.  
  We claim that $\jac(f_t)$ is constant along the family and that this phenomenon does not happen. Indeed if $S\ni t\mapsto \jac(f_t)$ were not constant, then for some parameter $t$ we would have 
  $ \log \abs{\jac(f_t)}> \chi(\mu_{f_b}) +\e$, and at such a parameter all the points of 
  $\mathrm{SPer}^+_n$ would have become repelling. But this is impossible since by~\cite{bls2} the majority of periodic points of $f_t$ are saddles. So the Jacobian is constant, 
  and the points of $\mathrm{SPer}^+_n$, $n\geq n_0$
   persist as saddles throughout the family,  which must then 
  be weakly stable, as asserted. 
\end{proof}

\subsection{Examples}

The assumptions of Theorems~\ref{thm:multipliers_huguin1} 
and~\ref{thm:multipliers_huguin2}  can be checked on certain perturbative examples. We  illustrate 
this on the Hénon family  
but   more general close-to-degenerate examples could be considered. 
For $d\geq 2$, let us parameterize 
  the space $\mathcal {H}_d$ of Hénon maps of degree $d$ by 
 $(a, c) \in \C^\times \times \C^{d-1} $, where 
 $f_{a,c}(z,w) = (aw+ p_c(z), z)$, with  $p_c(z) = z^d+\sum_{j=0}^{d-2} c_j z^j$. Note that $p_c$ is 
 \emph{exceptional} (or \emph{integrable}) when it is conjugated to $\pm T_d$, where 
 $T_d$ is a Chebychev polynomial, or when   $c=0$, in which case $p$ is monomial. 
We  obtain the  extended parameter space $\overline{\mathcal {H}_d}$ 
by adjoining the hypersurface $\set{a=0}$ 
of \emph{degenerate maps}:  $f_{0, c}$  maps $\C^2$ to the curve $\set{z=p_c(w)}$ and the dynamics on that curve is conjugated to that of $p_c$.   
 In this respect, a Hénon map $f_{a,c}$ 
with small Jacobian $a$ may be seen as a perturbation of $p_c$.

\begin{thm}\label{thm:multipliers_henon}
There exists a neighborhood
$\mathcal N$ of $\set{0}\times \C^{d-1} $ 
in the  extended Hénon parameter space $\overline{\mathcal {H}_d}$  
such that if 
$f\in \mathcal N$
then the unstable (resp. stable) multipliers 
of $f$ are not contained in a fixed number field. 
\end{thm}

It is interesting that, as opposed to \cite{huguin},  Chebychev and monomial polynomials do not play an exceptional role in this result. 

\begin{proof} 
Assume first that $p_{c_0}$ is neither monomial  nor conjugate to a Chebychev polynomial. 
Then by the above-mentioned theorem of Zdunik~\cite{zdunik:exponents}, $p_{c_0}$ 
admits a repelling point $r_0$ with 
$\chi(r_0)> \chi(\mu_{p_{c_0}})$. For $(a,c)$ close to $(0,c_0)$, $a\neq0$,
$r_0$ persists as a saddle point $r = r_{a,c}$ with 
unstable Lyapunov exponent close to $\chi(r_0)$. Furthermore $\chi^u(\mu_{f_{a,c}})$ is close to 
$\chi(\mu_{p_{c_0}})$ (see~\cite[\S 3]{continuity}; actually we only need the ``easy''  inequality 
$\limsup_{(a,c)\to(0, c_0)} \chi^u(\mu_{f_{a,c}}) \leq \chi(\mu_{p_{c_0}})$).  So 
Theorem~\ref{thm:multipliers_huguin2} applies and we are done. 

If $p_{c_0}$ is conjugated to  $\pm T_d$, we observe that all Lyapunov exponents of periodic points are equal to $\log d$, except for the   post-critical fixed points whose exponent is $2\log d$ (see \cite[Cor. 3.9]{milnor:lattes}).
Since $\pm T_d$ has a connected Julia set, $\chi(\mu_{\pm T_d}) = \log d$, so again the assumptions of 
Theorem~\ref{thm:multipliers_huguin2} are satisfied (\footnote{The difference between the one- and the two-dimensional situations is that in the one-dimensional case, the post-critical fixed points do not have transverse homoclinic intersections}).  

The most delicate case is that of perturbations of $z^d$, which is dealt with in the following lemma. 
\end{proof}

If $(a, c)$ is close to $(0,0)$, then  $f_{(a,c)}$ is uniformly hyperbolic and its Julia set is a solenoid. In particular for any $x\in J$, 
$J_{W^u(x)}$ is an unbounded simple  curve (see~\cite{fornaess-sibony:duke,HOV2}). 
This property propagates by stability to the whole hyperbolic  component containing these parameters.

\begin{lem}\label{lem:hyperbolic_mulitpliers}
Any automorphism belonging to the 
 hyperbolic component of $(0,0)$  satisfies the assumptions of Theorem~\ref{thm:multipliers_huguin1}, that is, 
it admits two saddle periodic points with distinct Lyapunov exponents. 
\end{lem}

This result is closely related 
to Corollary~\ref{cor:dimH} but for  convenience we give a direct argument.

\begin{proof}
Let $f = f_{a,c}$ be such an automorphism.   In the hyperbolic component, the modulus of the 
jacobian determinant is  must remain smaller than 1, because there is a sink that does not bifurcate. Now, assume by contradiction that all its unstable multipliers are equal. By Theorems 5.1 and~7.3 in \cite{bs6} (see page 733 for the detailed version of Theorem 7.3),
since $J$ is connected,  
$\chi^u(\mu_f) = \log d$, so by the equidistribution of saddle points and their
 exponents~\cite[Thm 2]{bls2}, all 
unstable Lyapunov exponents are equal to $\log d$. 

We claim that the unstable dimension of $f$, that is, the Hausdorff dimension of $J\cap W^u_\loc(x) = J^+\cap W^u_\loc(x)$ for any $x\in J$, is equal to 1. 
Indeed, the unstable dimension  $\delta^u$ is given by the formula 
  \begin{align} 
\delta^u & =  \dim_H\lrpar{ J\cap W^u_{\loc}(x)}  
 = \frac{h_{\kappa^u}(f)}{  \int \log \norm{Df\rest{E^u(x)}} d \kappa^u(x) }  \label{eq:thermo}
 \end{align}  
  where $\kappa^u$ is the unique equilibrium state associated to $\delta^u\log \norm{df\rest{E^u}}$ and
  ${h_{\kappa^u}(f)}$ is its measure theoretical entropy  (see the proof of Corollary~\ref{cor:dimH} and Pesin's book~\cite[Thm 22.1]{pesin_book}; this goes back to the work of Manning and McCluskey \cite{mccluskey-manning}).
   Since all unstable multipliers are equal to $\log d$, the Livschitz theorem entails that 
    $x\mapsto  \log \abs{Df\rest{E^u(x)}}$ is cohomologous to $\log d$, that is, there exists a measurable  function  $\phi$(\footnote{Actually $\phi$ is Hölder continuous, since   by the regularity of the unstable lamination, 
        $x\mapsto  \log \norm{Df\rest{E^u(x)}}$ is Hölder continuous.})      such   that 
    \begin{equation}
    \log \norm{Df\rest{E^u(\cdot)}} - \log d = \phi\circ f - \phi.
    \end{equation} 
It follows that $\int \log \norm{Df\rest{E^u(x)}} d \kappa^u(x) = \log d$.   
On the other hand, $ {h_{\kappa^u}(f)}$ is bounded by the topological entropy of $f$, that is by $\log d$. This shows that 
    $\delta^u\leq 1$. But since $f$ is unstably connected $\delta^u\geq  1$, and we are done. 
 
 In addition, $J\cap W^u_\loc(x)$ has positive and finite 1-dimensional Hausdorff measure (see~\cite[Thm 22.1]{pesin_book}). Being a curve, it is then rectifiable (see~\cite[\S 3.2]{falconer}), and Theorem~\ref{thm:rectifiable} provides the desired contradiction.
\end{proof}

 Pushing the argument of Lemma~\ref{lem:hyperbolic_mulitpliers} further, we obtain 
a non-perturbative sufficient condition for   Theorem~\ref{thm:multipliers_huguin1}. 

\begin{pro}\label{pro:examples_huguin1}
Let $f\in \Aut(\C^2)$ be uniformly  hyperbolic. If its Julia set $J$ is not totally disconnected, there exist two saddle 
periodic points with distinct Lyapunov exponents.
\end{pro}

\begin{proof}
We may assume that $\abs{\jac(f)}\leq 1$; equivalently  $\chi^u(\mu_f)+\chi^s(\mu_f)\leq 0$. Assume  that the   
 Lyapunov exponents of periodic points are all
 equal to some $\chi$.

If $J$ is connected, then by Theorem 0.2 of~\cite{bs6}, $f$ is unstably connected, and then by Theorem 7.3 of~\cite{bs6} (page 733),
$\chi^u(\mu_f) = \log d$. 
Arguing as in Lemma~\ref{lem:hyperbolic_mulitpliers}, the Hausdorff dimension $\delta^u$ of 
$ J\cap W^u_{\loc}(x)$ is almost surely equal to $1$, which contradicts Corollary~\ref{cor:dimH}. 

If $J$ is disconnected,   then by Theorems~5.1 and~7.3 of~\cite{bs6} we have $\chi^u(\mu_f) > \log d$. Since $\chi^u(\mu_f)+\chi^s(\mu_f)\leq 0$, we obtain 
 $\abs{\chi^s(\mu_f)} > \log d$ as well. 
 Thus  the Lyapunov exponent $\chi$ of   periodic points 
   is larger than $\log d$ and we infer  that for  $x\in J$,
 \begin{equation}
 \dim_H\lrpar{ J\cap W^u_{\loc}(x)} = \frac{h_{\kappa^u}(f)}{   \chi^u(\kappa^u) } \leq \frac{\log d}{ \chi^u(\kappa^u)}<1     
 \end{equation} 
 and likewise in the stable direction. Thus 
 $J\cap W^u_{\loc}(x)$ and  $J\cap W^s_{\loc}(x)$ are totally disconnected. Since the uniform hyperbolicity shows that $J$ has a local  product structure (of $J^+$ by $J^-$), $J$ is totally disconnected. This contradiction concludes the proof. \end{proof}

\begin{rem} 
As a consequence of this proposition and  \cite{wolf:TAMS} (see Theorem 4.2 there, and the remark following it), a dissipative and hyperbolic automorphism whose Julia set is not totally disconnected does not admit a measure of full dimension, that is, such that $\dim(\nu) = \dim_H(J)$. 
\end{rem}

\appendix

\section{Equidistribution of periodic orbits and Lyapunov exponents} \label{app:equidist}

Let $f\in \Aut(\C^2)$ be a loxodromic automorphism of dynamical degree $d$.
Let $\mathrm{SPer}_n$ denote the set of saddle periodic orbits of $f$ of exact period $n$.
The equidistribution  theorems for periodic orbits  of Bedford, Lyubich and Smillie~\cite{bls2} 
asserts first, that  almost every periodic point is a saddle point, i.e.
\begin{align} \label{eq:counting_saddle_points}
\lim_{n\to \infty}\frac{1}{d^n} \#\mathrm{SPer}_n = \lim_{n\to \infty}\frac{1}{d^n}\#\mathrm{Fix}_n &=1 ,
\end{align}
and that these points equidistribute towards $\mu_f$. More precisely,
\begin{align}\label{eq:equidistribution_BLS}
\lim_{n\to\infty} \unsur{d^n}\sum_{p\in\mathrm{SPer}_n}  \delta_p  & = \mu, \\
 \label{eq:equidistribution_eponent_BLS}\lim_{n\to\infty} \unsur{d^n}\sum_{p\in\mathrm{SPer}_n}  \chi^u(p) & = \chi^u(\mu).
\end{align}

Also, it is straightforward to adapt Theorem S.5.5 in~\cite{KH} to the 2-dimensional complex setting to show that given any $\rho>0$, and any
 finite set of continuous functions $(\varphi_i)_{i=1}^j$, there exists a saddle periodic point $p$ of period $n$ such that 
\begin{equation}\label{eq:equidistribution_KH}
  \abs{\unsur{n}\sum_{k=0}^{n-1}  \varphi_i(f^k(p))- \int \varphi_i\, \d\mu}<\rho
\quad \text{ and } \quad
 \abs{   \chi^u(p) - \chi^u(\mu)}<\rho
\end{equation}
for every $1\leq i \leq j$.

Here we reinforce these    statements as follows: we obtain a version of~\eqref{eq:equidistribution_BLS} and~\eqref{eq:equidistribution_eponent_BLS} in which the average is not on all saddle periodic points but on a {\emph{typical periodic orbit of period $n$}} instead; and we show that the estimate~\eqref{eq:equidistribution_KH} is typical among periodic points of period $n$. 

\begin{thm}\label{thm:equidistribution_reinforced}
Let $f\in \Aut(\C^2)$ be a loxodromic automorphism of dynamical degree $d$. There exists a set $\mathrm{SPer}^+_n$ of saddle periodic orbits of period $n$ such that 
\begin{equation}
\lim_{n\to \infty}\frac{1}{d^n}\#\mathrm{SPer}^+_n =1
\end{equation}
 and  if $(p_n)$ is an arbitrary sequence with 
$p_n \in \mathrm{SPer}^+_n$, then
\begin{equation}\label{eq:equidistribution_reinforced}
\lim_{n\to\infty} \unsur{n}\sum_{k=0}^{n-1}  \delta_{f^k(p_n)}  = \mu
\quad \text{ and } \quad 
\lim_{n\to\infty}   \chi^u(p_n) = \chi^u(\mu). 
\end{equation}
\end{thm}
 
\begin{rem}
The method of~\cite{bls2} was implemented in other settings such as automorphisms 
of complex surfaces~\cite{cantat:Acta} or birational maps~\cite{birat}, and this reinforcement holds in those cases too, with exactly the same proof. 
\end{rem}

To prove the theorem, we revisit the proof of~\eqref{eq:equidistribution_BLS} in~\cite{bls2} by showing how to input 
the necessary  additional arguments. Let us 
quickly review that proof, by keeping the same notation. It requires some  Pesin theory formalism, such as Pesin boxes, which are closed subsets with product structure and uniform estimates on angles and expansion/contraction. An important idea is that if a Pesin box $P$
has small enough diameter, then 
one can construct a ``common Lyapunov chart'', that is, a topological bidisk $B$ 
containing $P$ that is nicely contracted/expanded by $f^n$ for every $x\in P$ and $n\in \Z$. This being said, the proof goes as follows.
\begin{enumerate}
\item[Step 1.] Fix some small $\e>0$ and cover a set of measure $1-\e$ by a finite set of disjoint Pesin boxes of diameter $<\e$.   This constant $\e$ will also imply uniformity constants for  angles, expansion, etc. 
\item[Step 2.] By the Closing Lemma, 
for every Pesin box $P$, any $n\geq 0$, and any $x\in P\cap f^{-n}P$, there exists a saddle point $\alpha$ shadowing the orbit segment $\set{x, \ldots , f^n(x) =y}$, which is 
$Ce^{-n\theta}$-close to $P$, for some $C = C(\e) $ and $\theta = \theta(\e)>0$. 
More precisely, it is  $Ce^{-n\theta}$-close to $[x,y]$, where
as usual   $[x,y] = W^s_\loc(x)\cap W^u_\loc(y)$. 
\end{enumerate}
Let $A_n = A_n(P)$ be the set of all periodic points $\alpha$ constructed in this way from $P\cap f^{-n}(P)$. 
We want to estimate $\#A_n(P)$. 
To do this, we shall estimate the measure of the set of points $x\in P\cap f^{-n}(P)$ that lead to the same periodic point $\alpha\in A_n(P)$. 
\begin{enumerate}
\item[Step 3.] Denote by $L(z)$ the Lyapunov chart around a point $z$. Fix $x$ and $\alpha$ as in Step 2. Set 
$$ B^s_{n, \alpha} = \bigcap_{k\leq n} f^{-k} (L(f^k(x)))$$
 (this is the component of $B\cap f^{-n}(B)$ containing $x$). Then, 
$$
T(\alpha)=P\cap f^{-n}(P) \cap B^s_{n, \alpha}
$$
is the set of points $x'$ in $P$ such that Step 2 applied to $x'$ and $n$ provides the same periodic point $\alpha$ as $x$.
Then, as shown in Lemma 5 of~\cite{bls2}, the product structure of $\mu$ and the invariance relation for the unstable conditionals imply
that 
$$
\mu(T(\alpha))\leq \mu(P)d^{-n}.
$$ 
\item[Step 4.] Let $A_n = A_n(P)$ 
be the set of all periodic points associated to $P\cap f^{-n}(P)$ as above, 
 we get that $\mu(P)d^{-n} \# A_n \geq \mu(P\cap f^{-n}(P))$, hence 
mixing shows that 
$$\liminf_{n\to\infty} d^{-n}  \#A_n \geq \mu(P).
$$ 
Thus if $\nu$ is any cluster value 
of $\unsur{d^n}\sum_{p\in\mathrm{SPer}_n}  \delta_p$,  we get that  $\nu(P)\geq \mu(P)$. 
\item[Step 5.] Since $\nu$ does not depend on  $\e$ nor  on the choice of Pesin boxes, we conclude that 
$\nu\geq \mu$, hence $\nu= \mu$. This gives the convergnece~\eqref{eq:equidistribution_BLS}. In particular, we obtain $\#\mathrm{SPer}_n\sim \#\mathrm{Fix}_n \sim d^n$, as stated in Equation~\eqref{eq:counting_saddle_points}.
\item[Step 6.] For Equation~\eqref{eq:equidistribution_eponent_BLS}, note that the unstable directions are almost parallel in each Pesin box. Hence,
$\log \abs{Df_\alpha(e^u(\alpha))}$ is close to  $\log \abs{Df_x(e^u(x))}$, so that taking the average on $x$ or on $\alpha$ gives approximately the same value. Since the average is over 
 all saddle periodic points and  does not depend on $\e$ we conclude that any cluster value of 
$\unsur{d^n}\sum_{p\in\mathrm{SPer}_n}  \chi^u(p)$ must be equal to  $\chi^u(\mu)$. 
\end{enumerate}

\begin{proof}[Proof of Theorem~\ref{thm:equidistribution_reinforced}]
If $\varphi$ is a function and $n$ a positive integer, we denote by $S_n\varphi$ the Birkhoff sum $\unsur{n}\sum_{k=0}^{n-1} \varphi\circ f^k$. 

{\bf{Preliminaries.--}} For $\mu_f$-almost every point $x$, the unstable direction $E^u(x)$ is well defined by the Oseledets Theorem. We let $e^u(x)$ be a unit vector in $E^u(x)$ and we define $\varphi_0(x) = \log \norm{Df(e^u(x))}$. Thus, $\varphi_0$ is defined $\mu_f$-almost everywhere. We also define $\varphi_0$ at every saddle periodic point by the same formula. For $\mu$-almost every $x$, 
$S_n\varphi_0(x)\to \chi^u(\mu)$, and if $\alpha$ is a saddle periodic point of period $n$, 
$S_n\varphi_0(\alpha) = \chi^u(\alpha)$.


Let $(\varphi_i)_{i\geq 1}$ be a countable family  of $C^1$ functions that is dense in 
$C^0(\supp(\mu))$. The following construction can be applied to $(\varphi_i)_{i\geq 1}$ or $(\varphi_i)_{i\geq 0}$.

By ergodicity, for every integer  $j$ and every $\rho>0$, 
there exists a sequence    $\delta (j, \rho, n)$ converging to 0 as $n$ goes to $+\infty$
such that 
\begin{equation}\label{eq:delta1}
\mu\lrpar{\set{ x: \ \exists i \in [1, j], \  \abs{S_n\varphi_i (x) - \int \varphi_i \, \d\mu}>\rho }}\leq \delta (j, \rho, n).
\end{equation}
Using the fact that $\delta(j, \cdot, n)$ is non-increasing, it follows that there exists $\rho_n\to 0$ and 
  $\delta' = \delta' (j,n)$ such that 
\begin{equation}\label{eq:delta2}
\mu\lrpar{\set{ x: \ \exists  i \in [1, j], \  \abs{S_n\varphi_i (x) - \int \varphi_i \, \d\mu}>\rho_n }}\leq \delta' (j, n),
\end{equation}
and finally by a diagonal argument we get 
a sequence $j_n$ (slowly) increasing to infinity and a sequence 
 $(\delta''(n))$ tending to zero such that $\mu(E_n)\leq \delta''(n)$, where 
\begin{equation}\label{eq:delta3}
E_n :=\set{ x: \ \exists i \in [1, j_n], \  \abs{S_n\varphi_i (x) - \int \varphi_i \, \d\mu}>\rho_n }
\end{equation}
(see Remark~\ref{rem:decay} below for a more explicit version). 


{\bf{Saddle periodic points.--}} Now, in Step 2 of the proof of~\eqref{eq:equidistribution_BLS}, instead of an arbitrary $x\in P\cap f^{-n}(P)$, we take $x\in P\cap f^{-n}(P)\cap E^\complement_n$, 
which provides as in Step 4 a set of periodic points $A_n^+ = A_n^+(P)$ with 
\begin{align} 
\mu(P)d^{-n} \# A_n^+ &\geq \mu(P\cap f^{-n}(P)\cap E^\complement_n) \\
&\geq \mu(P\cap f^{-n}(P)) - \delta''(n)
\end{align}
 and 
as before $\liminf_{n\to\infty} d^{-n}  \#A_n^+\geq \mu(P)$.

If $\alpha = \alpha(x)$ is the saddle point obtained from $x$ 
by the Closing Lemma, we also have an estimate of the form 
\begin{align}\label{eq:shadowing}
  \ \dist(f^k(x), f^k(\alpha))&\leq C \max (\dist(x, \alpha), \dist (f^n(x), \alpha)) e^{-\theta \min(k, n-k)} \\&\notag\leq C \diam(P) e^{-\theta \min(k, n-k)}.
\end{align}
for all $0\leq k\leq n$ (see~\cite[Thm S.4.13]{KH}). This guarantees that for every 
$C^1$ function $\varphi$, 
\begin{equation}
\abs{S_n\varphi(x)  - S_n\varphi(\alpha)}\leq \frac{C}{n}\diam(P) \norm{\varphi}_{C^1} . 
\end{equation}
Therefore, replacing $(j_n)$ by 
$\min(j_n, j_n')$, where $j_n'$ satisfies 
\begin{equation}
\frac{C}{n}\diam(P) \max_{1\leq i\leq j_n'}\norm{\varphi_i}_{C^1}  <\rho_n,
\end{equation} 
 we deduce that, for every $\alpha \in A_n^+$ and every  $i\in [1,  j_n]$,
\begin{equation}\label{eq:average}
  \abs{S_n\varphi_i (\alpha )  - \int\varphi_i\, \d\mu} < 2\rho_n.
\end{equation}

For the function $\varphi_0$, we can argue similarly by using the Hölder continuity property of the unstable distribution. It holds not only on the Pesin box $P$ and its iterates, but it also extends to the approximating periodic orbit, as follows e.g. from  \cite[Thm 5.20]{barreira-pesin:book}. Together with~\eqref{eq:shadowing}, this shows that 
$\abs{S_n \varphi_0 (x)  - S_n\varphi_0(\alpha)}\leq \frac{C}{n} \diam(P)^\theta$ for some $\theta>0$. 
By incorporating this estimate, we can then ensure that~\eqref{eq:average} holds for $i=0$ as well.

At this stage we have constructed a set of good periodic orbits $A_n^+(P)$ associated to a given Pesin box. Since the Pesin boxes are disjoint, when $n$ is sufficiently large, the  $A_n^+(P)$ 
are disjoint as well, and taking the union we obtain a set 
$A_n^+ = A_n^+(\e)$ of saddle periodic points with $\liminf d^{-n} \#A_n^+\geq 1-\e$ satisfying~\eqref{eq:average} for $i\in [0, j_n]$. 

It only remains to make one last diagonal extraction to construct a set
$\mathrm{SPer}^+_n$  of saddle points independent of $\e$. We proceed as follows. 
By assumption for every $\e>0$  we have  a 
set $A_n^+(\e)$ satisfying~\eqref{eq:average} for $i\in [0, j_n]$ and  $d^{-n} \#A_n^+\geq 1-2\e$ for $n\geq N(\e)$. 
We simply put $\e_n = 1/n$ and 
define $\mathrm{SPer}^+_n = A_n^+(\e_n)$ for every $n \in [N(\e_n), N(\e_{n+1})- 1]$, which fulfills our requirements. Indeed  by construction 
$d^{-n}\#\mathrm{SPer}^+_n \to 1$ and 
if $(p_n)$  is a sequence of periodic points such that $p_n\in \mathrm{SPer}^+_n$, 
then for every fixed $i$, 
for large enough $n$ we have that $i\leq j_n$, so $S_n\varphi_i(p_n) \to \int \varphi_i\, d\mu$, which precisely means that the conclusions of the theorem hold.
\end{proof}

\begin{rem}\label{rem:decay}
The exponential decay of correlations for smooth observables 
makes it possible to render the decay 
functions in \eqref{eq:delta1}-\eqref{eq:delta3} more explicit for the family of smooth functions $(\varphi_i)_{i\geq 1}$. Indeed it was shown  in~\cite{dinh:decay} that 
if $\varphi$ is a $C^1$ function with zero mean, then 
\begin{equation}
\abs{\int (\varphi\circ f^n)\cdot  \varphi \, \d\mu }\leq C\norm{\varphi}^2 d^{-n/8}. 
\end{equation}
where $\norm{\varphi} = \norm{\varphi}_{C^1}$. From this, it is easy to deduce by developing the square (\footnote{We thank Sébastien Gouëzel for indicating this argument to us.}) that  
\begin{equation}
\int (S_n\varphi)^2\d\mu\leq \frac{C}{n}  \norm{\varphi}^2  ,
\end{equation}
 hence 
 \begin{equation}
 \mu\lrpar{\set{x, \abs{S_n\varphi} >\e} }\leq \frac{C}{\e^2 n} \norm{\varphi}^2 .
 \end{equation}
 It follows that  $\delta(j, \e, n) = C \max_{1\leq i\leq j} \norm{\tilde \varphi_i}^2 (\e^2 n)\inv$, where 
 $\tilde \varphi_i  = \varphi_i - \int\varphi_i \d\mu$, hence for $\e_n = n^{-1/4}$ one can choose $\delta'(j, n) = C \max_{1\leq i\leq j} \norm{\tilde \varphi_i}^2 n^{-1/4}$, and finally we can choose $(j_n)$ such that $\max_{1\leq i\leq j_n} \norm{\tilde \varphi_i}^2\leq C n^{1/8}$ and we reach 
 $\delta''(n)  = C n^{-1/8}$. \qed
 \end{rem}

 \bibliographystyle{plain}
\bibliography{bib-rigidity}

\begin{thebibliography}{10}

\bibitem{barreira-pesin:book}
Lu\'is Barreira and Yakov Pesin.
\newblock {\em Introduction to smooth ergodic theory}, volume 231 of {\em
  Graduate Studies in Mathematics}.
\newblock American Mathematical Society, Providence, RI, second edition, [2023]
  \copyright 2023.

\bibitem{topological}
Eric Bedford and Romain Dujardin.
\newblock Topological and geometric hyperbolicity criteria for polynomial
  automorphisms of {$\Bbb C^2$}.
\newblock {\em Ergodic Theory Dynam. Systems}, 42(7):2151--2171, 2022.

\bibitem{BGS}
Eric Bedford, Lorenzo Guerini, and John Smillie.
\newblock Hyperbolicity and quasi-hyperbolicity in polynomial diffeomorphisms
  of {$\Bbb C^2$}.
\newblock {\em Pure Appl. Math. Q.}, 18(1):5--32, 2022.

\bibitem{BK1}
Eric Bedford and Kyounghee Kim.
\newblock No smooth {J}ulia sets for polynomial diffeomorphisms of {$\Bbb C^2$}
  with positive entropy.
\newblock {\em J. Geom. Anal.}, 27(4):3085--3098, 2017.

\bibitem{BK2}
Eric Bedford and Kyounghee Kim.
\newblock Julia sets for polynomial diffeomorphisms of {$\Bbb C^2$} are not
  semianalytic.
\newblock {\em Doc. Math.}, 24:163--173, 2019.

\bibitem{bls2}
Eric Bedford, Mikhail Lyubich, and John Smillie.
\newblock Distribution of periodic points of polynomial diffeomorphisms of
  {$\bold C^2$}.
\newblock {\em Invent. Math.}, 114(2):277--288, 1993.

\bibitem{bls}
Eric Bedford, Mikhail Lyubich, and John Smillie.
\newblock Polynomial diffeomorphisms of {${\bf C}^2$}. {IV}. {T}he measure of
  maximal entropy and laminar currents.
\newblock {\em Invent. Math.}, 112(1):77--125, 1993.

\bibitem{bs1}
Eric Bedford and John Smillie.
\newblock Polynomial diffeomorphisms of {${\bf C}^2$}: currents, equilibrium
  measure and hyperbolicity.
\newblock {\em Invent. Math.}, 103(1):69--99, 1991.

\bibitem{bs2}
Eric Bedford and John Smillie.
\newblock Polynomial diffeomorphisms of {${\bf C}^2$}. {II}. {S}table manifolds
  and recurrence.
\newblock {\em J. Amer. Math. Soc.}, 4(4):657--679, 1991.

\bibitem{bs6}
Eric Bedford and John Smillie.
\newblock Polynomial diffeomorphisms of {${\bf C}^2$}. {VI}. {C}onnectivity of
  {$J$}.
\newblock {\em Ann. of Math. (2)}, 148(2):695--735, 1998.

\bibitem{bs7}
Eric Bedford and John Smillie.
\newblock Polynomial diffeomorphisms of {${\bf C}^2$}. {VII}. {H}yperbolicity
  and external rays.
\newblock {\em Ann. Sci. \'{E}cole Norm. Sup. (4)}, 32(4):455--497, 1999.

\bibitem{bs8}
Eric Bedford and John Smillie.
\newblock Polynomial diffeomorphisms of {$\bold C^2$}. {VIII}.
  {Q}uasi-expansion.
\newblock {\em Amer. J. Math.}, 124(2):221--271, 2002.

\bibitem{BSR}
Eric Bedford and John Smillie.
\newblock Real polynomial diffeomorphisms with maximal entropy: {T}angencies.
\newblock {\em Ann. of Math. (2)}, 160(1):1--26, 2004.

\bibitem{berger-dujardin}
Pierre Berger and Romain Dujardin.
\newblock On stability and hyperbolicity for polynomial automorphisms of {$\Bbb
  C^2$}.
\newblock {\em Ann. Sci. \'{E}c. Norm. Sup\'{e}r. (4)}, 50(2):449--477, 2017.

\bibitem{brunella}
Marco Brunella.
\newblock Minimal models of foliated algebraic surfaces.
\newblock {\em Bull. Soc. Math. France}, 127(2):289--305, 1999.

\bibitem{BHI}
Gregery~T. Buzzard, Suzanne~Lynch Hruska, and Yulij Ilyashenko.
\newblock Kupka-{S}male theorem for polynomial automorphisms of {$\Bbb C^2$}
  and persistence of heteroclinic intersections.
\newblock {\em Invent. Math.}, 161(1):45--89, 2005.

\bibitem{cantat:Acta}
Serge Cantat.
\newblock Dynamique des automorphismes des surfaces {$K3$}.
\newblock {\em Acta Math.}, 187(1):1--57, 2001.

\bibitem{cantat:BHPS}
Serge Cantat.
\newblock Bers and {H}\'{e}non, {P}ainlev\'{e} and {S}chr\"{o}dinger.
\newblock {\em Duke Math. J.}, 149(3):411--460, 2009.

\bibitem{conjugate}
Serge Cantat and Romain Dujardin.
\newblock Holomorphically conjugate polynomial automorphisms of $\mathbb{C}^2$
  are polynomially conjugate.
\newblock {\em Bull. Lond. Math. Soc.}, to appear, 2024.
\newblock https://doi.org/10.1112/blms.13164.

\bibitem{cantat-loray}
Serge Cantat and Frank Loray.
\newblock Dynamics on character varieties and {M}algrange irreducibility of
  {P}ainlev\'e{} {VI} equation.
\newblock {\em Ann. Inst. Fourier (Grenoble)}, 59(7):2927--2978, 2009.

\bibitem{cartan}
Elie Cartan.
\newblock Sur la g\'{e}om\'{e}trie pseudo-conforme des hypersurfaces de
  l'espace de deux variables complexes.
\newblock {\em Ann. Mat. Pura Appl.}, 11(1):17--90, 1933.

\bibitem{delellis:lectures_on_rectifiability}
Camillo De~Lellis.
\newblock {\em Rectifiable sets, densities and tangent measures}.
\newblock Zurich Lectures in Advanced Mathematics. European Mathematical
  Society (EMS), Z\"urich, 2008.

\bibitem{dinh:decay}
Tien-Cuong Dinh.
\newblock Decay of correlations for {H}\'enon maps.
\newblock {\em Acta Math.}, 195:253--264, 2005.

\bibitem{birat}
Romain Dujardin.
\newblock Laminar currents and birational dynamics.
\newblock {\em Duke Math. J.}, 131(2):219--247, 2006.

\bibitem{connex}
Romain Dujardin.
\newblock Some remarks on the connectivity of {J}ulia sets for 2-dimensional
  diffeomorphisms.
\newblock In {\em Complex dynamics}, volume 396 of {\em Contemp. Math.}, pages
  63--84. Amer. Math. Soc., Providence, RI, 2006.

\bibitem{continuity}
Romain Dujardin.
\newblock Continuity of {L}yapunov exponents for polynomial automorphisms of
  {$\Bbb C^2$}.
\newblock {\em Ergodic Theory Dynam. Systems}, 27(4):1111--1133, 2007.

\bibitem{closing}
Romain Dujardin.
\newblock A closing lemma for polynomial automorphisms of {$\Bbb C^2$}.
\newblock {\em Ast\'{e}risque}, (415):35--43, 2020.
\newblock Some aspects of the theory of dynamical systems: a tribute to
  Jean-Christophe Yoccoz. Vol. I.

\bibitem{degenerate}
Romain Dujardin.
\newblock Degenerate homoclinic bifurcations in complex dimension two.
\newblock arXiv:2306.08160, 2023.

\bibitem{DF}
Romain Dujardin and Charles Favre.
\newblock The dynamical {M}anin-{M}umford problem for plane polynomial
  automorphisms.
\newblock {\em J. Eur. Math. Soc. (JEMS)}, 19(11):3421--3465, 2017.

\bibitem{tangencies}
Romain Dujardin and Mikhail Lyubich.
\newblock Stability and bifurcations for dissipative polynomial automorphisms
  of {$\Bbb{C}^2$}.
\newblock {\em Invent. Math.}, 200(2):439--511, 2015.

\bibitem{eremenko-vanstrien}
Alexandre Eremenko and Sebastian van Strien.
\newblock Rational maps with real multipliers.
\newblock {\em Trans. Amer. Math. Soc.}, 363(12):6453--6463, 2011.

\bibitem{falconer}
K.~J. Falconer.
\newblock {\em The geometry of fractal sets}, volume~85 of {\em Cambridge
  Tracts in Mathematics}.
\newblock Cambridge University Press, Cambridge, 1986.

\bibitem{fatou:3e_memoire}
P.~Fatou.
\newblock Sur les \'equations fonctionnelles {Troisi\`eme} m\'emoire.
\newblock {\em Bulletin de la Soci\'et\'e Math\'ematique de France},
  48:208--314, 1920.

\bibitem{FLRT}
Tanya Firsova, Mikhail Lyubich, Remus Radu, and Raluca Tanase.
\newblock Hedgehogs for neutral dissipative germs of holomorphic
  diffeomorphisms of {$(\Bbb C^2,0)$}.
\newblock {\em Ast\'{e}risque}, (416):193--211, 2020.
\newblock Some aspects of the theory of dynamical systems: a tribute to
  Jean-Christophe Yoccoz. Vol. II.

\bibitem{fisher-hasselblatt:hyperbolic_flows}
Todd Fisher and Boris Hasselblatt.
\newblock {\em Hyperbolic flows}.
\newblock Zurich Lectures in Advanced Mathematics. EMS Publishing House,
  Berlin, 2019.

\bibitem{fornaess-sibony:duke}
John~Erik Forn{\ae}ss and Nessim Sibony.
\newblock Complex {H}\'{e}non mappings in {${\bf C}^2$} and
  {F}atou-{B}ieberbach domains.
\newblock {\em Duke Math. J.}, 65(2):345--380, 1992.

\bibitem{fornaess-sibony}
John~Erik Forn{\ae}ss and Nessim Sibony.
\newblock Complex dynamics in higher dimensions.
\newblock In {\em Complex potential theory ({M}ontreal, {PQ}, 1993)}, volume
  439 of {\em NATO Adv. Sci. Inst. Ser. C: Math. Phys. Sci.}, pages 131--186.
  Kluwer Acad. Publ., Dordrecht, 1994.
\newblock Notes partially written by Estela A. Gavosto.

\bibitem{friedland-milnor}
Shmuel Friedland and John Milnor.
\newblock Dynamical properties of plane polynomial automorphisms.
\newblock {\em Ergodic Theory Dynam. Systems}, 9(1):67--99, 1989.

\bibitem{garnett-marshall}
John~B. Garnett and Donald~E. Marshall.
\newblock {\em Harmonic measure}, volume~2 of {\em New Mathematical
  Monographs}.
\newblock Cambridge University Press, Cambridge, 2005.

\bibitem{Ghys:holomorphic_anosov}
\'{E}tienne Ghys.
\newblock Holomorphic {A}nosov systems.
\newblock {\em Invent. Math.}, 119(3):585--614, 1995.

\bibitem{guerini-peters}
Lorenzo Guerini and Han Peters.
\newblock Julia sets of complex {H}\'{e}non maps.
\newblock {\em Internat. J. Math.}, 29(7):1850047, 22, 2018.

\bibitem{hamilton}
David~H. Hamilton.
\newblock Length of {J}ulia curves.
\newblock {\em Pacific J. Math.}, 169(1):75--93, 1995.

\bibitem{HOV1}
John~H. Hubbard and Ralph~W. Oberste-Vorth.
\newblock H\'{e}non mappings in the complex domain. {I}. {T}he global topology
  of dynamical space.
\newblock {\em Inst. Hautes \'{E}tudes Sci. Publ. Math.}, (79):5--46, 1994.

\bibitem{HOV2}
John~H. Hubbard and Ralph~W. Oberste-Vorth.
\newblock H\'{e}non mappings in the complex domain. {II}. {P}rojective and
  inductive limits of polynomials.
\newblock In {\em Real and complex dynamical systems ({H}iller\o d, 1993)},
  volume 464 of {\em NATO Adv. Sci. Inst. Ser. C: Math. Phys. Sci.}, pages
  89--132. Kluwer Acad. Publ., Dordrecht, 1995.

\bibitem{huguin}
Valentin Huguin.
\newblock Rational maps with rational multipliers.
\newblock {\em J. \'{E}c. polytech. Math.}, 10:591--599, 2023.

\bibitem{ji-xie:multipliers}
Zhuchao Ji and Junyi Xie.
\newblock Homoclinic orbits, multiplier spectrum and rigidity theorems in
  complex dynamics.
\newblock {\em Forum Math. Pi}, 11:Paper No. e11, 37, 2023.

\bibitem{ji-xie:injective}
Zhuchao Ji and Junyi Xie.
\newblock The multiplier spectrum morphism is generically injective.
\newblock arXiv math:2309.15382, 2023.

\bibitem{KH}
Anatole Katok and Boris Hasselblatt.
\newblock {\em Introduction to the modern theory of dynamical systems},
  volume~54 of {\em Encyclopedia of Mathematics and its Applications}.
\newblock Cambridge University Press, Cambridge, 1995.
\newblock With a supplementary chapter by Katok and Leonardo Mendoza.

\bibitem{Lee:equidistribution}
Chong~Gyu Lee.
\newblock The equidistribution of small points for strongly regular pairs of
  polynomial maps.
\newblock {\em Math. Z.}, 275(3-4):1047--1072, 2013.

\bibitem{lyubich-peters}
Mikhail Lyubich and Han Peters.
\newblock Structure of partially hyperbolic {H}\'{e}non maps.
\newblock {\em J. Eur. Math. Soc. (JEMS)}, 23(9):3075--3128, 2021.

\bibitem{LRT}
Mikhail Lyubich, Remus Radu, and Raluca Tanase.
\newblock Hedgehogs in higher dimensions and their applications.
\newblock {\em Ast\'{e}risque}, (416):213--251, 2020.
\newblock Some aspects of the theory of dynamical systems: a tribute to
  Jean-Christophe Yoccoz. Vol. II.

\bibitem{makarov}
N.~G. Makarov.
\newblock On the distortion of boundary sets under conformal mappings.
\newblock {\em Proc. London Math. Soc. (3)}, 51(2):369--384, 1985.

\bibitem{mccluskey-manning}
Heather McCluskey and Anthony Manning.
\newblock Hausdorff dimension for horseshoes.
\newblock {\em Ergodic Theory Dynam. Systems}, 3(2):251--260, 1983.

\bibitem{mcm_algorithms}
Curt McMullen.
\newblock Families of rational maps and iterative root-finding algorithms.
\newblock {\em Ann. of Math. (2)}, 125(3):467--493, 1987.

\bibitem{milnor:lattes}
John Milnor.
\newblock On {L}att\`es maps.
\newblock In {\em Dynamics on the {R}iemann sphere}, pages 9--43. Eur. Math.
  Soc., Z\"{u}rich, 2006.

\bibitem{Narasimhan}
Raghavan Narasimhan.
\newblock {\em Introduction to the theory of analytic spaces}, volume No. 25 of
  {\em Lecture Notes in Mathematics}.
\newblock Springer-Verlag, Berlin-New York, 1966.

\bibitem{pesin_book}
Yakov~B. Pesin.
\newblock {\em Dimension theory in dynamical systems}.
\newblock Chicago Lectures in Mathematics. University of Chicago Press,
  Chicago, IL, 1997.
\newblock Contemporary views and applications.

\bibitem{przytycki-zdunik}
Feliks Przytycki and Anna Zdunik.
\newblock On {H}ausdorff dimension of polynomial not totally disconnected
  {J}ulia sets.
\newblock {\em Bull. Lond. Math. Soc.}, 53(6):1674--1691, 2021.

\bibitem{robinson}
Clark Robinson.
\newblock {\em Dynamical systems}.
\newblock Studies in Advanced Mathematics. CRC Press, Boca Raton, FL, 1995.
\newblock Stability, symbolic dynamics, and chaos.

\bibitem{wang-sun}
Zhenqi Wang and Wenxiang Sun.
\newblock Lyapunov exponents of hyperbolic measures and hyperbolic periodic
  orbits.
\newblock {\em Trans. Amer. Math. Soc.}, 362(8):4267--4282, 2010.

\bibitem{wolf:TAMS}
Christian Wolf.
\newblock On measures of maximal and full dimension for polynomial
  automorphisms of {$\Bbb C^2$}.
\newblock {\em Trans. Amer. Math. Soc.}, 355(8):3227--3239, 2003.

\bibitem{henon-problem-list}
Julia Xenelkis~de H\'enon.
\newblock H{\'e}non maps: a list of open problems.
\newblock {\em Arnold Mathematical Journal}, to appear, 2024.

\bibitem{Xu-Zhang}
Disheng Xu and Jiesong Zhang.
\newblock On holomorphic partially hyperbolic systems.
\newblock arXiv math:2401.04310, 2024.

\bibitem{Yuan:Inventiones}
Xinyi Yuan.
\newblock Big line bundles over arithmetic varieties.
\newblock {\em Invent. Math.}, 173(3):603--649, 2008.

\bibitem{zdunik:dimension}
Anna Zdunik.
\newblock Parabolic orbifolds and the dimension of the maximal measure for
  rational maps.
\newblock {\em Invent. Math.}, 99(3):627--649, 1990.

\bibitem{zdunik:exponents}
Anna Zdunik.
\newblock Characteristic exponents of rational functions.
\newblock {\em Bull. Pol. Acad. Sci. Math.}, 62(3):257--263, 2014.

\end{thebibliography}

\end{document}